\documentclass[11pt]{article}

\usepackage[letterpaper,margin=1in]{geometry}
\usepackage{amsthm,mathtools,amsfonts,amssymb,bm,graphicx,booktabs}
\usepackage[ruled,vlined]{algorithm2e}
\usepackage[numbers,sort&compress]{natbib}
\usepackage{xcolor}
\usepackage{microtype}
\usepackage[hidelinks,hypertexnames=false]{hyperref}
\hypersetup{
  pdftitle={Weak Convergence Rates for Partial-Sum Processes of Nonlinear Stochastic Approximation},
  pdfauthor={Xiang Li, Jiadong Liang, and Zhihua Zhang},
  pdfkeywords={stochastic approximation, functional central limit theorem,
    Gaussian approximation, Prokhorov distance}
}

\allowdisplaybreaks
\theoremstyle{plain}
\newtheorem{theorem}{Theorem}[section]
\newtheorem{lemma}[theorem]{Lemma}
\theoremstyle{definition}
\newtheorem{definition}[theorem]{Definition}
\newtheorem*{definition*}{Definition}
\newtheorem{assumption}{Assumption}
\newtheorem*{remark}{Remark}

\newcommand{\R}{\mathbb{R}}
\newcommand{\E}{\mathbb{E}}
\newcommand{\Pp}{\mathbb{P}}
\newcommand{\rateNL}{a_{\mathrm{nl}}}
\newcommand{\rateCov}{a_{\mathrm{cov}}}
\newcommand{\rateOsc}{a_{\mathrm{osc}}}
\newcommand{\rateSm}{a_{\mathrm{sm}}}
\newcommand{\cF}{\mathcal{F}}
\newcommand{\bx}{\bm{x}}
\newcommand{\bmM}{\bm{M}}
\newcommand{\bPhi}{\bm{\Phi}}
\newcommand{\bG}{\bm{G}}
\newcommand{\bS}{\bm{S}}
\newcommand{\bI}{\bm{I}}
\newcommand{\bA}{\bm{A}}
\newcommand{\bB}{\bm{B}}
\newcommand{\bV}{\bm{V}}
\newcommand{\bJ}{\bm{J}}
\newcommand{\bN}{\bm{N}}
\newcommand{\bDelta}{\bm{\Delta}}
\newcommand{\bepsilon}{\bm{\epsilon}}
\newcommand{\btheta}{\bm{\theta}}
\newcommand{\br}{\bm{r}}
\newcommand{\by}{\bm{y}}
\newcommand{\bz}{\bm{z}}
\newcommand{\bpsi}{\bm{\psi}}
\newcommand{\bZ}{\bm{Z}}
\newcommand{\bW}{\bm{W}}
\newcommand{\norm}[1]{\left\lVert #1\right\rVert}
\newcommand{\tnorm}[1]{\left\lvert\!\left\lvert\!\left\lvert #1
  \right\rvert\!\right\rvert\!\right\rvert}
\newcommand{\suppappendixref}[1]{Appendix~\ref{#1}}

\makeatletter
\newcommand{\appendixtableofcontents}{%
  \section*{Contents of the Appendix}%
  \begingroup
    \setcounter{tocdepth}{2}%
    \@starttoc{apc}%
  \endgroup
}
\newcommand{\redirectappendixcontents}{%
  \let\appendix@addtocontents\addtocontents
  \renewcommand{\addtocontents}[2]{%
    \def\appendix@file{##1}%
    \def\appendix@toc{toc}%
    \ifx\appendix@file\appendix@toc
      \appendix@addtocontents{apc}{##2}%
    \else
      \appendix@addtocontents{##1}{##2}%
    \fi
  }%
}
\makeatother

\newcommand{\isarxivpreprint}{}

\title{Weak Convergence Rates for Partial-Sum Processes of Nonlinear
Stochastic Approximation}
\author{%
  Xiang Li \qquad Jiadong Liang \qquad Zhihua Zhang\\[0.5em]
  \normalsize School of Mathematical Sciences, Peking University\\
  \small \href{mailto:lx10077@pku.edu.cn}{\texttt{lx10077@pku.edu.cn}}\quad
  \href{mailto:jdliang@pku.edu.cn}{\texttt{jdliang@pku.edu.cn}}\quad
  \href{mailto:zhzhang@math.pku.edu.cn}{\texttt{zhzhang@math.pku.edu.cn}}%
}
\date{}

\begin{document}

\maketitle

\begin{abstract}
Stochastic approximation provides a general framework for
online estimation and optimization. Statistical inference
based on the resulting estimates requires understanding
their fluctuations around the target. For Polyak--Ruppert
averaging, functional central limit theorems describe the
normalized cumulative estimation errors through a Brownian
limit. However, weak convergence alone does not determine how
quickly the distribution of this process approaches that limit.
We establish a quantitative functional central limit theorem for
nonlinear stochastic approximation with polynomially decreasing step sizes
and martingale-difference noise. Under suitable regularity and
moment conditions, we derive an explicit finite-sample bound
on the Prokhorov distance between each fixed scalar projection
of the normalized partial-sum process and its Brownian
limit. The bound reveals a step-size tradeoff: faster
step-size decay reduces the contribution of the nonlinear
remainder but increases that of the smoothing induced by
the recursion.
To assess the sharpness of the upper bound, we develop
four lower-bound constructions, each designed to isolate
one source of error. They show that covariance stabilization,
finite-moment noise, algorithmic smoothing, and nonlinearity
can each separately limit the Brownian approximation rate of
the SA path.

\end{abstract}

\medskip
\noindent\textbf{MSC2020 subject classifications.}
Primary 60F17; secondary 62L20, 62E17, 60B10.


\section{Introduction}
\label{sec:rate-introduction}

Stochastic approximation (SA) provides a simple way to learn from a stream of noisy observations. Starting with the Robbins--Monro procedure for stochastic root finding \cite{robbins1951stochastic}, SA has developed into a general framework for stochastic optimization and learning
\cite{benveniste2012adaptive,kushner2003stochastic,borkar2009stochastic}.
It underlies derivative-free optimization and stochastic-gradient methods \cite{kiefer1952stochastic,bottou2018optimization}, recursive identification and adaptive filtering \cite{benveniste2012adaptive}, and temporal-difference
and Q-learning algorithms \cite{watkins1989learning,tsitsiklis1994asynchronous}. In all these examples, each new observation is processed through one recursive update. The full data set need not be stored, and no growing optimization problem needs to be solved. This feature is especially useful when data arrive along a time series, a single simulation run, or a controlled Markov trajectory \cite{liang2010trajectory,even2023markov}.

Under standard stability conditions and a suitably decaying step size, the SA iterates converge to the target \cite{kushner2003stochastic}. Convergence alone, however, is not enough for statistical inference: one must also understand how the iterates fluctuate around the target. Classical work characterized these fluctuations through asymptotic normality \cite{sacks1958asymptotic,fabian1968asymptotic}. Ruppert and Polyak then showed that a simple running average of the iterates can make SA asymptotically efficient at essentially no additional computational cost
\cite{ruppert1988efficient,polyak1990new,polyak1992acceleration}. The resulting central limit theorem was later extended to broad martingale- and Markov-noise settings
\cite{kushner1993stochastic,liang2010trajectory}, and the local asymptotic optimality of averaged stochastic-gradient methods was studied in \cite{duchi2021asymptotic}. This theory has also led to online methods for estimating the limiting covariance and constructing confidence sets \cite{godichon2019online,chen2020statistical,zhu2021online,roy2023online,fang2018online}. These results, however, concern the averaged estimator at a fixed terminal horizon, rather than its evolution over the entire path.

To make this distinction precise, let $\bx_t$ denote the $t$-th SA iterate and $\bx^\star$ the target solution. The Polyak--Ruppert estimator at horizon $T$ is $\overline{\bx}_T=T^{-1}\sum_{t=1}^T\bx_{t+1}$, as in Algorithm~\ref{alg:nonlinear-pr-sa}. The same iterates also generate the partial-sum process $\bPhi_T:[0,1]\to\R^d$, a random continuous mapping obtained by interpolating the normalized cumulative errors. Its values at the grid points are
\[
  \bPhi_T(n/T)=\frac{1}{\sqrt T}\sum_{t=1}^n(\bx_{t+1}-\bx^\star),
  \qquad n=0,\ldots,T,
\]
with an empty sum equal to zero. Between these points, $\bPhi_T$ varies linearly, so it is defined at every time $r\in[0,1]$.
At $r=1$, this process equals $\sqrt{T}(\overline{\bx}_T-\bm x^\star)$. Therefore, the classical central limit theorem describes only the terminal value of a full stochastic path. The rest of the path is also useful for statistical inference. Batch-means estimators use its increments over different time blocks \citep{chen2020statistical,zhu2021batchmeans}, while self-normalized and random-scaling methods use the path to remove the unknown limiting covariance \citep{lee2021fast}. Procedures that conduct inference at several intermediate times likewise require joint control over the process.

A growing literature has established functional central limit theorems for this partial-sum path under decreasing step sizes
\citep{lee2021fast,li2021statistical,li2023nonlinear,li2025stream}. These results show that the full path converges weakly to a Gaussian process and justify path-based inference asymptotically. For constant-step SA, Xie and Zhang~\cite{xie2022statistical} establish a related FCLT for partial sums centered at the stationary mean, which may differ from $\bx^\star$, and use it for random-scaling inference. These functional limit results do not, however, quantify the approximation error at a finite horizon $T$. By contrast, quantitative Gaussian-approximation bounds are available for the Polyak--Ruppert estimator
\citep{anastasiou2019normal,shao2022berry,samsonov2024gaussian, sheshukova2026gaussian}, but these bounds control only the terminal value of the partial-sum process, not its block increments, maxima, or other path-dependent functionals. The open question is how fast its law approaches that limit.
Answering this question is not a direct extension of the endpoint theory. A functional approximation must control both errors that accumulate over the full horizon and local fluctuations that may cancel at the terminal time but remain visible along the path. These errors are shaped by the nonlinear SA dynamics, the noise process, and the smoothing induced by the recursive
updates. These observations lead to the central question of this paper:
\begin{center}
\emph{What is the rate of weak convergence in the functional central limit theorem for the partial-sum process associated with Polyak--Ruppert averaging?}
\end{center}

\subsection{Our contributions}
\label{sec:rate-contributions}

We study nonlinear SA with a polynomially decaying step size and martingale-difference noise whose conditional $p$-th moments are uniformly bounded for some $p>2$.
Our main result quantifies how accurately Brownian motion approximates the cumulative estimation error over an entire run of nonlinear SA.%
\ifdefined\isarxivpreprint
\footnote{A preliminary version of our upper bound appeared in our earlier work \cite[Theorem~5]{li2023nonlinearv2}. The present paper substantially develops and improves that result; see the comparison at the end of Section~\ref{sec:rate-upper-bounds}.
}
\fi
For a fixed direction $\btheta\in\R^d$, the scalar path $\btheta^\top\bPhi_T$ records this error along that direction. We compare its law with that of $\btheta^\top\bZ_\star$, where $\bZ_\star: [0,1] \to \R^d$ is a mean-zero Brownian motion with structured covariance. Under polynomially decaying step sizes and common quantitative assumptions (stated in Section~\ref{sec:rate-quantitative-assumptions}), we obtain the following bound in the Prokhorov distance, denoted by $d_{\mathrm P}(\cdot,\cdot)$, up to logarithmic factors:
\[
d_{\mathrm P}\!\left(\btheta^\top\bPhi_T,\btheta^\top\bZ_\star\right)
\lesssim
\underbrace{T^{-\rateNL}}_{\text{nonlinearity}}
+\underbrace{T^{-\rateCov}}_{\text{covariance error}}
+\underbrace{T^{-\rateOsc}}_{\text{noise oscillations}}
+\underbrace{T^{-\rateSm}}_{\text{algorithmic smoothing}}.
\]

The four terms in the bound reflect different features of the problem. First, the nonlinear recursion differs from its linear approximation near the target. Better control of the iterate errors makes this approximation more accurate. Second, the accumulated conditional covariance of the noise may differ from that of the Brownian target. Third, rare large noise increments may create local fluctuations that remain visible along the path. Finally, each noise input affects several subsequent iterates, so its contribution to the partial-sum process builds up gradually. We refer to this effect of the recursion as algorithmic smoothing, since spreading each noise input across several updates can reduce abrupt local fluctuations in the partial-sum path.

The four exponents quantify these effects. For step sizes $\eta_t=\eta t^{-\alpha}$, they are
\[
\rateNL\coloneqq\frac{(\alpha-1/2)q}{q+1},\qquad
\rateCov\coloneqq\gamma,\qquad
\rateOsc\coloneqq \frac{p-2}{2(p+1)},\qquad
\rateSm\coloneqq \frac{1-\alpha}{2}.
\]
Here $p>2$ is the order of the uniformly bounded conditional noise moment, while $2q$ is the order of the controlled iterate (that is $\bx_t -\bx^\star$) moment. Larger values of $p$ and $q$ give stronger control of rare large noise increments and iterate errors, respectively. The parameter $\gamma$ describes how quickly the cumulative conditional covariance approaches its limiting linear profile. These quantities are specified in Assumptions~\ref{ass:rate-gain}--\ref{ass:rate-trajectory-stability}. A larger exponent means faster decay of the corresponding term in the bound. In particular, increasing $\alpha$ improves the nonlinear exponent but worsens the smoothing exponent, revealing the step-size tradeoff.

To assess whether the four rates can be improved, we study each source of error separately, using a different class of models in each case. For the two noise terms \(T^{-\rateCov}\) and \(T^{-\rateOsc}\), we construct models where a mismatch in accumulated variability or rare large noise inputs can limit the accuracy of Brownian approximation. For algorithmic smoothing, even a scalar SA model with Gaussian noise can produce a partial-sum path that is locally smoother than Brownian motion, leading to the lower bound \(T^{-\rateSm}\sqrt{\log T}\). For nonlinearity, we construct a model in which the error from linearization persists over many updates, yielding a lower bound of order \(T^{-\rateNL}\).
We then combine these results into an overall lower bound for the SA path over a union class containing all four constructed models. This allows us to study the four sources separately, without requiring them to occur together in a single model. Section~\ref{sec:rate-lower-bounds} gives the constructions, their conditions, and the argument combining their lower bounds.

In summary, our results quantify the accuracy of Brownian approximation over a finite run of nonlinear SA and clarify how it depends on the noise, the nonlinear dynamics, and the choice of step size.

\subsection{Notation}
\label{sec:rate-notation}

For scalars \(a,b\), we use \(a\vee b\) and \(a\wedge b\) for
\(\max\{a,b\}\) and \(\min\{a,b\}\), respectively, and, for
\(x\in\R\), write \((x)_+\coloneqq\max\{x,0\}\).  For nonnegative
quantities indexed by an asymptotic parameter,
\(a\lesssim b\) means \(a\le Cb\) for a constant independent of that
parameter. Here \(\gtrsim\) denotes the reverse relation, and \(a\asymp b\) means
that both relations hold.  Throughout the proofs, unsubscripted
\(C,c\in(0,\infty)\) denote generic constants whose values may change from
line to line.  Unless stated otherwise, they may depend on fixed model
parameters but not on \(T\) or on an auxiliary tail threshold.  Named or
subscripted constants remain fixed within the statement in which they occur.
We use \(\Pp\) and \(\E\) for probability and expectation, respectively.
For a probability law \(P\), \(\E_P\) denotes expectation under \(P\).
When the random element is displayed explicitly, we equivalently write
\(\E_{X\sim P}\).  Let \(\mathcal L(X)\) denote the law of a random
element \(X\).  For random elements \(X_T,X\) in a common metric space, we
write \(X_T\Rightarrow X\) to mean that \(\mathcal L(X_T)\) converges weakly
to \(\mathcal L(X)\).

Vectors are written in boldface.  For a vector, \(\|\cdot\|\) denotes the
Euclidean norm.  For a random vector \(\bm X\) and \(r\ge1\), write
\(\|\bm X\|_{L^r}\coloneqq
\bigl(\E[\|\bm X\|^r]\bigr)^{1/r}\).

Matrices are also written in boldface.  For a matrix, \(\|\cdot\|\) denotes
the induced operator norm.  We write \(\bI\) for the identity matrix of the
dimension determined by context and \(\bA^\top\) for the transpose of \(\bA\)
and, when \(\bA\) is invertible,
\(\bA^{-\top}\coloneqq(\bA^{-1})^\top\).

For a function \(f:[0,1]\to\R^m\), let
\(\tnorm{f}\coloneqq\sup_{0\le r\le1}\|f(r)\|\).  We write
\(C([0,1],\R)\) for the set of continuous functions
\(f:[0,1]\to\R\), and \(D([0,1],\R)\) for the set of functions
\(f:[0,1]\to\R\) that are right-continuous at every \(t\in[0,1)\)
and have a finite left limit
\(f(t-)\coloneqq\lim_{s\uparrow t}f(s)\) at every \(t\in(0,1]\).
Thus \(C([0,1],\R)\subset D([0,1],\R)\).

\subsection{Organization}
\label{sec:organization}


The remainder of the paper is organized as follows.
Section~\ref{sec:rate-preliminaries} introduces the problem setup, notation, and assumptions.
Section~\ref{sec:rate-model-results} presents our upper and lower bounds. Specifically, Subsection~\ref{sec:rate-upper-bounds} states the upper bounds
and explains the origin of each term, while
Subsection~\ref{sec:rate-lower-bounds} develops the corresponding lower-bound mechanisms term by term.
Section~\ref{sec:rate-proof-overview} proves the upper bounds, and Section~\ref{sec:rate-related} discusses related work. All remaining proofs are deferred to the appendix.

\section{Preliminaries}
\label{sec:rate-preliminaries}
We first define the stochastic-approximation iterates and the partial-sum path underlying Polyak--Ruppert averaging. Section~\ref{sec:rate-path-model} specifies the distance used to compare path laws, and Section~\ref{sec:rate-quantitative-assumptions} states the assumptions for the quantitative analysis.

\subsection{Polyak--Ruppert averaging and its functional limit}
\label{sec:rate-pr-trajectory}

Let \(\xi\) be a generic observation taking values in a measurable space
\((\mathsf X,\mathcal F)\), and let \(\mathbb P_\xi\) denote its probability law.
Fix an integer \(d\ge1\), and let
\(H\colon \R^d\times\mathsf X\to\R^d\) be measurable with
\(\int_{\mathsf X}\|H(\bx,\xi)\|\,\mathbb P_\xi(\mathrm d\xi)<\infty\) for every
\(\bx\in\R^d\).  SA is a recursive method for finding a
root \(\bx^\star\in\R^d\) of the population equation
\begin{equation}
  g(\bx^\star)=0,
  \qquad
  g(\bx)\coloneqq\int_{\mathsf X}H(\bx,\xi)\,\mathbb P_\xi(\mathrm d\xi).
  \label{eq:population-root-problem}
\end{equation}
For i.i.d.\ observations, \(\mathbb P_\xi\) is their common law. For Markov observations,
it is typically the invariant law of the data process.  This root-finding
formulation includes stochastic optimization, estimating equations, and
fixed-point algorithms
\cite{robbins1951stochastic,benveniste2012adaptive,kushner2003stochastic}.

Observations $(\xi_t)_{t\ge1}$ arrive sequentially, with natural filtration $\cF_t=\sigma(\xi_1,\ldots,\xi_t)$ and trivial $\cF_0$. We write $\Pp$ and $\E$ for probability and expectation on the common probability space supporting the entire sequence. The marginal observation law $\mathbb P_\xi$ in \eqref{eq:population-root-problem} is a separate object.

Starting from a deterministic point $\bx_1$, stochastic approximation updates the current estimate using the next observation and a positive step size $\eta_t$. Polyak--Ruppert averaging then takes the running average of the updated estimates $\bx_2,\ldots,\bx_{t+1}$. Algorithm~\ref{alg:nonlinear-pr-sa} gives both updates. In particular, $\bx_t$ is $\cF_{t-1}$-measurable, and the average requires only a running vector sum.

\begin{algorithm}[H]
\caption{Nonlinear stochastic approximation with Polyak--Ruppert averaging}
\label{alg:nonlinear-pr-sa}
\KwIn{Horizon \(T\), initial point \(\bx_1\), step sizes
  \((\eta_t)_{t\ge1}\), and update map \(H\)}
\KwOut{Iterates \((\bx_t)_{t=1}^{T+1}\) and sequential averages
  \((\overline{\bx}_t)_{t=1}^{T}\)}
Set \(\overline{\bx}_0\leftarrow0\)\;
\For{\(t=1,\ldots,T\)}{
  Observe \(\xi_t\)\;
  \(\bx_{t+1}\leftarrow \bx_t-\eta_tH(\bx_t,\xi_t)\)\;
  \(\overline{\bx}_t\leftarrow\overline{\bx}_{t-1}
    +t^{-1}(\bx_{t+1}-\overline{\bx}_{t-1})\)\;
}
\end{algorithm}
The average after $T$ updates is $\overline{\bx}_T=T^{-1}\sum_{t=1}^T\bx_{t+1}$. To study its fluctuations throughout the run, we now give the full interpolation formula for the partial-sum process introduced in Section~\ref{sec:rate-introduction}. Define the random continuous mapping $\bPhi_T:[0,1]\to\R^d$ by
\begin{equation}
\begin{aligned}
  \bPhi_T(r)
  &\coloneqq\frac1{\sqrt T}\left\{
    \sum_{t=1}^n(\bx_{t+1}-\bx^\star)
    +(Tr-n)(\bx_{n+2}-\bx^\star)\right\},\\
  &\hspace{1em}r\in[n/T,(n+1)/T],\qquad n=0,\ldots,T-1.
\end{aligned}
  \label{eq:partial-sum-process}
\end{equation}
The sum is empty when $n=0$, and the formulas on adjacent intervals agree at their shared endpoint. Thus \eqref{eq:partial-sum-process} specifies the entire path, including $\bPhi_T(0)=0$ and $\bPhi_T(1)=\sqrt T(\overline{\bx}_T-\bx^\star)$. At every grid point $n/T$ with $n\ge1$, it gives $\bPhi_T(n/T)=n(\overline{\bx}_n-\bx^\star)/\sqrt T$, recovering the cumulative sums in the introduction. We use $\bDelta_t=\bx_t-\bx^\star$ for the iterate error.

\begin{remark}[Indexing convention]
The partial-sum path in \eqref{eq:partial-sum-process} uses the updated
iterates $\bx_2,\ldots,\bx_{n+1}$, whereas a seemingly more natural convention
would use $\bx_1,\ldots,\bx_n$.  At the grid point $n/T$, the latter differs
from \(\bPhi_T\) by $(\bx_1-\bx_{n+1})/\sqrt T$.  We adopt the former only
because it makes the main proof ideas more transparent.  This discrepancy is
negligible in our final analysis and is never rate-determining; we record it
solely to clarify the indexing convention.
\end{remark}

Classical functional limit results identify the Brownian limit of the path in \eqref{eq:partial-sum-process}. We recall their conclusion below, with the regularity conditions taken from the cited results.

\begin{theorem}[Qualitative FCLT for the Polyak--Ruppert trajectory]
\label{thm:fclt_informal}
Under the regularity conditions of \cite[Assumption~1]{lee2021fast} for SGD, or \cite[Assumptions~1--4 and Theorem~1]{li2023nonlinear} for nonlinear SA with Markovian data, the corresponding partial-sum path satisfies
\begin{equation}
  \bPhi_T\Rightarrow\bZ_\star
  \quad\text{in }C([0,1],\R^d),
  \label{eq:rate-qualitative-fclt}
\end{equation}
where $\bZ_\star(r)=\bG^{-1}\bS^{1/2}\bW_\star(r)$, $\bG$ is the Jacobian of the mean field at the target, and $\bS$ is the limiting covariance of the leading martingale noise. Here, $\bW_\star$ is standard $d$-dimensional Brownian motion, and the path space carries the uniform norm.
\end{theorem}

Lee et al.~\cite{lee2021fast} use this path limit to construct an asymptotically pivotal statistic for online inference with SGD, an approach known as random scaling. Li, Liang, and Zhang~\cite{li2023nonlinear} extend the approach to nonlinear SA with Markovian observations, using a Poisson-equation decomposition to account for dependence in the data. Related work considers Local SGD \cite{li2021statistical}, stream SGD with Markov data \cite{li2025stream}, and asynchronous averaged Q-learning \cite{liu2026asynchronous}. Setting $r=1$ in \eqref{eq:rate-qualitative-fclt} recovers the endpoint CLT. Here, we seek to quantify the approximation error for the entire path at a finite horizon $T$.

\subsection{Path-space metrics}
\label{sec:rate-path-model}

We measure the error in the Brownian approximation \eqref{eq:rate-qualitative-fclt} using the Prokhorov distance between path laws. For a fixed scalar projection, the underlying path space is one-dimensional \(D([0,1],\R)\). Let \(\Lambda\) be the set of strictly increasing continuous bijections of \([0,1]\) onto itself.

\begin{definition}[\(J_1\) Skorokhod metric]
\label{def:rate-j1-metric}
For \(f,g\in D([0,1],\R)\), define
\begin{equation}
  d_{\mathrm S}(f,g)
  \coloneqq\inf_{\lambda\in\Lambda}
  \left\{
    \sup_{0\le s<t\le1}
    \left|\log\frac{\lambda(t)-\lambda(s)}{t-s}\right|
    \vee
    \sup_{0\le r\le1}|f(\lambda(r))-g(r)|
  \right\}.
  \label{eq:rate-j1-metric}
\end{equation}
\end{definition}

The \(J_1\) Skorokhod metric is widely used in the analysis of stochastic
processes and induces the Skorokhod \(J_1\) topology
\cite[Section~12]{billingsley2013convergence}.  Compared with the uniform
metric
\(
  \tnorm{f-g}\coloneqq\sup_{0\le r\le1}|f(r)-g(r)|
\),
it accommodates a more general class of random-path convergence by allowing
small shifts in jump times.  We use this specific metric throughout the paper. Quantitative rates need not be preserved when it is replaced by a topologically equivalent metric.
Moreover, choosing the identity time change
(i.e., \(\lambda = \mathrm{id}\))
gives the following very useful relationship
\begin{equation}
  d_{\mathrm S}(f,g)\le\tnorm{f-g}.
  \label{eq:j1-bounded-by-uniform}
\end{equation}

We next formally introduce the Prokhorov distance induced by the \(J_1\)
Skorokhod metric.

\begin{definition}[Prokhorov distance]
\label{def:rate-prokhorov}
Consider two Borel probability measures \(\mu\) and \(\nu\) on
\((D([0,1],\R),d_{\mathrm S})\). For a Borel set \(\Gamma\subset D([0,1],\R)\), let \(\Gamma^\delta\coloneqq\{g:\inf_{f\in\Gamma}d_{\mathrm S}(f,g)<\delta\}\) denote its \(\delta\)-enlargement. The Prokhorov distance is
\begin{equation}
\begin{split}
  d_{\mathrm P}(\mu,\nu)\coloneqq\inf\bigl\{\delta>0:
  &\ \mu(\Gamma)\le\nu(\Gamma^\delta)+\delta
  \ \text{and}\
  \nu(\Gamma)\le\mu(\Gamma^\delta)+\delta\\
  &\ \text{for every Borel set }\Gamma\bigr\}.
\end{split}
  \label{eq:rate-prokhorov-definition}
\end{equation}
For random elements \(X,Y\), we use the shorthand
\(d_{\mathrm P}(X,Y)\coloneqq d_{\mathrm P}(\mathcal L(X),\mathcal L(Y))\).
\end{definition}

For random paths $X_T,X$ in $D([0,1],\R)$, convergence $d_{\mathrm P}(X_T,X)\to0$ implies $X_T\Rightarrow X$ in the $J_1$ topology \cite[Sections~6 and~12]{billingsley2013convergence}.

\subsection{Quantitative model assumptions}
\label{sec:rate-quantitative-assumptions}

We now impose quantitative conditions on the recursion, its noise, and the iterate errors. These conditions specify the constants and exponents in the finite-horizon bounds below.

\begin{assumption}[Remainder bounds and positive stability]
\label{ass:rate-linear-dynamics}
For the given population function \(g\) defined by \eqref{eq:population-root-problem}, there exist a matrix \(\bG\in\R^{d\times d}\), a small constant \(\delta_g > 0\), and a constant
\(C_{\mathrm{sm}}<\infty\) such that, for every \(\norm{\bx - \bx^\star} \le \delta_g\),
\begin{equation}
  \norm{g(\bx)-\bG(\bx-\bx^\star)}
  \le C_{\mathrm{sm}}\norm{\bx-\bx^\star}^{2}.
  \label{eq:rate-trajectory-smoothness}
\end{equation}
In addition, for some constant \(C_{\mathrm{bd}}<\infty\),
\[
  \sup_{\bx\in\R^d}
  \norm{g(\bx)-\bG(\bx-\bx^\star)}\le C_{\mathrm{bd}}.
\]
Moreover, \(-\bG\) is Hurwitz. Equivalently, every eigenvalue of \(\bG\) has
strictly positive real part.
\end{assumption}

The local quadratic condition and positive stability are standard in stochastic approximation and are satisfied,
for example, by strongly convex SGD with a globally Lipschitz Hessian and by
entropy-regularized Q-learning under standard feature-stability conditions.
See \cite{anastasiou2019normal,rubtsov2026gaussian}.

\begin{assumption}[Power-law step-size schedule]
\label{ass:rate-gain}
For constants $\eta>0$ and $\alpha\in(1/2,1)$, the step sizes are $\eta_t=\eta t^{-\alpha}$ for $t\ge1$.
\end{assumption}

The range \(\alpha\in(1/2,1)\) gives the classical condition for Polyak--Ruppert averaging where \(\sum_t\eta_t=\infty\) and \(\sum_t\eta_t^2<\infty\)
\cite{ruppert1988efficient,polyak1992acceleration,kushner2003stochastic}.
The exponent \(\alpha\) affects the final approximation rate, as shown in our main theorem below.

\begin{assumption}[Martingale-difference noise]
\label{ass:rate-martingale-noise}
For every \(t\ge1\), define the centered observation noise by \(\bepsilon_t\coloneqq g(\bx_t)-H(\bx_t,\xi_t)\).
We assume that $(\bepsilon_t,\cF_t)_{t\ge1}$ is a martingale-difference sequence. Thus $\bepsilon_t$ is an increment of the martingale $\sum_{j=1}^t\bepsilon_j$, whereas its contribution to the SA update is $\eta_t\bepsilon_t$.  Moreover, for some \(p>2\) and deterministic \(M_p<\infty\), the
following relations hold almost surely for every \(t\ge1\):
\begin{equation}
  \E\!\left[\bepsilon_t\mid\cF_{t-1}\right]=0,
  \qquad
  \E\!\left[\|\bepsilon_t\|^p\mid\cF_{t-1}\right]\le M_p.
  \label{eq:linear-sa-rate-moments}
\end{equation}
\end{assumption}


This assumption allows the noise distribution to depend on past
observations and to change over time; neither independence nor
stationarity is required. The conditional mean-zero property ensures
that the noise introduces no systematic drift at each update, while
the conditional $p$th-moment bound limits the probability of large
noise increments. Together, these conditions allow us to control
the maximum of the noise partial sums using the martingale inequality
of \cite{fan2017deviation}. Related martingale-difference assumptions
are used in SA and SGD analysis
\cite{kushner2003stochastic,anastasiou2019normal}.

\begin{assumption}[Quantitative stabilization of cumulative conditional covariance]
\label{ass:rate-covariance-stabilization}
There exist a positive-semidefinite matrix \(\bS\in\R^{d\times d}\), a constant
\(C_S<\infty\), and an exponent \(\gamma>0\) such that, for every integer
\(T\ge2\),
\begin{equation}
  \E\!\left[\max_{0\le n\le T}
  \left\|
    \frac1T\sum_{t=1}^n
      \E\!\left[\bepsilon_t\bepsilon_t^\top\mid\cF_{t-1}\right]
    -\frac nT \bS
  \right\|
  \right]
  \le C_S T^{-3\gamma}.
  \label{eq:linear-sa-rate-covariance-stabilization}
\end{equation}
\end{assumption}


The normalized cumulative conditional covariance in
\eqref{eq:linear-sa-rate-covariance-stabilization} is the predictable
quadratic variation of the normalized noise partial sums
\cite{hall2014martingale,whitt2007proofs}.
This assumption requires it to be close to $(n/T)\bS$ at every
intermediate time $n\le T$, with the maximum discrepancy bounded
in expectation by $C_S T^{-3\gamma}$.
The moment bound in Assumption~\ref{ass:rate-martingale-noise}
controls the size of the noise but does not guarantee this convergence.
If the bound holds for some $\gamma>0$, it also holds with the same
constant for any smaller positive exponent.
If $\E[\bepsilon_t\bepsilon_t^\top\mid\cF_{t-1}]=\bS$
almost surely at every step, the discrepancy is zero and we may
take $C_S=0$. The matrix $\bS$ may be singular, allowing zero
limiting variance in some directions.

\begin{assumption}[Trajectory moment bound]
\label{ass:rate-trajectory-stability}
For a given \(q \in [1,p/2]\) (\(p\)
is defined in Assumption~\ref{ass:rate-martingale-noise}),
 there is a constant
\(C_\Delta<\infty\) such that, for every \(t\ge1\),
\begin{equation}
  \E\big[\norm{\bDelta_t}^{2q}\big]\le C_\Delta\eta_t^q.
  \label{eq:rate-trajectory-stability}
\end{equation}
\end{assumption}

For \(q=1\), this assumption reduces to the familiar nonasymptotic
mean-square bound
\(\E[\norm{\bx_t-\bx^\star}^{2}]=O(\eta_t)\).
Higher-moment bounds are also available under suitable noise and
stability conditions. For strongly convex nonlinear SGD with finite
fourth moments, Shao and Zhang
\cite[Lemmas~5.12 and~5.14]{shao2022berry} establish the cases
\(q=1\) and \(q=2\). Under stronger sub-Gaussian assumptions,
Sheshukova et al.~\cite[Lemma~14]{sheshukova2026gaussian}
obtain analogous bounds for arbitrary fixed moment orders.
Related nonasymptotic mean-square bounds under contractivity or
Lyapunov-drift conditions can be found in
\cite{moulines2011non,chen2020finite,chen2021lyapunov}.

We impose this bound explicitly because local positive stability
alone does not guarantee a finite-sample bound of this order.
The parameters \(p\) and \(q\) describe different moment conditions:
\(p\) concerns the noise, whereas \(2q\) concerns the iterate error.
When trajectory moments are derived using only \(p\)th-moment
control of the noise, the available range is typically \(2q\le p\),
as imposed here. Higher orders \(q>p/2\) require additional
integrability and a corresponding trajectory-moment bound.
Combining Assumptions~\ref{ass:rate-linear-dynamics}
and~\ref{ass:rate-trajectory-stability} yields $\E\big[\norm{g(\bx_t)-\bG\bDelta_t}^{q}\big] \le C\eta_t^q,$ which is the nonlinear-remainder estimate used below.

Let \(\bW_\star\) be a standard \(d\)-dimensional Brownian motion,
and let \(\bS^{1/2}\) denote the symmetric positive-semidefinite
square root of \(\bS\). Define the Brownian target by
\begin{equation}
  \bZ_\star(r)\coloneqq\bG^{-1}\bS^{1/2}\bW_\star(r),
  \qquad 0\le r\le1.
  \label{eq:linear-sa-formal-brownian-target}
\end{equation}
Its covariance at time \(r\) is \(r\bG^{-1}\bS\bG^{-\top}\).
We use \(W_\star^{(1)}\) to denote a standard one-dimensional
Brownian motion when specifying a scalar target.
\section{Main Results}
\label{sec:rate-model-results}

This section presents the quantitative results in two parts.
Subsection~\ref{sec:rate-upper-bounds} gives the functional central limit
theorem upper bound for martingale-difference noise and discusses its
i.i.d. specialization.
Subsection~\ref{sec:rate-lower-bounds} gives four benchmark lower bounds that
isolate the obstructions produced by conditional-covariance stabilization,
finite moments, algorithmic smoothing, and the nonlinear remainder.

\subsection{Upper Bounds}
\label{sec:rate-upper-bounds}

We first state the upper bound under the path-space formulation and
assumptions of Section~\ref{sec:rate-quantitative-assumptions}, using the metric in Section~\ref{sec:rate-path-model}.
\begin{theorem}[Quantitative functional CLT for fixed projections of nonlinear SA]
\label{thm:linear-fclt-rate}
Suppose Assumptions~\ref{ass:rate-linear-dynamics},
\ref{ass:rate-gain}, \ref{ass:rate-martingale-noise},
\ref{ass:rate-covariance-stabilization}, and
\ref{ass:rate-trajectory-stability} hold.
Recall the four rate exponents
\[
  \rateCov=\gamma,\qquad
  \rateOsc=\frac{p-2}{2(p+1)},\qquad
  \rateSm=\frac{1-\alpha}{2},\qquad
  \rateNL=\frac{(\alpha-1/2)q}{q+1}.
\]
For every fixed \(\btheta\in\R^d\), there is a finite constant \(C\), which
is independent of \(T\), such that for every integer \(T\ge2\),
\begin{equation}
\begin{aligned}
  d_{\mathrm P}\!\left(
    \btheta^\top\bPhi_T,
    \btheta^\top\bZ_\star
  \right)
  \le C\Bigl\{&
    \underbrace{T^{-\rateCov}\sqrt{\log T}}_{\text{covariance stabilization}}
    +\underbrace{T^{-\rateOsc}\sqrt{\log T}}_{\text{finite-moment noise}}\\
    &+\underbrace{T^{-\rateSm}\sqrt{\log T}}_{\text{algorithmic smoothing}}
    +\underbrace{T^{-\rateNL}}_{\text{nonlinear remainder}}
  \Bigr\}.
\end{aligned}
  \label{eq:linear-sa-fclt-rate}
\end{equation}
\end{theorem}

\noindent\textit{\textbf{The sources of the four rates.}}
The four rates in \eqref{eq:linear-sa-fclt-rate} arise from three structural components: the leading noise partial sum, the algorithmic smoothing induced by the SA recursion, and the nonlinear remainder. The proof separates these components, bounds their contributions to the path error, and then combines the bounds. At a grid point $n/T$, the decomposition has the schematic form
\begin{align*}
  \bPhi_T(n/T)
  ={}&\underbrace{\frac1{\sqrt T}\sum_{j=1}^n\bG^{-1}\bepsilon_j}_{\text{leading noise partial sum}}
  +\underbrace{\frac1{\sqrt T}\sum_{j=1}^n
    (\mathsf{Poly}(\bG)-\bG^{-1})\bepsilon_j}_{\text{algorithmic smoothing}}\\
  &-\underbrace{\frac1{\sqrt T}\sum_{j=1}^n
    \mathsf{Poly}(\bG)(g(\bx_j)-\bG\bDelta_j)}_{\text{nonlinear remainder}}
  +\underbrace{\vphantom{\frac1{\sqrt T}\sum_{j=1}^n
    \mathsf{Poly}(\bG)(g(\bx_j)-\bG\bDelta_j)}\text{initialization term}}_{O(T^{-1/2})}.
\end{align*}
The notation $\mathsf{Poly}(\bG)$ records the matrix coefficients generated by iterating the SA recursion. Each such coefficient is a finite weighted sum of products of the update matrices $\bI-\eta_i\bG$, and is therefore a matrix polynomial in $\bG$. We suppress the dependence of its coefficients on $j$, $n$, and, when relevant, $T$. Different occurrences, including different summands within a single sum, may represent different polynomials. The notation serves only to display the sources of error here. In particular, the smoothing term compares these recursion coefficients with the coefficient $\bG^{-1}$ in the leading noise partial sum.

Write $\bmM_T(n/T)=T^{-1/2}\sum_{j=1}^n\bepsilon_j$ and interpolate affinely. Then $\bmM_T$ is the noise partial-sum process, $\bG^{-1}\bmM_T$ is the leading term of the SA path, and $\btheta^\top\bG^{-1}\bmM_T$ is the scalar process compared with Brownian motion. This leading term contributes two rates. Its cumulative conditional covariance need not equal the linear profile $r\bG^{-1}\bS\bG^{-\top}$. The stabilization error in Assumption~\ref{ass:rate-covariance-stabilization} contributes $T^{-\rateCov}\sqrt{\log T}$. Even when that profile is exact, the conditional $p$th-moment bound in Assumption~\ref{ass:rate-martingale-noise} permits rare large noise increments. These produce the separate finite-moment contribution $T^{-\rateOsc}\sqrt{\log T}$.

The leading partial sum adds $\bG^{-1}\bepsilon_j$ once at time $j/T$. In the SA path, $\bepsilon_j$ first affects $\bx_{j+1}$ and then propagates through subsequent updates, spreading its contribution over a sequence of partial sums. We call this gradual accumulation \emph{algorithmic smoothing}. It can make the SA partial-sum path locally smoother than its Brownian target, as the Gaussian model in Subsection~\ref{sec:rate-lower-smoothing} demonstrates. Controlling the difference from the leading noise process produces $T^{-\rateSm}\sqrt{\log T}$. The nonlinear term instead contains the error $g(\bx_j)-\bG\bDelta_j$ from replacing the mean field by its linearization. Step~2 of the proof uses an $L^q$ bound of order $\eta_j$ for this error. After summation and normalization by $\sqrt T$, the resulting $L^q$ bound has order $T^{-(\alpha-1/2)}$. Its conversion to a probability radius gives $T^{-\rateNL}$.

\noindent\textit{\textbf{Dimension and scope.}}
Theorem~\ref{thm:linear-fclt-rate} is a fixed-dimensional result for each fixed projection $\btheta$. Its constant may depend on $d$, $\btheta$, the matrix $\bG$, and the moment and stability constants in the assumptions. The proof does not track these dependencies sharply. In particular, the displayed powers of $T$ do not constitute a dimension-free guarantee or a bound uniform over all projection directions. The scalar lower bounds below establish obstructions in $T$, without establishing optimal dependence on $d$.

\noindent\textit{\textbf{The step-size tradeoff.}}
Among the four contributions in the bound \eqref{eq:linear-sa-fclt-rate}, the step-size exponent $\alpha$ trades off the nonlinear remainder $T^{-\rateNL}$ against algorithmic smoothing $T^{-\rateSm}\sqrt{\log T}$. Faster step-size decay reduces the iterate errors but spreads the response to each noise input over more updates. For a fixed admissible $q$, increasing $\alpha$ in $\eta_t=\eta t^{-\alpha}$ therefore increases $\rateNL=(\alpha-1/2)q/(q+1)$ but decreases $\rateSm=(1-\alpha)/2$. Ignoring logarithmic factors, the slower of these two terms is determined by $\min\{\rateNL,\rateSm\}$. Balancing the exponents gives $\alpha_q^\star=(2q+1)/(3q+1)$ and the common exponent $q/\{2(3q+1)\}$, illustrated in Figure~\ref{fig:rate-tradeoffs} (left). This choice optimizes the balance between the two SA contributions. If a noise contribution decays more slowly, it determines the overall bound, and other step-size exponents may attain the same overall rate.

\noindent\textit{\textbf{The i.i.d. specialization.}}
When the noise increments $\bepsilon_t$ are i.i.d. and independent of $\cF_{t-1}$, their conditional covariance is constant, so the covariance-stabilization term vanishes. The remaining rate still depends on two kinds of moment control. Recall that $p>2$ is the moment order of the noise in Assumption~\ref{ass:rate-martingale-noise}, whereas $2q$ is the moment order of the iterate error $\bDelta_t$ in Assumption~\ref{ass:rate-trajectory-stability}. The former limits the frequency and size of large noise inputs. The latter controls how strongly the iterates concentrate around the target and thus enters the nonlinear remainder estimate. With these two moment orders specified, optimizing the upper bound over $\alpha$ gives, up to logarithmic factors, the exponent
\begin{equation}
  \min\left\{\frac{p-2}{2(p+1)},\frac{q}{2(3q+1)}\right\}.
  \label{eq:rate-iid-optimized-general-q}
\end{equation}
The first term is the finite-moment noise exponent, and the second is the best balance between nonlinearity and algorithmic smoothing. The smaller one determines the overall rate.

A natural special case uses the same moment order for the noise and the iterate error, so that $2q=p$, or $q=p/2$. For example, fourth-moment control of both corresponds to $p=4$ and $q=2$. This matching requires a moment-propagation argument under suitable stability conditions. It does not follow from independence alone. When that argument supplies Assumption~\ref{ass:rate-trajectory-stability} with $q=p/2$, the balancing step-size exponent is $\alpha=2(p+1)/(3p+2)$, and \eqref{eq:rate-iid-optimized-general-q} becomes
\begin{equation}
  \min\left\{\frac{p-2}{2(p+1)},\frac{p}{2(3p+2)}\right\}.
  \label{eq:rate-iid-optimized-moment-matched}
\end{equation}
For such moment bounds, \cite[Lemma~5.14]{shao2022berry} establishes the fourth-moment case $q=2$, while \cite[Lemma~14]{sheshukova2026gaussian} provides arbitrary fixed trajectory-moment orders under stronger sub-Gaussian conditions. Choosing $q=p$ would instead require a $2p$th trajectory moment, which is not supplied by a $p$th noise-moment bound alone. In the moment-matched case of \eqref{eq:rate-iid-optimized-moment-matched}, the exponent tends to $1/6$ as $p\to\infty$. The two branches meet at $p_c=(5+\sqrt{57})/4\approx3.137$, as shown in Figure~\ref{fig:rate-tradeoffs} (right).

\begin{figure}[t]
  \centering
  \includegraphics[width=\textwidth]{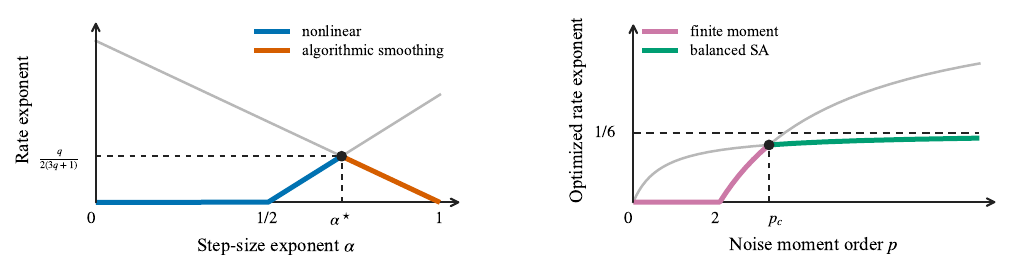}
  \caption{Step-size and moment tradeoffs at the level of polynomial
  exponents.  Thick colored segments form the rate-limiting lower envelope.
  Thin gray segments are inactive.  The left panel illustrates the two
  \(\alpha\)-dependent exponents at \(q=2\).  The right panel shows the
  optimized i.i.d. exponent under the moment-matched choice \(q=p/2\).
  The colored horizontal extensions mark the admissible-range thresholds
  \(\alpha=1/2\) and \(p=2\).}
  \label{fig:rate-tradeoffs}
\end{figure}

\ifdefined\isarxivpreprint
\noindent\textit{\textbf{Comparison with the preliminary bound.}}
A preliminary upper bound appeared in our earlier
work~\cite[Theorem~5]{li2023nonlinearv2}. For ease of comparison, we specialize both results to the i.i.d.\ setting with \(p\ge 4\). The terms that determine the rates then take the following schematic forms:
\(
  T^{-a_{\mathrm{osc}}}
  +T^{-a_{\mathrm{sm}}}
  +T^{-a_{\mathrm{nl}}}
\).
The table below compares the corresponding exponents.
\begin{center}
\renewcommand{\arraystretch}{1.55}
\setlength{\tabcolsep}{1.2em}
\begin{tabular}{@{}lccc@{}}
  \toprule
  Result & \(a_{\mathrm{osc}}\) & \(a_{\mathrm{sm}}\) & \(a_{\mathrm{nl}}\) \\
  \midrule
  \cite[Theorem~5]{li2023nonlinearv2}
    & \(\dfrac{p-2}{2(p+1)}(1-\alpha)\)
    & \(\dfrac{1-\alpha}{3}\)
    & \(\dfrac{\alpha-1/2}{\alpha+1}\) \\
  \midrule
  Theorem~\ref{thm:linear-fclt-rate}
    & \(\dfrac{p-2}{2(p+1)}\)
    & \(\dfrac{1-\alpha}{2}\)
    & \(\dfrac{(\alpha-1/2)q}{q+1}\) \\
  \bottomrule
\end{tabular}
\end{center}
Under the moment-matched choice \(q=p/2\) discussed above, every exponent in
this paper is strictly larger than the previous version.  Therefore, Theorem~\ref{thm:linear-fclt-rate} yields faster and sharper polynomial decay for all three contributions: finite-moment noise, algorithmic smoothing, and nonlinearity.
\fi

\subsection{Mechanism-specific Lower Bounds}
\label{sec:rate-lower-bounds}

In this section, we will give the matching lower bound for Theorem~\ref{thm:linear-fclt-rate}.
In order to make the analysis of this lower bound more rigorous,
before entering the specific analysis of every case,
we first formalize the stochastic oracles used in the lower-bound
constructions.
\begin{definition}[Stochastic oracle]
A stochastic oracle \(\mathcal M\) is a tuple
\[
  \mathcal M
  \coloneqq \left(
    g_{\mathcal M},
    (\epsilon_t^{\mathcal M})_{t\ge1},
    (\mathcal F_t^{\mathcal M})_{t\ge0},
    \mathbb P_{\mathcal M}
  \right),
\]
where \(g_{\mathcal M}:\R\to\R\) is the mean field in
\eqref{eq:population-root-problem}, and
\((\epsilon_t^{\mathcal M})_{t\ge1}\) is a martingale-difference sequence
adapted to \((\mathcal F_t^{\mathcal M})_{t\ge0}\) under the probability
\(\mathbb P_{\mathcal M}\).
\end{definition}
Throughout this subsection, all lower-bound dynamics use the common
initialization \(x_1^{\mathcal M}=x^\star=0\).  Given the step sizes, the
stochastic approximation sequence driven by \(\mathcal M\) is denoted by
\(x_t^{\mathcal M}\).
We write \(\E_{\mathcal M}\) for expectation under
\(\mathbb P_{\mathcal M}\) and \(\Phi_T^{\mathcal M}\) for the resulting
partial-sum path in \eqref{eq:partial-sum-process}.  Consistently with the
notation above, let \(\bmM_T^{\mathcal M}\) be the affine interpolation of
\[
  \bmM_T^{\mathcal M}(n/T)
  \coloneqq
  \frac1{\sqrt T}\sum_{j=1}^n\epsilon_j^{\mathcal M},
  \qquad 0\le n\le T.
\]
When the relevant
assumptions hold, let \(G_{\mathcal M}\) and \(S_{\mathcal M}\) denote the
linearization coefficient and limiting variance in Assumption
\ref{ass:rate-linear-dynamics} and \ref{ass:rate-covariance-stabilization},
respectively, and define the
prescribed Brownian target by
\(
  Z^{\mathcal M}
  \coloneqq
  G_{\mathcal M}^{-1}S_{\mathcal M}^{1/2}W_\star^{(1)}
\).

Now, let \(\mathfrak C\) be a nonempty stochastic oracle class.  For the
prescribed stochastic approximation algorithm, define the worst-case Brownian
approximation error over \(\mathfrak C\) by
\begin{equation}
  \mathcal D_T(\mathfrak C)
  \coloneqq
  \sup_{\mathcal M\in\mathfrak C}
  d_{\mathrm P}\!\left(\Phi_T^{\mathcal M},Z^{\mathcal M}\right).
  \label{eq:rate-worst-case-path-error}
\end{equation}

The four terms in \eqref{eq:linear-sa-fclt-rate} arise from different sources of the Brownian approximation error $d_{\mathrm P}\!(
    \btheta^\top\bPhi_T,
    \btheta^\top\bZ_\star)$.
    In the rest of this section, we assess the sharpness of each term separately, using a different class of stochastic oracles \(\mathfrak C\). Within each class, we construct a lower bound for the corresponding error. Therefore, the four effects need not arise simultaneously in a single model.

To combine the four lower bounds, we consider the union of the corresponding oracle classes. The worst-case Brownian approximation error over this union of models is no smaller than that over any individual class. That is, denoting the four classes by \(\mathfrak C_1,\ldots,\mathfrak C_4\), we have
\[
  \mathcal D_T\!\left(\bigcup_{j=1}^4\mathfrak C_j\right)
  = \max_{1\le j\le4}\mathcal D_T(\mathfrak C_j)
  \ge \frac14\sum_{j=1}^4 \mathcal D_T(\mathfrak C_j).
\]
We next give the four constructions and their conditions. Subsection~\ref{sec:rate-overall-lower} then combines these model classes and states the overall lower bound.

\subsubsection{Error due to conditional-covariance stabilization}
\label{sec:rate-lower-covariance}

For the covariance contribution, we use martingale-difference noises with uniformly bounded conditional moments of every finite order. In this class, an expected discrepancy of order \(T^{-3\gamma}\) in cumulative conditional covariance can still cause a Prokhorov error of order \(T^{-\gamma}\). We first prove this lower bound for the noise partial sums and then relate it to the full SA path.

\begin{theorem}[Sharpness of the conditional-covariance-stabilization exponent]
\label{thm:rate-covariance-stabilization-lower}
Suppose Assumption~\ref{ass:rate-gain} holds.  Fix \(p>2\),
\(q\in[1,p/2]\), and suppose
\(\rateCov<\min\{\rateOsc,\rateSm\}\).  One can construct an oracle class
\(\mathfrak C_{\mathrm{cov}}\) whose dynamics satisfy
Assumptions~\ref{ass:rate-linear-dynamics}--
\ref{ass:rate-trajectory-stability} uniformly over the class, with limiting
variance one and covariance-stabilization exponent \(\rateCov\).  Moreover,
there exist constants \(c>0\) and \(T_0<\infty\) such that, for every integer
\(T\ge T_0\),
\begin{equation}
  \mathcal D_T(\mathfrak C_{\mathrm{cov}})
  \ge cT^{-\rateCov}.
  \label{eq:rate-covariance-stabilization-lower}
\end{equation}
\end{theorem}

\begin{proof}[Proof idea]
The key idea is to construct a noise process whose cumulative conditional variance is close to that of Brownian motion, while its terminal value has an atom at zero. This atom forces a Brownian approximation error of order \(T^{-\gamma}\). Let \(W\) be a standard Brownian motion, and let \(\tau_T\) be its first hitting time of zero after \(T-\frac14T^{1-2\gamma}\), truncated at \(T\):
\[
\tau_T
\coloneqq
\inf\left\{
t\ge T-\frac14T^{1-2\gamma}:W(t)=0
\right\}\wedge T.
\]
Define the noise sequence of the stochastic oracle \(\mathcal M_T\) by
\[
\epsilon_j^{\mathcal M_T}
\coloneqq
W(j\wedge\tau_T)-W((j-1)\wedge\tau_T),
\qquad 1\le j\le T.
\]
These are martingale differences obtained by stopping Brownian motion at \(\tau_T\). If \(W\) hits zero before \(T\), the stopped process remains at zero, creating an atom in the terminal partial sum. Specifically,
\[
  \Pp(W(\tau_T)=0)=\Pp(\tau_T<T)\asymp T^{-\gamma},
  \qquad
  \frac{\E(T-\tau_T)}{T}\asymp T^{-3\gamma}.
\]
The first relation gives an atom of mass at least \(cT^{-\gamma}\) at zero in \(\bmM_T^{\mathcal M_T}(1)=W(\tau_T)/\sqrt T\). The second quantifies the expected maximum discrepancy in normalized cumulative conditional variance, which arises from the Brownian increments removed after \(\tau_T\).

To obtain the Prokhorov lower bound, take \(A=\{x:x(1)=0\}\). 
Under the $J_1$ topology, every admissible time change fixes the endpoint, so the \(\delta\)-neighborhood \(A^\delta\) of \(A\) satisfies \(A^\delta\subseteq\{x:|x(1)|<\delta\}\). Since \(W_\star^{(1)}(1)\) is standard Gaussian, we have
\[
  \Pp(\bmM_T^{\mathcal M_T}\in A)=\Pp(W(\tau_T)=0)\ge cT^{-\gamma},
  \qquad
  \Pp(W_\star^{(1)}\in A^\delta)\le C\delta.
\]
A Prokhorov distance smaller than \(\delta\) would require \(cT^{-\gamma}\le (C+1)\delta\), which fails when \(\delta\) is a sufficiently small multiple of \(T^{-\gamma}\). Thus
\(d_{\mathrm P}(\bmM_T^{\mathcal M_T},W_\star^{(1)})\gtrsim T^{-\gamma}\).
Finally, when \(\gamma<\min\{\rateOsc,\rateSm\}\), the linear-SA remainder estimates yield $ d_{\mathrm P}(\Phi_T^{\mathcal M_T},\bmM_T^{\mathcal M_T})
  =o(T^{-\gamma}).$
The triangle inequality transfers the lower bound to the full SA path. \suppappendixref{app:rate-covariance-stabilization-lower} gives the complete construction and verifies the assumptions.

\end{proof}


\subsubsection{Error due to limited noise moments}
\label{sec:rate-lower-moments}

For the finite-moment contribution, we consider scalar linear SA with centered, unit-variance i.i.d.\ noise under one fixed \(p\)th-moment bound. The conditional covariance is now exact, but rare large increments can still create unusually large local fluctuations. The following lower bound applies directly to the full SA path over this class.

\begin{theorem}[Worst-case finite-moment lower bound for linear SA]
\label{thm:rate-iid-leading-minimax}
Suppose Assumption~\ref{ass:rate-gain} holds and \(\rateOsc<\rateSm\).
Fix \(p>2\) and \(q\in[1,p/2]\).  One can construct an oracle class
\(\mathfrak C_{\mathrm{osc}}\) whose dynamics satisfy
Assumptions~\ref{ass:rate-linear-dynamics}--
\ref{ass:rate-trajectory-stability} uniformly over the class.  Moreover,
there exist constants \(c>0\) and \(T_0<\infty\) such that, for every integer
\(T\ge T_0\),
\begin{equation}
  \mathcal D_T(\mathfrak C_{\mathrm{osc}})
  \ge cT^{-\frac{p-2}{2(p+1)}}.
  \label{eq:linear-iid-leading-minimax}
\end{equation}
\end{theorem}

\begin{proof}[Proof idea]
The lower bound comes from a short-time fluctuation that occurs
much more often in the SA path than in Brownian motion. With \(\rho_T=T^{-\rateOsc}\), we construct centered, unit-variance i.i.d.\ noise with a rare component: an input has magnitude of order \(\sqrt T\,\rho_T\) with probability of order \(\rho_T/T\), and is otherwise bounded. This construction satisfies the uniform \(p\)th-moment bound because $ \frac{\rho_T}{T}(\sqrt T\,\rho_T)^p=1.$
Among the first \(T\) updates, the probability that exactly one large input occurs at an index \(j\in[T/3,T/2]\) is of order \(\rho_T\).

Suppose this large input occurs at \(j\). It enters the update for \(x_{j+1}\), and its effect persists through the subsequent SA iterates. 
Over the \(m_j=\lfloor\eta_j^{-1}\rfloor\asymp T^\alpha\) updates following index \(j\), it produces a change of order \(\rho_T\) in the normalized path \(\Phi_T^{\mathcal M_T}\).
 The other inputs do not typically obscure this change: conditional on there being no other large input, their contribution is centered with standard deviation at most
\(CT^{-\rateSm}=o(\rho_T)\), since \(\rateOsc<\rateSm\). Thus the path changes by at least \(c_0\rho_T\) over this interval with conditional probability tending to one, for some \(c_0>0\). We record such a change through the event
\[
  \mathcal J_T
  \coloneqq
  \left\{x:
    \max_{\lceil T/3\rceil\le j\le\lfloor T/2\rfloor}
    \left|x\!\left(\frac{j+m_j}{T}\right)
          -x\!\left(\frac jT\right)\right|
    \ge c_0\rho_T
  \right\}.
\]

The intervals in \(\mathcal J_T\) have length
\(m_j/T\asymp T^{\alpha-1}\). A Brownian increment over an interval this short is unlikely to reach size \(\rho_T\), even after searching over all the indices in the definition of \(\mathcal J_T\). This remains true for a sufficiently small \(J_1\) neighborhood \(\mathcal J_T^{c'\rho_T}\): being within \(J_1\) distance \(c'\rho_T\) can change the size of an increment by at most \(2c'\rho_T\) and stretch its time interval by at most a factor \(e^{c'\rho_T}\). Hence the construction and the Brownian modulus bound yield
\[
  \Pp(\Phi_T^{\mathcal M_T}\in\mathcal J_T)\ge c\rho_T,
  \qquad
  \Pp\!\left(W_\star^{(1)}\in\mathcal J_T^{c'\rho_T}\right)
  =o(\rho_T).
\]
For \(c'<c\) sufficiently small, these two probabilities violate the Prokhorov inequality at radius \(c'\rho_T\) for all sufficiently large \(T\). Therefore
\[
  d_{\mathrm P}(\Phi_T^{\mathcal M_T},W_\star^{(1)})
  \gtrsim \rho_T=T^{-\rateOsc}.
\]
The complete proof is given in \suppappendixref{app:rate-iid-leading-lower}.

\end{proof}


\subsubsection{Error due to noise propagation through SA}
\label{sec:rate-lower-smoothing}

For algorithmic smoothing, a class containing just one scalar linear SA model with i.i.d.\ Gaussian noise is enough. Each noise input contributes to the partial-sum path over several updates, making it locally smoother than Brownian motion. The model remains unchanged as \(T\) grows, so the lower bound below reflects a persistent feature of the recursion.

\begin{theorem}[Sharp intrinsic rate in the scalar Gaussian model]
\label{thm:linear-gaussian-j1-sharpness}
Suppose Assumption~\ref{ass:rate-gain} holds.  Let
\(\mathfrak C_{\mathrm{sm}}=\{\mathcal M_{\mathrm{sm}}\}\), where
\(\mathcal M_{\mathrm{sm}}\) is the exact-linear scalar oracle with
\(g_{\mathcal M_{\mathrm{sm}}}(x)=x\) and i.i.d.\ standard Gaussian noise.
For this oracle,
\(G_{\mathcal M_{\mathrm{sm}}}=S_{\mathcal M_{\mathrm{sm}}}=1\), so
\(Z^{\mathcal M_{\mathrm{sm}}}=W_\star^{(1)}\).
The dynamics generated by this oracle satisfy
Assumptions~\ref{ass:rate-linear-dynamics}--
\ref{ass:rate-trajectory-stability}.  There are constants \(c>0\) and
\(T_0<\infty\), independent of \(T\), such that, for every integer
\(T\ge T_0\),
\begin{equation}
  \mathcal D_T(\mathfrak C_{\mathrm{sm}})
  \ge cT^{-\frac{1-\alpha}{2}}\sqrt{\log T}.
  \label{eq:linear-gaussian-j1-order}
\end{equation}
\end{theorem}

\begin{proof}[Proof idea]
The lower bound comes from a difference in local fluctuations:
the SA path is likely to have small increments over all short
intervals, whereas Brownian motion is unlikely to do so.
Let \(\Phi_T=\Phi_T^{\mathcal M_{\mathrm{sm}}}\) denote the
normalized partial-sum path in the scalar Gaussian model, and
let \(W_\star^{(1)}\) be its standard Brownian target.
Fix \(0<u<v<1\), and put \(h_T=c_0T^{\alpha-1}\),
\(b_T=\sqrt{2h_T\log(1/h_T)}\), and \(\delta_T=c_1b_T\),
where \(c_0>0\) is sufficiently small and
\(0<c_1<(v-u)/2\). Consider the event
\[
  \mathcal J_T
  \coloneqq
  \left\{x\in C([0,1],\R):
    \sup_{\substack{s,t\in[1/4,3/4]\\
                    |t-s|\le e^{\delta_T}h_T}}
    |x(t)-x(s)|\le ub_T
  \right\},
\]
which requires every increment over the indicated short
intervals to have magnitude at most \(ub_T\).

For the SA path \(\Phi_T\), linearity of the recursion and
Gaussian noise imply that each increment
\(\Phi_T(k/T)-\Phi_T(j/T)\) is centered Gaussian.
The recursion bounds the variances of the grid increments
needed to cover these short intervals by \(Cc_0h_T\).
Gaussian tail bounds, a union bound, and linear interpolation
therefore give \(\Pp(\Phi_T\in\mathcal J_T)\to1\) for
sufficiently small \(c_0\).
In contrast, Brownian motion \(W_\star^{(1)}\) has order
\(h_T^{-1}\) disjoint increments of length \(h_T\) in
\([1/3,2/3]\). These increments are independent, each with
distribution \(N(0,h_T)\). The Gaussian tail lower bound
shows that at least one exceeds \(vb_T\) in absolute value
with probability tending to one.

Under the \(J_1\) topology, with the metric \(d_{\mathrm S}\),
a path \(y\in\mathcal J_T^{\delta_T}\) can be compared with
some \(x\in\mathcal J_T\) through a time change \(\lambda\)
such that
\[
  |y(r)-x(\lambda(r))|<\delta_T,
  \qquad
  |\lambda(t)-\lambda(s)|\le e^{\delta_T}|t-s|.
\]
Thus, if \(|t-s|\le h_T\), the corresponding times
\(\lambda(s)\) and \(\lambda(t)\) are at most
\(e^{\delta_T}h_T\) apart. For sufficiently large \(T\),
the time change maps \([1/3,2/3]\) into \([1/4,3/4]\),
so the definition of \(\mathcal J_T\) gives
\(|x(\lambda(t))-x(\lambda(s))|\le ub_T\).
The values of \(y\) differ from those of the time-changed
path \(x\) by less than \(\delta_T\) at each endpoint.
Therefore \(|y(t)-y(s)|\le ub_T+2\delta_T<vb_T\).
Consequently,
\[
  \Pp(\Phi_T\in\mathcal J_T)\to1,
  \qquad
  \Pp(W_\star^{(1)}\in\mathcal J_T^{\delta_T})\to0.
\]
Given the last result, since \(\delta_T\to0\), for all sufficiently large \(T\),
\[
  \Pp(\Phi_T\in\mathcal J_T)
  >
  \Pp(W_\star^{(1)}\in\mathcal J_T^{\delta_T})+\delta_T.
\]
This violates the defining Prokhorov inequality at radius
\(\delta_T\), so
\(d_{\mathrm P}(\Phi_T,W_\star^{(1)})
\ge\delta_T\asymp T^{-\rateSm}\sqrt{\log T}\).
The complete proof is given in
\suppappendixref{app:linear-gaussian-j1-sharpness}.

\end{proof}

This single model rules out improving either the algorithmic smoothing exponent or its square-root logarithmic factor in a uniform bound over any class containing it, with the same normalization and \(J_1\) Skorokhod metric.

\subsubsection{Error due to nonlinearity}
\label{sec:rate-lower-nonlinearity}

For the nonlinear contribution, we allow the mean field to vary across models while keeping a common trajectory-moment bound. When the mean field is small relative to the iterate error, an initial deviation can persist and produce a cumulative error. The construction has \(S=0\), so its Gaussian target is the zero path. Its local regularity parameters vary with the construction horizon, which limits the sense in which this lower bound can match the upper bound.

\begin{theorem}[A nonlinear obstruction under a trajectory moment bound]
\label{thm:rate-nonlinear-oracle-lower}
Fix \(\alpha\in(1/2,1)\), \(\eta>0\), \(p>2\), and
\(q\in[1,p/2]\), and set \(\eta_t=\eta t^{-\alpha}\).
One can construct a fixed stochastic-oracle class
\(\mathfrak C_{\mathrm{nl}}\) such that every
\(\mathcal M\in\mathfrak C_{\mathrm{nl}}\) satisfies
Assumptions~\ref{ass:rate-linear-dynamics}--
\ref{ass:rate-trajectory-stability}, with limiting variance
\(S_{\mathcal M}=0\).  Consequently, the limiting path
\(Z^{\mathcal M}\) is identically zero for every
\(\mathcal M\in\mathfrak C_{\mathrm{nl}}\).  Moreover, there exist a
constant \(c>0\) and
\(T_0<\infty\) such that, for every integer \(T\ge T_0\),
\begin{equation}
  \sup_{\mathcal M\in\mathfrak C_{\mathrm{nl}}}
  d_{\mathrm P}\!\left(\Phi_T^{\mathcal M},0\right)
  \ge cT^{-\frac{(\alpha-1/2)q}{q+1}}.
  \label{eq:rate-nonlinear-oracle-lower}
\end{equation}
\end{theorem}

\begin{proof}[Proof idea]
A rare initial noise input moves the iterate away from zero to
a region where the mean field is too small to bring it back quickly.
The resulting error persists over many updates, placing the
partial-sum path at distance of order \(T^{-\rateNL}\) from the
zero path with probability of the same order. This yields the
Prokhorov lower bound.

Define \(\rho_T\coloneqq T^{-(\alpha-1/2)q/(q+1)}\).
For each \(T\), we select an oracle
\(\mathcal M_T\in\mathfrak C_{\mathrm{nl}}\) whose mean field
satisfies \(g_{\mathcal M_T}(x)=\lambda x\) below the threshold
\(b_T=\rho_T^{3/2}/\sqrt T\) and
\(g_{\mathcal M_T}(x)=c_0T^{\alpha-1}x\) on \([2b_T,1]\), with smooth
transitions. Write
\(\Delta_t^{\mathcal M_T}=x_t^{\mathcal M_T}-x^\star\).
Starting from \(x_1^{\mathcal M_T}=x^\star=0\), we choose a
centered first noise input so that the event
\(E_T=\{\Delta_2^{\mathcal M_T}=\rho_T/\sqrt T\}\)
has probability \(\rho_T\). On its complement \(E_T^c\),
the first update gives
\(\Delta_2^{\mathcal M_T}=-\rho_T^2/\{\sqrt T(1-\rho_T)\}\).
All subsequent noise increments are exactly zero.
On \(E_T\), induction gives, for some \(\kappa>0\),
\[
  \frac{\kappa\rho_T}{\sqrt T}
  \le \Delta_t^{\mathcal M_T}
  \le \frac{\rho_T}{\sqrt T},
  \qquad 2\le t\le T.
\]
Indeed, since \(\rho_T\to0\), for sufficiently large \(T\)
we have \(2b_T<\kappa\rho_T/\sqrt T\) and
\(\rho_T/\sqrt T<1\). The induction bounds therefore ensure
that \(g_{\mathcal M_T}(\Delta_t^{\mathcal M_T})
=c_0T^{\alpha-1}\Delta_t^{\mathcal M_T}\),
so the next update satisfies
\(\Delta_{t+1}^{\mathcal M_T}
=(1-c_0T^{\alpha-1}\eta_t)\Delta_t^{\mathcal M_T}\).
For sufficiently large \(T\), each factor lies in \([1/2,1]\).
Together with
\(T^{\alpha-1}\sum_{t=2}^T\eta_t=O(1)\),
this gives a positive lower bound \(\kappa\), independent of
\(T\), for all partial products, closing the induction.
On \(E_T^c\), the iterates remain negative and,
because \(g_{\mathcal M_T}(x)=\lambda x\) for \(x<0\), satisfy
\(\Delta_{t+1}^{\mathcal M_T}
=(1-\lambda\eta_t)\Delta_t^{\mathcal M_T}\),
where \(\lambda\) is chosen so that \(\lambda\eta\le1/4\).

The resulting dynamics satisfy
Assumptions~\ref{ass:rate-linear-dynamics}--\ref{ass:rate-trajectory-stability}.
Summing the positive errors on \(E_T\) gives, for a suitable
constant \(c>0\),
\[
  \mathbb P_{\mathcal M_T}
  (\Phi_T^{\mathcal M_T}\in\Gamma_T)\ge\rho_T 
  \quad \text{with} \quad
    \Gamma_T\coloneqq\{x:d_{\mathrm S}(x,0)\ge c\rho_T\}.
\]
Since the noise vanishes after the first update, the limiting
variance is zero and the target is the zero path.
This path lies outside \(\Gamma_T^\delta\) whenever
\(\delta<c\rho_T\). For \(\delta<\min\{c,1\}\rho_T\),
the target probability plus the Prokhorov tolerance \(\delta\)
is therefore smaller than the SA probability.
This yields the lower bound of order \(\rho_T\).
The complete proof is given in
\suppappendixref{app:rate-nonlinear-oracle-lower}.


\end{proof}

\subsubsection{Putting all pieces together}
\label{sec:rate-overall-lower}

In previous sections, we discussed the matching lower bounds for the
four terms in \eqref{eq:linear-sa-fclt-rate} respectively.

Theorem~\ref{thm:rate-covariance-stabilization-lower} shows that cumulative
conditional variance does not determine the local behavior of individual
paths.  Exploiting this gap, the stopped-Brownian construction produces
the oracle class \(\mathfrak C_{\mathrm{cov}}\) and the covariance lower
bound.

Theorems~\ref{thm:rate-iid-leading-minimax} and
\ref{thm:linear-gaussian-j1-sharpness} instead examine local path
oscillations.  In \(\mathfrak C_{\mathrm{osc}}\), rare large noise inputs
make the normalized noise partial-sum path locally too rough.  In the
singleton Gaussian class \(\mathfrak C_{\mathrm{sm}}\), propagation through
the SA recursion makes the full SA path locally too smooth.  These
independent but opposite deviations both conflict with the characteristic
square-root local modulus of Brownian motion.

Finally, Theorem~\ref{thm:rate-nonlinear-oracle-lower} constructs
\(\mathfrak C_{\mathrm{nl}}\) so that, conditional on a rare first input,
the iterates evolve deterministically from the second iterate onward and
remain throughout the construction horizon in a region where the mean field
differs substantially from its linearization at the root.  The resulting
persistent deviation gives the nonlinear lower bound.  We now combine the
four oracle classes to obtain an overall lower bound matching the four
polynomial exponents in Theorem~\ref{thm:linear-fclt-rate}.
Let
\[
  \mathfrak C_{\mathrm{lb}}
  =\mathfrak C_{\mathrm{cov}}\cup\mathfrak C_{\mathrm{osc}}
    \cup\mathfrak C_{\mathrm{sm}}\cup\mathfrak C_{\mathrm{nl}}.
\]
Here \(\mathfrak C_{\mathrm{cov}}\) is understood to be empty when
\(\rateCov\ge\min\{\rateOsc,\rateSm\}\), and
\(\mathfrak C_{\mathrm{osc}}\) is understood to be empty when
\(\rateOsc\ge\rateSm\).  In the first case,
\(T^{-\rateCov}\le\max\{T^{-\rateOsc},T^{-\rateSm}\}\), where the
larger rate is furnished by \(\mathfrak C_{\mathrm{osc}}\) if
\(\rateOsc<\rateSm\) and by \(\mathfrak C_{\mathrm{sm}}\) otherwise.  In
the second case,
\(T^{-\rateOsc}\le T^{-\rateSm}\le
T^{-\rateSm}\sqrt{\log T}\) for all sufficiently large \(T\).
For nonempty classes \(\mathfrak C_1,\ldots,\mathfrak C_m\), definition~\eqref{eq:rate-worst-case-path-error} directly gives
\begin{equation}
  \mathcal D_T\!\left(\bigcup_{j=1}^m\mathfrak C_j\right)
  =\max_{1\le j\le m}\mathcal D_T(\mathfrak C_j).
  \label{eq:rate-union-error}
\end{equation}
Combining the four lower bounds through \eqref{eq:rate-union-error} gives the following theorem.

\begin{theorem}[Overall lower bound for the full SA path]
\label{thm:rate-overall-lower}
There exist constants \(c>0\) and \(T_0<\infty\), independent of \(T\),
such that, for every integer \(T\ge T_0\),
\begin{equation}
  \mathcal D_T(\mathfrak C_{\mathrm{lb}})
  \ge c\left\{
    T^{-\rateCov}
    +T^{-\rateOsc}
    +T^{-\rateSm}\sqrt{\log T}
    +T^{-\rateNL}
  \right\}.
  \label{eq:rate-overall-lower}
\end{equation}
\end{theorem}

\begin{proof}
By \eqref{eq:rate-union-error} and
Theorems~\ref{thm:rate-covariance-stabilization-lower}--
\ref{thm:rate-nonlinear-oracle-lower}, for all sufficiently large \(T\),
\[
\begin{aligned}
  \mathcal D_T(\mathfrak C_{\mathrm{lb}})
  &=
  \max_{\substack{
    \mathfrak C\in\{\mathfrak C_{\mathrm{cov}},
      \mathfrak C_{\mathrm{osc}},
      \mathfrak C_{\mathrm{sm}},
      \mathfrak C_{\mathrm{nl}}\}\\
    \mathfrak C\ne\varnothing}}
  \mathcal D_T(\mathfrak C)\\
  &\gtrsim
  \max\!\left\{
    T^{-\rateCov},
    T^{-\rateOsc},
    T^{-\rateSm}\sqrt{\log T},
    T^{-\rateNL}
  \right\}\\
  &\ge
  \frac14\left\{
    T^{-\rateCov}
    +T^{-\rateOsc}
    +T^{-\rateSm}\sqrt{\log T}
    +T^{-\rateNL}
  \right\}.
\end{aligned}
\]
The middle inequality also covers the nominal rate of an empty component class
by the preceding dominance relations.  This proves
\eqref{eq:rate-overall-lower}.
\end{proof}

The covariance and finite-moment bounds match the polynomial exponents in the stated regimes, while the Gaussian bound also matches the logarithmic factor.

\section{Proofs of the Upper Bounds}
\label{sec:rate-proof-overview}

The proof of Theorem~\ref{thm:linear-fclt-rate} separates two tasks. First, we compare the SA path with its leading noise partial sum by controlling four remainder paths. We then approximate the projected noise partial sum by Brownian motion. Probability-radius bounds connect these two tasks and convert the resulting couplings into the Prokhorov bound.

\subsection{Probability and coupling tools}
\label{sec:rate-probability-coupling-tools}


For a random error $Y$ in a normed path space, we seek a small $\varepsilon$ such that $\Pp(\tnorm{Y}>\varepsilon)\le\varepsilon$. This means that the error is at most $\varepsilon$ with probability at least $1-\varepsilon$. We summarize this control through the probability radius defined by
\[
  \widetilde d(Y)
  \coloneqq\inf_{\varepsilon\ge0}\left\{\varepsilon\vee\Pp(\tnorm{Y}>\varepsilon)\right\}.
\]

\begin{lemma}[Probability-radius calculus]
\label{lem:rate-probability-radius-calculus}
If $X$ and $Y$ are real-valued c\`adl\`ag processes defined on the same
probability space, then
\begin{equation}
  d_{\mathrm P}(X,Y)
  \le \widetilde d(X-Y),
  \label{eq:rate-coupling-prokhorov}
\end{equation}
where the Prokhorov distance is induced by the \(J_1\) Skorokhod metric in
Definition~\ref{def:rate-j1-metric}.  For random elements
$Y_1,\ldots,Y_M$ in the same normed path space, with fixed $M$,
\begin{equation}
  \widetilde d\!\left(\sum_{m=1}^M Y_m\right)
  \le \sum_{m=1}^M\widetilde d(Y_m)
  \label{eq:rate-ky-fan-subadditivity}
\end{equation}
and, for integers \(d_1,d_2\ge1\), every deterministic linear map
\(L:\R^{d_1}\to\R^{d_2}\), acting pointwise on an
\(\R^{d_1}\)-valued random path \(Y\), satisfies
\begin{equation}
  \widetilde d(LY)\le(\|L\|\vee1)\widetilde d(Y).
  \label{eq:rate-ky-fan-linear-map}
\end{equation}
\end{lemma}

\begin{proof}
The proof is given in
\suppappendixref{app:proof-rate-probability-radius-calculus}.
\end{proof}

The next inequality controls weighted noise sums. Its two terms retain the distinct effects of Gaussian fluctuations and finite-moment tails.
\begin{lemma}[Maximal martingale Fuk--Nagaev inequality]
\label{lem:rate-fuk-nagaev}
Let $(D_i,\cF_i)_{i=1}^n$ be a real-valued martingale-difference sequence,
let $p>2$, and suppose that there are deterministic constants
$V_n,L_{p,n}<\infty$ such that, almost surely,
\begin{equation*}
  \sum_{i=1}^n
  \E\!\left[D_i^2\mid\cF_{i-1}\right]\le V_n,
  \qquad
  \sum_{i=1}^n
  \E\!\left[|D_i|^p\mid\cF_{i-1}\right]\le L_{p,n}.
\end{equation*}
Then, for every $x>0$,
\begin{equation}
  \Pp\!\left(
    \max_{k\le n}\left|\sum_{i=1}^kD_i\right|>x
  \right)
  \le C_p\exp(-c_px^2/V_n)+C_pL_{p,n}x^{-p}.
  \label{eq:rate-fuk-nagaev}
\end{equation}
Here $C_p,c_p>0$ depend only on $p$, and the exponential term is
interpreted as zero when $V_n=0$.
\end{lemma}

\begin{proof}
This two-sided form follows from the one-sided maximal martingale inequality
in \cite[Corollary~2.5]{fan2017deviation}.  The short derivation is given in
\suppappendixref{app:proof-rate-fuk-nagaev}.
\end{proof}

\begin{lemma}[Two-regime tail-to-radius conversion]
\label{lem:rate-tail-to-radius}
Let $p>0$, $A_T\ge0$, $N_T\ge1$, and $0<s_T\le1$.  Suppose that,
for fixed constants $C_0,c_0>0$, a random path $Y_T$ satisfies
\begin{equation}
  \Pp(\tnorm{Y_T}>x)
  \le C_0A_Tx^{-p}+C_0N_T\exp(-c_0x^2/s_T^2),
  \qquad x>0.
  \label{eq:rate-two-regime-tail}
\end{equation}
Then
\begin{equation}
  \widetilde d(Y_T)
  \le C\left\{
    A_T^{1/(p+1)}
    +s_T\sqrt{\log(eN_T/s_T)}
  \right\},
  \label{eq:rate-tail-to-ky-fan}
\end{equation}
where $C$ depends only on $p,C_0,c_0$.
\end{lemma}

\begin{proof}
The proof is given in \suppappendixref{app:proof-rate-tail-to-radius}.
\end{proof}

For martingale terms whose available moment order is at most two, the following
bound replaces the two-regime inequality.  It will control the centered
fluctuations of the embedding-time increments in the Brownian representation.
\begin{lemma}[Subquadratic BDG maximal inequality]
\label{lem:rate-subquadratic-bdg}
Let $(D_j,\mathcal G_j)_{j=1}^n$ be a real martingale-difference sequence
such that \(\E[|D_j|^\nu]<\infty\) for some \(1<\nu\le2\).  Then
\begin{equation*}
  \E\!\left[\max_{k\le n}\left|\sum_{j=1}^kD_j\right|^\nu\right]
  \le C_\nu\E\!\left[\left(\sum_{j=1}^nD_j^2\right)^{\nu/2}\right]
  \le C_\nu\sum_{j=1}^n\E\!\left[|D_j|^\nu\right],
\end{equation*}
where \(C_\nu<\infty\) depends only on \(\nu\).
\end{lemma}

\begin{proof}
The first inequality is the discrete-time Burkholder--Davis--Gundy maximal
inequality. See \cite{burkholder1988sharp}.  The second follows from
\((\sum_j a_j)^{\nu/2}\le\sum_j a_j^{\nu/2}\) for \(a_j\ge0\).
\end{proof}

To approximate the leading noise partial sum, we represent it as Brownian
motion observed at stopping times.  The next embedding matches the conditional
mean of each embedding-time increment to the corresponding conditional variance.
\begin{lemma}[Scott--Huggins martingale embedding]
\label{lem:rate-scott-huggins-embedding}
Let $(Y_j,\cF_j)_{j\ge1}$ be a real square-integrable
martingale-difference sequence and put $S_k=\sum_{j=1}^kY_j$.  On a complete
extension of the probability space, there exist a Brownian motion
\((W(u),\mathcal H_u)_{u\ge0}\), increasing
\((\mathcal H_u)\)-stopping times
\(0=\tau_0\le\tau_1\le\cdots\), and an increasing family
\((\mathcal G_j)_{j\ge0}\) of \(\sigma\)-fields such that \(\mathcal F_j\subseteq\mathcal G_j\), each \(\tau_j\) is \(\mathcal G_j\)-measurable, and
\begin{equation}
  W(\tau_k)=S_k,
  \qquad
  \E\!\left[\rho_j\mid\mathcal G_{j-1}\right]
  =\E\!\left[Y_j^2\mid\cF_{j-1}\right],
  \qquad \rho_j\coloneqq\tau_j-\tau_{j-1}.
  \label{eq:rate-scott-huggins-time-change}
\end{equation}
In particular, \(\rho_j\) is \(\mathcal G_j\)-measurable.
Moreover, for every $r>1$ there are constants $0<c_r<C_r<\infty$,
depending only on $r$, such that
\begin{equation}
  c_r\E\!\left[\rho_j^{r/2}\mid\mathcal G_{j-1}\right]
  \le \E\!\left[|Y_j|^r\mid\cF_{j-1}\right]
  \le C_r\E\!\left[\rho_j^{r/2}\mid\mathcal G_{j-1}\right]
  \quad\text{almost surely}.
  \label{eq:rate-scott-huggins-moments}
\end{equation}
The embedding retains the original $\cF_j$ inside the enlarged stopped
filtration, so the conditional quantities in
\eqref{eq:rate-scott-huggins-time-change}--\eqref{eq:rate-scott-huggins-moments}
are the original ones.
\end{lemma}

\begin{proof}
The source and filtration identification are recorded in
\suppappendixref{app:proof-rate-scott-huggins-embedding}.
\end{proof}

Finally, we give a standard concentration property of Brownian motion.
\begin{lemma}[Brownian tail bound]
\label{lem:rate-brownian-modulus}
Let \(W\) be standard Brownian motion and let \(H\in(0,\infty)\).  There is
a constant \(C_H<\infty\), depending only on \(H\), such that, for every
\(0<h\le1\) and \(u>0\),
\begin{equation*}
  \Pp\!\left(
    \sup_{\substack{s,t\in[0,H]\\|s-t|\le h}}
    |W(t)-W(s)|>u
  \right)
  \le C_Hh^{-1}\exp\{-u^2/(C_Hh)\}.
\end{equation*}
\end{lemma}

\begin{proof}
The proof is given in \suppappendixref{app:proof-rate-brownian-modulus}.
\end{proof}

\subsection{Basic decomposition and term-wise analysis}
In the rest of the paper, we set
\begin{equation}
  \bB_t\coloneqq \bI-\eta_t\bG,
  \qquad
  \bA_j^n\coloneqq\eta_j\sum_{k=j}^n\prod_{i=j+1}^k\bB_i,
  \qquad n\ge j\ge1,
  \label{eq:rate-recursion-coefficients}
\end{equation}
where an empty product is \(\bI\).  Because every \(\bB_i\) is a polynomial in
the same matrix \(\bG\), the factors commute.

Retain $\bDelta_t=\bx_t-\bx^\star$ and set $\br_t=g(\bx_t)-\bG\bDelta_t$.
The update and the definition of \(\bepsilon_t\) in
Assumption~\ref{ass:rate-martingale-noise} give the exact recursion
\begin{equation*}
  \bDelta_{t+1}=\bB_t\bDelta_t+\eta_t\bepsilon_t-\eta_t\br_t.
\end{equation*}
By the definition in \eqref{eq:partial-sum-process},
\(\bPhi_T(n/T)=T^{-1/2}\sum_{t=1}^n\bDelta_{t+1}\).
Direct iteration gives
\begin{equation*}
  \bDelta_{t+1}=\left(\prod_{i=1}^t\bB_i\right)\bDelta_1
  +\sum_{j=1}^t\left(\prod_{i=j+1}^t\bB_i\right)
    \eta_j(\bepsilon_j-\br_j).
\end{equation*}
Summing over $1\le t\le n$, exchanging the finite sums, and using \eqref{eq:rate-recursion-coefficients} yields
\begin{equation*}
  \sum_{t=1}^n\bDelta_{t+1}
  =
  \frac{\bA_1^n(\bI-\eta_1\bG)\bDelta_1}{\eta_1}
  +\sum_{j=1}^n\bA_j^n\bepsilon_j
  -\sum_{j=1}^n\bA_j^n\br_j.
\end{equation*}
Consequently, at every grid point with $1\le n\le T$,
\begin{equation}
\begin{aligned}
  \bPhi_T(n/T)-\frac1{\sqrt T}\sum_{j=1}^n\bG^{-1}\bepsilon_j
  &=\underbrace{\frac{\bA_1^n(\bI-\eta_1\bG)\bDelta_1}
  {\sqrt T\eta_1}}_{
    \eqqcolon \bpsi_{0,T}(n)
  }+\underbrace{\left(-\frac1{\sqrt T}\sum_{j=1}^n\bA_j^n\br_j\right)}_{
    \eqqcolon\bpsi_{1,T}(n)}\\
    &\quad +\underbrace{
      \frac1{\sqrt T}\sum_{j=1}^n(\bA_j^T-\bG^{-1})\bepsilon_j
    }_{\eqqcolon \bpsi_{2,T}(n)}
    +\underbrace{\frac1{\sqrt T}\sum_{j=1}^n
  (\bA_j^n-\bA_j^T)\bepsilon_j}_{
    \eqqcolon \bpsi_{3,T}(n)}.
\end{aligned}
  \label{eq:rate-four-remainder-decomposition}
\end{equation}
Set $\bpsi_{i,T}(0)=0$ for $i\in\{0,1,2,3\}$, and henceforth use
$\bpsi_{i,T}$ for the path obtained by linear interpolation of
$\bpsi_{i,T}(n)$ placed at times $n/T$. In grid formulas, the integer argument $n$ denotes this discrete value. In path-space expressions, the argument is continuous time in $[0,1]$.  Convexity of the norm on
each grid interval gives the useful identity
\begin{equation}
  \tnorm{\bpsi_{i,T}}
  =\max_{0\le n\le T}\norm{\bpsi_{i,T}(n)},
  \qquad i\in\{0,1,2,3\}.
  \label{eq:intpl_to_grid}
\end{equation}
The terms $\bpsi_{0,T}$ and $\bpsi_{1,T}$ describe initialization and nonlinearity. The two remaining paths split the algorithmic smoothing error: fixing the endpoint at $T$ makes $\bpsi_{2,T}$ a martingale partial sum, while $\bpsi_{3,T}$ contains the varying-endpoint correction. We control these four paths in that order.

\subsubsection*{Step 1: Initialization bias}

The term \(\bpsi_{0,T}\) is the initialization bias.  The next two
deterministic lemmas provide the step-size comparisons and the uniform
coefficient control used throughout the remainder analysis.
\begin{lemma}[Power-law contraction-sum estimates]
\label{lem:rate-power-law-contraction-sums}
Let $\eta_t=\eta t^{-\alpha}$ with $\eta>0$ and $\alpha\in(0,1)$.  For
every $a,b>0$, there are constants $C_a<\infty$ and $j_b<\infty$ such that
\begin{equation}
  \sum_{k=j}^{\infty}
  \exp\left(-a\sum_{i=j+1}^k\eta_i\right)
  \le C_a\eta_j^{-1},
  \qquad
  \sup_{n\ge1}\sum_{j=1}^n\eta_j
  \exp\left(-a\sum_{i=j+1}^n\eta_i\right)
  \le C_a.
  \label{eq:rate-contraction-sums}
\end{equation}
Moreover,
\begin{equation}
  \frac{\eta_j}{\eta_k}
  \le \exp\left\{b\sum_{i=j+1}^k\eta_i\right\},
  \qquad k\ge j\ge j_b.
  \label{eq:linear-rate-gain-ratio}
\end{equation}
\end{lemma}

\begin{proof}
The proof is given in
\suppappendixref{app:proof-rate-power-law-contraction-sums}.
\end{proof}

\begin{lemma}[Deterministic power-law step-size bounds]
\label{lem:rate-deterministic-products}
Suppose \(-\bG\) is Hurwitz and
\(\eta_t=\eta t^{-\alpha}\) with \(\eta>0\) and \(\alpha\in(0,1)\).
For \(\bB_t\) and \(\bA_j^n\) defined in
\eqref{eq:rate-recursion-coefficients}, there are constants \(C,c>0\)
such that, for all \(n\ge j\ge1\),
\begin{equation}
  \left\|\prod_{i=j}^n\bB_i\right\|
  \le C\exp\left(-c\sum_{i=j}^n\eta_i\right),
  \qquad
  \sup_{n\ge j\ge1}\|\bA_j^n\|\le C.
  \label{eq:rate-deterministic-product-bounds}
\end{equation}
\end{lemma}

\begin{proof}
The proof is given in
\suppappendixref{app:proof-rate-deterministic-products}.
\end{proof}

The uniform coefficient bound in Lemma~\ref{lem:rate-deterministic-products} gives
\begin{equation}
  \max_{n\le T}\norm{\bpsi_{0,T}(n)}\le CT^{-1/2}.
  \label{eq:linear-rate-R0-bound}
\end{equation}

\subsubsection*{Step 2: Nonlinear remainder}

For the remainder \(\br_j=g(\bx_j)-\bG\bDelta_j\) defined above, we first bound its moment and then sum its contribution along the path. The local quadratic and global boundedness conditions in Assumption~\ref{ass:rate-linear-dynamics} give
\[
  \norm{\br_j}
  \le \max\{C_{\mathrm{sm}},C_{\mathrm{bd}}\delta_g^{-2}\}
  \norm{\bDelta_j}^2.
\]
Indeed, the local quadratic bound applies when
\(\norm{\bDelta_j}\le\delta_g\), while otherwise the global bound is at
most \(C_{\mathrm{bd}}\delta_g^{-2}\norm{\bDelta_j}^2\).  Hence
\eqref{eq:rate-trajectory-stability} implies
\begin{equation*}
  \|\br_j\|_{L^q}
  \le C
    \bigl(\E\!\left[\norm{\bDelta_j}^{2q}\right]\bigr)^{1/q}
  =C\|\bDelta_j\|_{L^{2q}}^2
  \le C\eta_j.
\end{equation*}
Minkowski's inequality and the uniform coefficient bound in Lemma~\ref{lem:rate-deterministic-products} then give
\begin{align*}
  \left\|\tnorm{\bpsi_{1,T}}\right\|_{L^q}
  &\le \frac C{\sqrt T}\sum_{j=1}^T\|\br_j\|_{L^q}
  \le \frac {C}{\sqrt T}\sum_{j=1}^T\eta_j
  \le CT^{-(\alpha-1/2)}.
\end{align*}
The last inequality follows from \(\eta_j=\eta j^{-\alpha}\).  Markov's inequality
therefore yields
\begin{equation*}
  \Pp(\tnorm{\bpsi_{1,T}}>x)
  \le CT^{-q(\alpha-1/2)}x^{-q}.
\end{equation*}
Taking \(x=K T^{-(\alpha-1/2)q/(q+1)}\), with \(K\) sufficiently large,
makes the right-hand side at most \(x\). Plugging this
result into the definition of the probability radius \(\widetilde{d}(\cdot)\)
gives
\begin{equation}
  \widetilde d(\bpsi_{1,T})
  \le C T^{-\frac{(\alpha-1/2)q}{q+1}}.
  \label{eq:rate-nonlinear-remainder-radius}
\end{equation}

\subsubsection*{Step 3: Fixed-endpoint coefficient remainder}


We next control the fixed-endpoint martingale $\bpsi_{2,T}$.  Let
$\mathbf{e}_r$ be the $r$th coordinate vector.  For fixed $r\le d$, set
\begin{equation*}
  D_{j,r}=T^{-1/2}(\mathbf{e}_r)^\top
  (\bA_j^T-\bG^{-1})\bepsilon_j.
\end{equation*}
The sequence $(D_{j,r},\cF_j)$ is a martingale-difference sequence.
Assumption~\ref{ass:rate-martingale-noise} gives, almost surely,
\begin{equation}
  \sum_{j=1}^T
  \E\!\left[|D_{j,r}|^s\mid\cF_{j-1}\right]
  \le \frac C{T^{s/2}}\sum_{j=1}^T
  \norm{\bA_j^T-\bG^{-1}}^s,
  \qquad s\in\{2,p\}.
  \label{eq:bd_E_D_jr_mmt}
\end{equation}

The coefficient bounds below quantify two effects: the slow change of the step size, and the truncated response to noise arriving near the terminal time. Their summed bounds provide the variance and moment budgets required by the martingale inequality.
\begin{lemma}[Quantitative recursion coefficients]
\label{lem:linear-rate-coefficients}
Suppose \(\bG\) is positive stable and
\(\eta_t=\eta t^{-\alpha}\) with \(\alpha\in(1/2,1)\).
There exist constants \(C,c>0\) such that, for all integers
\(T\ge j\ge1\),
\begin{equation*}
  \|\bA_j^T-\bG^{-1}\|
  \le C\left\{
    j^{\alpha-1}
    +\exp\left(-c\sum_{i=j+1}^{T+1}\eta_i\right)
  \right\}.
\end{equation*}
Moreover, for every real \(s\ge2\), there exists a constant
\(C_s<\infty\) such that, for every integer \(T\ge1\),
\begin{equation*}
  \sum_{j=1}^T\|\bA_j^T-\bG^{-1}\|^s
  \le C_sT^\alpha.
\end{equation*}
\end{lemma}

\begin{proof}
The proof is given in \suppappendixref{app:proof-linear-rate-coefficients}.
\end{proof}

Applying Lemma~\ref{lem:linear-rate-coefficients} to
\eqref{eq:bd_E_D_jr_mmt} yields, almost surely,
\begin{align*}
  \sum_{j=1}^T
  \E\!\left[D_{j,r}^2\mid\cF_{j-1}\right]
  \le
   CT^{\alpha-1},\qquad
  \sum_{j=1}^T
  \E\!\left[|D_{j,r}|^p\mid\cF_{j-1}\right]
  \le CT^{\alpha-p/2}.
\end{align*}

Using the grid-to-path identity \eqref{eq:intpl_to_grid}, apply Lemma~\ref{lem:rate-fuk-nagaev} coordinatewise. The norm equivalence
  \(\|\bz\|\le\sqrt d\max_{r\le d}|\mathbf e_r^\top \bz|\) and a union bound over
the fixed dimension give
\begin{equation*}
  \Pp(\tnorm{\bpsi_{2,T}} > x)
  \le C\exp(-cT^{1-\alpha}x^2)
  +CT^{\alpha-p/2}x^{-p}.
\end{equation*}
The two terms satisfy Lemma~\ref{lem:rate-tail-to-radius} with $A_T=T^{\alpha-p/2}$, $N_T=1$, and $s_T=T^{-(1-\alpha)/2}$.
This implies
\begin{equation}
  \widetilde d(\bpsi_{2,T})
  \le C\left\{
    T^{-(1-\alpha)/2}\sqrt{\log T}
    +T^{-(p/2-\alpha)/(p+1)}
  \right\}.
  \label{eq:linear-rate-R2-ky-fan}
\end{equation}

\subsubsection*{Step 4: Endpoint-varying coefficient remainder}

For the endpoint-varying term, define
\begin{equation*}
  \by_{n+1}=\sum_{j=1}^n
  \left(\prod_{i=j+1}^n\bB_i\right)\eta_j\bepsilon_j,
  \quad 0\le n \le T.
\end{equation*}
Expanding \(\bA_j^n-\bA_j^T\) using \eqref{eq:rate-recursion-coefficients} and factoring the commuting products at time $n$ gives
\begin{align*}
  \bpsi_{3,T}(n)
  &=-\frac1{\sqrt T}
  \sum_{j=1}^n\eta_j\sum_{k=n+1}^T
  \left(\prod_{i=j+1}^k\bB_i\right)\bepsilon_j\\
  &=-\frac1{\sqrt T}\left(
    \sum_{k=n+1}^T\prod_{i=n+1}^k\bB_i
  \right)\by_{n+1}.
\end{align*}
The stable-product bound in Lemma~\ref{lem:rate-deterministic-products} and
the forward contraction-sum estimate
\eqref{eq:rate-contraction-sums} imply
\begin{equation*}
  \left\|\sum_{k=n+1}^T\prod_{i=n+1}^k\bB_i\right\|
  \le C\sum_{k=n+1}^T
  \exp\left(-c\sum_{i=n+1}^k\eta_i\right)
  \le C\eta_{n+1}^{-1}.
\end{equation*}
Consequently,
\begin{equation*}
  \max_{n\le T}\norm{\bpsi_{3,T}(n)}
  \le C\max_{n\le T}
  \frac{\norm{\by_{n+1}}}{\sqrt T\eta_{n+1}}.
\end{equation*}

The next lemma controls this maximum using blocks of length comparable to $t^\alpha$, the memory scale of the recursion. It combines contraction between blocks with a martingale maximal bound within each block.
\begin{lemma}[Adaptive-block maximal inequality]
\label{lem:linear-rate-filter}
Let \(p>2\), and consider
\begin{equation*}
  \by_{t+1}=(\bI-\eta_t\bG)\by_t+\eta_t\bepsilon_t,
  \qquad \by_1=0,
\end{equation*}
where \((\bepsilon_t,\cF_t)\) is a fixed-dimensional martingale-difference
sequence.  Suppose that, for some deterministic $M_\epsilon<\infty$ and
every $t\ge1$,
\begin{equation}
  \E\!\left[\|\bepsilon_t\|^p\mid\cF_{t-1}\right]
  \le M_\epsilon
  \quad\text{almost surely}.
  \label{eq:linear-rate-filter-moment}
\end{equation}
If \(\bG\) is positive stable and
\(\eta_t=\eta t^{-\alpha}\) with \(\alpha\in(1/2,1)\), then there are
\(C,c>0\) such that, for every integer \(T\ge2\) and \(x>0\),
\begin{equation}
\begin{split}
  \Pp\!\left(
    \max_{t\le T}
    \frac{\|\by_{t+1}\|}{\sqrt T\,\eta_{t+1}}>x
  \right)
  \le{}&CT^{1-p/2}x^{-p}
  +CT^{1-\alpha}
  \exp\{-cT^{1-\alpha}x^2\}.
\end{split}
  \label{eq:linear-rate-filter-tail}
\end{equation}
\end{lemma}

\begin{proof}
The proof is in \suppappendixref{app:proof-linear-rate-filter}.
\end{proof}

Define the scalar path \(\mathcal R_T\) by
\begin{equation*}
  \mathcal R_T(n/T)\coloneqq\frac{\|\by_{n+1}\|}{\sqrt T\,\eta_{n+1}},
  \qquad n=0,\ldots,T,
\end{equation*}
and interpolate linearly between consecutive grid points. Then $\tnorm{\mathcal R_T}=\max_{n\le T}\mathcal R_T(n/T)$, and the preceding grid bound gives $\tnorm{\bpsi_{3,T}}\le C\tnorm{\mathcal R_T}$ by \eqref{eq:intpl_to_grid}. Applying Lemma~\ref{lem:rate-tail-to-radius} to \eqref{eq:linear-rate-filter-tail}, with $A_T=T^{1-p/2}$, $N_T=T^{1-\alpha}$, and $s_T=T^{-(1-\alpha)/2}$, yields
\begin{equation}
  \widetilde d(\bpsi_{3,T})
  \le C \widetilde d(\mathcal R_T)
  \le C^2\left\{
    T^{-\frac{p-2}{2(p+1)}}
    +T^{-(1-\alpha)/2}\sqrt{\log T}
  \right\}.
  \label{eq:linear-rate-R3-ky-fan}
\end{equation}

\subsection{The leading martingale process}

The remaining task is Gaussian approximation of the leading noise path.  The
next lemma combines a Brownian embedding with quantitative control of the
embedding times.  Their deviation from the deterministic linear time scale has
two sources: the discrepancy of the cumulative conditional covariance from its
deterministic linear profile and the centered fluctuations of the embedding-time
increments.
Assumption~\ref{ass:rate-covariance-stabilization} controls the former, while
the moment assumption controls the latter and the large increments within the
interpolation grid.

\begin{lemma}[Normalized martingale functional approximation]
\label{lem:rate-leading-mds-coupling}
Let \((Y_t,\cF_t)_{t\ge1}\) be real-valued martingale differences.  Suppose
that, for some deterministic $M_Y<\infty$, some $p>2$, and every $t\ge1$,
\begin{equation}
  \E\!\left[|Y_t|^p\mid\cF_{t-1}\right]\le M_Y
  \quad\text{almost surely}.
  \label{eq:rate-leading-mds-moment}
\end{equation}
Let \(\sigma^2\ge0\) and suppose that, for some \(K<\infty\), some
\(\gamma>0\) satisfying $\gamma<1/4$ and $\gamma\le(p-2)/\{2(p+1)\}$, and every \(T\ge2\),
\begin{equation}
  \E\!\left[\max_{0\le n\le T}
  \left|
    \frac1T\sum_{j=1}^n\E\!\left[Y_j^2\mid\cF_{j-1}\right]
    -\frac nT\sigma^2
  \right|
  \right]
  \le KT^{-3\gamma}.
  \label{eq:rate-leading-scalar-covariance}
\end{equation}
For each \(T\), let \(\mathcal S_T\) be the random continuous path obtained by
affine interpolation of
$\mathcal S_T(n/T)=T^{-1/2}\sum_{j=1}^nY_j$, $0\le n\le T$.
Then there is \(C<\infty\), independent of \(T\), such that
\begin{equation}
  d_{\mathrm P}(\mathcal S_T,\sigma W_\star^{(1)})
  \le C\sqrt{\log T}\,T^{-\gamma}.
  \label{eq:rate-leading-mds-coupling}
\end{equation}
\end{lemma}

\begin{proof}
The proof is given in \suppappendixref{app:proof-rate-leading-mds-coupling}.
\end{proof}

Recall that $\bmM_T$ affinely interpolates $\bmM_T(n/T)=T^{-1/2}\sum_{j=1}^n\bepsilon_j$ for $0\le n\le T$.
For a fixed test vector \(\btheta\in\mathbb R^d\), set
\(Y_j^{\btheta}\coloneqq\btheta^\top \bG^{-1}\bepsilon_j\),
\(M_T^{\btheta}\coloneqq\btheta^\top \bG^{-1}\bmM_T\), and
\(\sigma_{\btheta}^2\coloneqq
\btheta^\top \bG^{-1}\bS\bG^{-\top}\btheta\), where
\(\sigma_{\btheta}\ge0\).
The conditional \(p\)th-moment bound in
\eqref{eq:rate-leading-mds-moment} follows immediately from
\eqref{eq:linear-sa-rate-moments}.  Moreover,
\begin{align*}
 &\E\!\left[\max_{0\le n\le T}
 \left|\frac1T\sum_{j=1}^n
   \E\!\left[(Y_j^{\btheta})^2\mid\cF_{j-1}\right]
   -\frac nT\sigma_{\btheta}^2\right|\right]\\
 &\quad\le
 \|\bG^{-\top}\btheta\|^2
 \E\!\left[\max_{0\le n\le T}
 \left\|\frac1T\sum_{j=1}^n
   \E\!\left[\bepsilon_j\bepsilon_j^\top\mid\cF_{j-1}\right]
   -\frac nT \bS\right\|\right]
 \le CT^{-3\gamma}.
\end{align*}
Consequently, for every $\widetilde\gamma>0$ with $\widetilde\gamma\le\gamma$, $\widetilde\gamma<1/4$, and $\widetilde\gamma\le(p-2)/\{2(p+1)\}$, Lemma~\ref{lem:rate-leading-mds-coupling} applied to $(Y_j^{\btheta},\cF_j)$ gives
\begin{equation}
  d_{\rm P}(M_T^{\btheta},\sigma_{\btheta}W_\star^{(1)})
  \le CT^{-\widetilde\gamma}\sqrt{\log T}.
  \label{eq:mds_coupling_rate}
\end{equation}
By the definition of \(\bZ_\star\), the process \(\btheta^\top\bZ_\star\) has the
same law as \(\sigma_{\btheta} W_\star^{(1)}\).

\subsection{Completion of the theorem proofs}

\begin{proof}[Proof of Theorem~\ref{thm:linear-fclt-rate}]
Put $\overline\gamma\coloneqq\min\{\gamma,(p-2)/\{2(p+1)\},(1-\alpha)/2\}$.
Since \((1-\alpha)/2<1/4\), we have
\(0<\overline\gamma<1/4\) and
\(\overline\gamma\le(p-2)/\{2(p+1)\}\).  Moreover, the
covariance-stabilization bound at exponent \(\gamma\) remains valid, with the
same constant, at exponent
\(\overline\gamma\).  Thus \eqref{eq:mds_coupling_rate} applies with
\(\widetilde\gamma=\overline\gamma\).
Because $\bPhi_T$, $\bmM_T$, and the four remainder paths are affine on the
same grid, linear interpolation of
\eqref{eq:rate-four-remainder-decomposition} gives the exact
pathwise identity
\begin{equation}
  \bPhi_T-\bG^{-1}\bmM_T
  =\bpsi_{0,T}+\bpsi_{1,T}+\bpsi_{2,T}+\bpsi_{3,T}
  \quad\text{on }[0,1].
  \label{eq:rate-four-remainder-pathwise}
\end{equation}
Combining Lemma~\ref{lem:rate-probability-radius-calculus} with
\eqref{eq:linear-rate-R0-bound},
\eqref{eq:rate-nonlinear-remainder-radius},
\eqref{eq:linear-rate-R2-ky-fan},
\eqref{eq:linear-rate-R3-ky-fan},
\eqref{eq:mds_coupling_rate}, and
\eqref{eq:rate-four-remainder-pathwise} gives
\begin{align*}
  d_{\mathrm P}\!\left(\btheta^\top\bPhi_T,\btheta^\top\bZ_\star\right)
  &\le
  \widetilde d\!\left(
    \btheta^\top(\bPhi_T-\bG^{-1}\bmM_T)
  \right)
  +d_{\mathrm P}\!\left(M_T^{\btheta},\sigma_{\btheta} W_\star^{(1)}\right)\\
  &\le
  (\|\btheta\|\vee1)\sum_{i=0}^3\widetilde d(\bpsi_{i,T})
  +CT^{-\overline\gamma}\sqrt{\log T}\\
  &\le C\bigg\{
    T^{-\overline\gamma}\sqrt{\log T}
    +T^{-\frac{p-2}{2(p+1)}}
    +T^{-\frac{1-\alpha}{2}}\sqrt{\log T}
    +T^{-\frac{(\alpha-1/2)q}{q+1}}
  \bigg\}.
\end{align*}
Since
\(
  T^{-\overline\gamma}
  \le
  T^{-\gamma}
  +T^{-\frac{p-2}{2(p+1)}}
  +T^{-(1-\alpha)/2}
\),
and a constant multiple of \(\sqrt{\log T}\) is at least one for
\(T\ge2\), this estimate implies \eqref{eq:linear-sa-fclt-rate} while keeping
the conditional-covariance-stabilization and finite-moment scales explicit.
If \(g(\bx_t)=\bG(\bx_t-\bx^\star)\) almost surely for every \(t\), then
\(\br_t=0\) and hence
\(\bpsi_{1,T}=0\) in \eqref{eq:rate-four-remainder-pathwise}.  Repeating the
same bounds without \(\widetilde d(\bpsi_{1,T})\) removes the nonlinear term in
\eqref{eq:linear-sa-fclt-rate}.
\end{proof}

\section{Related Work}
\label{sec:rate-related}

\subsection*{Finite-time analysis and Gaussian approximation}

Finite-time analyses of SA quantify iterate errors through moment bounds.
Examples include mean-square bounds for SGD under strong convexity
\cite{moulines2011non} and for contractive SA \cite{chen2020finite}. For
contractive SA with Markovian noise, Chen et al.~\cite{chen2021lyapunov}
develop a Lyapunov framework for finite-sample analysis. For nonlinear
two-time-scale SA, Han et al.~\cite{han2026finite} establish decoupled
mean-square convergence rates under nested local linearity and suitable
step-size conditions, using fourth-moment estimates to control nonlinear
errors. These results quantify iterate error, while a complementary question
concerns the accuracy of Gaussian approximations to their distributions.

For Polyak--Ruppert averaged SGD, Anastasiou et
al.~\cite{anastasiou2019normal} develop smooth-test bounds, while Shao and
Zhang~\cite{shao2022berry} obtain convex-set Berry--Esseen bounds. For i.i.d.\
linear SA, Samsonov et al.~\cite{samsonov2024gaussian} obtain convex-set bounds
with best polynomial order \(T^{-1/4}\) over their step-size family, while
Butyrin et al.~\cite{butyrin2025improved} reach \(T^{-1/3}\) under refined
assumptions. Further Gaussian approximation results cover nonlinear averaged
SGD \cite{sheshukova2026gaussian}, Markovian linear SA
\cite{samsonov2025markov}, temporal-difference learning
\cite{srikant2024rates,wuliweirinaldo2026td}, and two-time-scale SA
\cite{kong2025twotimescale,butyrin2026twotimescale}. Recent Wasserstein
analyses also treat nonlinear SA \cite{kong2026wasserstein}, asynchronous
averaged Q-learning \cite{liu2026asynchronous}, and the last SA iterate
\cite{haque2026accurately}. These quantitative distributional bounds concern
endpoint laws, whereas our criterion compares the laws of entire projected
partial-sum paths.

\subsection*{Functional limits and statistical inference}

Functional central limit theorems for Polyak--Ruppert-type partial-sum
processes support random-scaling inference in online SGD \cite{lee2021fast}, Local SGD
\cite{li2021statistical}, and SA with dependent data
\cite{li2025stream}. A partial-sum FCLT is also established for
asynchronous averaged Q-learning \cite{liu2026asynchronous}. For constant-step
SA, Xie and Zhang~\cite{xie2022statistical} obtain a related FCLT with partial
sums centered at the stationary mean, which may differ from the target
solution. These functional limits justify path-based inference asymptotically,
but do not themselves quantify the approximation error at a finite horizon.

A distinct line of work studies continuous-time paths formed from rescaled
iterate errors. Classical diffusion limits for SA
\cite{kushner2003stochastic} have recently been developed further for SGD
\cite{flamand2026functional}, Local SGD \cite{liang2023asymptotic}, and
two-time-scale SA\@. In particular, Han et
al.~\cite{han2026decoupled}
rescale the fast and slow errors on their respective time scales and establish
convergence to stationary Ornstein--Uhlenbeck processes. These results
describe local iterate fluctuations. Their path constructions and diffusion
targets differ from the cumulative-error path and Brownian target considered
here.

\subsection*{Quantitative trajectory approximations}

\ifdefined\isarxivpreprint
For quantitative trajectory approximation, the result most closely related to
the present work is a preliminary upper bound from our earlier
work~\cite[Theorem~5]{li2023nonlinearv2}.
As detailed following Theorem~\ref{thm:linear-fclt-rate}, the present work
substantially improves the polynomial dependencies associated with
finite-moment noise, algorithmic smoothing, and nonlinearity.  The preliminary
rate result is not part of the current version of the earlier
preprint~\cite{li2023nonlinear}.

Apart
from this earlier preprint,
quantitative results on rates of weak convergence in functional central
limit theorems for SA are scarce. One exception is
\cite{wang2026scaling}, which provides finite-sample path-space approximation
bounds for the rescaled iterates of univariate constant-step SGD and SGLD,
with an Ornstein--Uhlenbeck process as the approximating target.
\else
Finite-horizon functional approximation bounds are available for rescaled
raw-iterate paths of univariate constant-step SGD and SGLD with
Ornstein--Uhlenbeck limits \cite{wang2026scaling}.
\fi
Under the moment assumptions of Corollary~6 therein and bounded model
constants, the pure-SGD specialization of the resulting
L\'evy--Prokhorov bound, after identifying the number of updates as
\(T\asymp h^{-1}\), has a polynomial rate approaching \(T^{-1/20}\)
as the available moment order \(p\to\infty\).
To the best of our knowledge,
no previous work has established an explicit finite-horizon
Prokhorov-distance rate in a Brownian invariance principle for each fixed
scalar projection of the full partial-sum path considered here in nonlinear SA
with decreasing step sizes;
corresponding path-space lower bounds are likewise unavailable.

Broadening the scope beyond SA, the two noise-driven terms in our main bound,
\(T^{-\rateCov}\) and \(T^{-\rateOsc}\), are closely related to a substantial
literature on approximation rates for noise partial-sum paths.  For independent
noise satisfying Assumption~\ref{ass:rate-martingale-noise}, the uniform-metric
Prokhorov bound of order \(T^{-\rateOsc}\) was first established for
\(2<p\le3\) in \cite{borovkov1974rate}.  The independent-variable theory was
later extended to all \(p>2\) in \cite{sakhanenko2006estimates}, with
corresponding rates for special martingale-difference arrays and related bounds
for semimartingales given in \cite{haeusler1984exact,kubilius1985rate}.

Matching lower bounds for both of these noise-driven terms are also available.
Under the same noise requirement
as in Theorem~\ref{thm:rate-iid-leading-minimax},
a lower bound of order \(T^{-\rateOsc}\) for the noise partial-sum path was
established in \cite{arak1976estimate}.  Theorem~\ref{thm:rate-iid-leading-minimax}
adapts this construction to the SA partial-sum path.  For the
\(T^{-\rateCov}\) term, a similar lower bound for the Kolmogorov distance in
the terminal martingale CLT was obtained in \cite{mourrat2013rate}, although
not a uniform-in-horizon Prokhorov lower bound for the entire path.

\section{Conclusion}
\label{sec:rate-conclusion}

In this work, we establish a quantitative functional central limit theorem for each fixed scalar projection of the normalized partial-sum path of nonlinear stochastic approximation. Under the Prokhorov distance induced by the specified \(J_1\) Skorokhod metric, the finite-time upper bound separates four sources of approximation error: conditional-covariance stabilization, finite-moment noise, algorithmic smoothing, and the nonlinear remainder. This decomposition shows how the noise, the recursion, and the step-size exponent jointly determine the accuracy of Brownian approximation at the path level. For the prescribed recursion, four lower-bound constructions recover the corresponding polynomial exponents over their respective oracle classes. The scalar Gaussian example further shows that the logarithmic factor in the smoothing term is unavoidable.

Several directions remain open. First, the upper bound applies to each fixed scalar projection, with constants whose dependence on dimension is not tracked. Extending the result to vector-valued paths or uniformly over projection directions, while determining the optimal dependence on \(d\), would broaden its use in high-dimensional settings. Second, convergence of the path law does not by itself give an informative relative approximation for rare-event probabilities. When an event has probability smaller than the approximation error, the bound may provide little information about its likelihood. Deviation estimates for SA trajectories and their Polyak--Ruppert-averaged paths would therefore complement the present results.
Finally, the smoothing term \(T^{-\rateSm}\sqrt{\log T}\) arises from the SA recursion itself, even in a scalar linear model with Gaussian noise. It remains to determine whether other SA-type algorithms can reduce or eliminate this contribution. More broadly, for a fixed stochastic oracle class, what is the optimal rate of Brownian approximation for normalized partial-sum paths when one is free to choose among admissible SA-type algorithms? We leave these questions for future work.

\bibliographystyle{plainnat}
\bibliography{bib/stat}

\clearpage
\appendix
\begin{center}
  {\Large\bfseries Supplementary\par}
\end{center}
\appendixtableofcontents
\bigskip
\noindent This supplement contains proofs of the auxiliary lemmas and
lower-bound results used in the manuscript. All notation and the numbering of
equations, theorems, and references follow the manuscript.
\medskip
\begingroup
  \redirectappendixcontents
  \section{Proofs of Auxiliary Lemmas}
\label{app:linear-fclt-rate}


\subsection{Probability-radius calculus}
\label{app:proof-rate-probability-radius-calculus}

\begin{proof}[Proof of Lemma~\ref{lem:rate-probability-radius-calculus}]
We first record a consequence of the definition of \(\widetilde d\).  If
\(a>\widetilde d(Y)\), then there is \(b<a\) such that
\(b\vee\Pp(\tnorm{Y}>b)<a\).  Monotonicity of the tail probability therefore
gives
\begin{equation*}
  \Pp(\tnorm{Y}>a)\le\Pp(\tnorm{Y}>b)<a.
\end{equation*}

By \eqref{eq:j1-bounded-by-uniform}, the identity time change gives
\(d_{\mathrm S}(X,Y)\le\tnorm{X-Y}\).  Apply the preceding observation to
\(X-Y\), then use the coupling characterization of Prokhorov distance and let
\(a\downarrow\widetilde d(X-Y)\).  This proves
\eqref{eq:rate-coupling-prokhorov}.

For each \(m\), choose \(a_m>\widetilde d(Y_m)\).  The triangle inequality
and a union bound give
\begin{equation*}
  \Pp\!\left(
    \tnorm{\sum_{m=1}^M Y_m}>\sum_{m=1}^M a_m
  \right)
  \le\sum_{m=1}^M\Pp(\tnorm{Y_m}>a_m)
  \le\sum_{m=1}^M a_m.
\end{equation*}
Letting every \(a_m\downarrow\widetilde d(Y_m)\) proves
\eqref{eq:rate-ky-fan-subadditivity}.  Finally, put
\(a=\|L\|\vee1\).  For every \(x>\widetilde d(Y)\),
\begin{equation*}
  \Pp(\tnorm{LY}>ax)
  \le\Pp(\tnorm{Y}>x)
  \le x\le ax.
\end{equation*}
Letting \(x\downarrow\widetilde d(Y)\) proves
\eqref{eq:rate-ky-fan-linear-map}.
\end{proof}

\subsection{Maximal martingale Fuk--Nagaev inequality}
\label{app:proof-rate-fuk-nagaev}

\begin{proof}[Proof of Lemma~\ref{lem:rate-fuk-nagaev}]
Corollary~2.5 of \cite{fan2017deviation}, with its predictable variance and
conditional $p$th-moment quantities bounded by $V_n$ and $L_{p,n}$,
respectively, gives the one-sided estimate
\[
  \Pp\!\left(\max_{k\le n}\sum_{i=1}^kD_i>x\right)
  \le C_p\exp(-c_px^2/V_n)+C_pL_{p,n}x^{-p}.
\]
The same estimate applies to the martingale differences $-D_i$, with the
same $V_n$ and $L_{p,n}$.  A union bound over the two signs proves
\eqref{eq:rate-fuk-nagaev}, after absorbing the factor two into $C_p$.
If $V_n=0$, then every $D_i=0$ almost surely, so the stated convention is
consistent.
\end{proof}

\subsection{Two-regime tail-to-radius conversion}
\label{app:proof-rate-tail-to-radius}

\begin{proof}[Proof of Lemma~\ref{lem:rate-tail-to-radius}]
Set \(x = K\big(
  A_T^{1/(p+1)} + s_T\sqrt{\log(eN_T/s_T)}
\big)\) in \eqref{eq:rate-two-regime-tail},
where \(K\ge1\) is a constant chosen below.  If \(x\ge1\), the conclusion
is immediate because \(\widetilde d(Y_T)\le1\).  Suppose therefore that
\(x<1\).  Since \(x\ge KA_T^{1/(p+1)}\), \eqref{eq:rate-two-regime-tail}
gives
\begin{align*}
  \mathbb P(\tnorm{Y_T} > x)
  &\le C_0 A_Tx^{-p} + C_0 N_T \exp\{-c_0x^2/s_T^2\}\\
  &\le C_0 K^{-(p+1)}x + C_0 N_T (e N_T/s_T)^{-c_0K^2}\\
  &= \frac{C_0}{K^{p+1}}x
   + C_0e^{-c_0K^2}N_T^{1- c_0K^2}s_T^{c_0K^2}
  \le x,
\end{align*}
where the last inequality holds by taking \(K\) large enough and 
invoking \(s_T \le 1\) and \(\log (eN_T/ s_T) \ge 1\).
Hence \eqref{eq:rate-two-regime-tail} gives
\(\Pp(\tnorm{Y_T}>x)\le x\), which proves
\eqref{eq:rate-tail-to-ky-fan}.
\end{proof}

\subsection{Scott--Huggins martingale embedding}
\label{app:proof-rate-scott-huggins-embedding}

\begin{proof}[Source of Lemma~\ref{lem:rate-scott-huggins-embedding}]
This is Theorem~1 and equations~(2.1)--(2.2) of
\cite[pp.~447--450]{scott1983embedding}.  That theorem supplies both the
continuous Brownian filtration and an increasing discrete family
\((\mathcal G_j)\) for which
\(\mathcal F_j\subseteq\mathcal G_j\) and \(\tau_j\) is
\(\mathcal G_j\)-measurable.  Completing and enlarging the probability space
is part of the coupling construction and leaves the law of \((S_k)\)
unchanged.
\end{proof}

\subsection{Brownian modulus tail bound}
\label{app:proof-rate-brownian-modulus}

\begin{proof}[Proof of Lemma~\ref{lem:rate-brownian-modulus}]
Fix \(0<h\le1\), and put \(N=\lceil H/h\rceil\).  For
\(\ell=0,\ldots,N-1\), define the double block
\[
  I_\ell=[\ell h,\{(\ell+2)h\}\wedge H].
\]
We first check that these blocks cover every pair appearing in the modulus.
Let \(0\le s\le t\le H\) and \(t-s\le h\).  If \(s=H\), then
\(s=t=H\) and the increment is zero.  Otherwise, let
\(\ell=\lfloor s/h\rfloor\).  Then \(0\le\ell\le N-1\), and
\[
  \ell h\le s< (\ell+1)h,
  \qquad
  t\le s+h< (\ell+2)h,
\]
while \(t\le H\). Hence \(s,t\in I_\ell\).  Consequently,
\[
  \sup_{\substack{s,t\in[0,H]\\|s-t|\le h}}|W(t)-W(s)|
  \le
  \max_{0\le\ell<N}\sup_{s,t\in I_\ell}|W(t)-W(s)|.
\]

Fix one block and set
\[
  R_\ell\coloneqq\sup_{0\le v\le2h}
  |W(\ell h+v)-W(\ell h)|.
\]
Because \(I_\ell\subseteq[\ell h,\ell h+2h]\), the triangle inequality
gives, for every \(s,t\in I_\ell\),
\[
  |W(t)-W(s)|
  \le |W(t)-W(\ell h)|+|W(s)-W(\ell h)|
  \le 2R_\ell.
\]
Here Brownian motion is defined for all nonnegative times.  Thus, for the
last block, replacing its clipped interval \(I_\ell\) by
\([\ell h,\ell h+2h]\) merely enlarges the supremum and does not require a
new process.

By stationary increments, \(R_\ell\) has the same distribution as
\(\sup_{0\le v\le2h}|W(v)|\).  Applying the reflection principle
separately to \(W\) and \(-W\), and then using a union bound, yields
\begin{align*}
  \Pp\!\left(\sup_{s,t\in I_\ell}|W(t)-W(s)|>u\right)
  &\le \Pp(R_\ell>u/2)\\
  &=\Pp\!\left(\sup_{0\le v\le2h}|W(v)|>u/2\right)\\
  &\le 2\Pp\!\left(
    \sup_{0\le v\le2h}W(v)>u/2
  \right)\\
  &=4\Pp\!\left(W(2h)>u/2\right)\\
  &\le4\exp\{-u^2/(16h)\}.
\end{align*}
The last line uses \(W(2h)/\sqrt{2h}\sim N(0,1)\) and the Gaussian
Chernoff bound.  Notice that the factor \(4\) consists of a factor \(2\)
from the two signs and a factor \(2\) from the reflection principle for
each one-sided maximum.

A union bound over the \(N\) double blocks yields
\[
  \Pp\!\left(
    \sup_{\substack{s,t\in[0,H]\\|s-t|\le h}}
    |W(t)-W(s)|>u
  \right)
  \le 4\lceil H/h\rceil\exp\{-u^2/(16h)\}.
\]
Since \(h\le1\),
\(\lceil H/h\rceil\le(H+1)h^{-1}\).  Taking
\(C_H\ge\max\{4(H+1),16\}\) proves the stated form.
\end{proof}

\subsection{Power-law contraction-sum estimates}
\label{app:proof-rate-power-law-contraction-sums}

\begin{proof}[Proof of Lemma~\ref{lem:rate-power-law-contraction-sums}]
An integral comparison gives, for $k\ge j$,
\begin{equation*}
  \sum_{i=j+1}^k\eta_i
  \ge c_0\left\{(k+1)^{1-\alpha}-(j+1)^{1-\alpha}\right\}.
\end{equation*}
A second integral comparison, followed by the substitution
$u=x^{1-\alpha}-(j+1)^{1-\alpha}$, yields
\begin{align*}
  \sum_{k=j}^{\infty}
  \exp\left(-a\sum_{i=j+1}^k\eta_i\right)
  &\le C+C\int_{j+1}^{\infty}
  e^{-c_1\{x^{1-\alpha}-(j+1)^{1-\alpha}\}}\,\mathrm dx\\
  &\le C(j+1)^\alpha
  \int_0^\infty e^{-c_1u}(1+u)^{\alpha/(1-\alpha)}\,\mathrm du\\
  &\le C(j+1)^\alpha\le C_a\eta_j^{-1}.
\end{align*}
This proves the first inequality in~\eqref{eq:rate-contraction-sums}.

For the reverse sum, fix an integer $n\ge1$ and, for $0\le j\le n$, put
$r_j=\exp\{-a\sum_{i=j+1}^n\eta_i\}$.  Since $(\eta_j)_{j\ge1}$ is bounded,
there is $c_a>0$ such that, for $1\le j\le n$,
\[
  r_j-r_{j-1}=r_j(1-e^{-a\eta_j})\ge c_a\eta_jr_j.
\]
Summing over $j=1,\ldots,n$ proves the second inequality in
\eqref{eq:rate-contraction-sums}.

Finally,
\[
  \log\frac{\eta_j}{\eta_k}
  =\alpha\sum_{i=j+1}^k\log\frac{i}{i-1}
  \le\alpha\sum_{i=j+1}^k\frac1{i-1}.
\]
Because $i^{\alpha-1}\to0$, for each $b>0$ one can choose $j_b$ so that
$\alpha/(i-1)\le b\eta i^{-\alpha}=b\eta_i$ whenever $i>j_b$.
Summation proves~\eqref{eq:linear-rate-gain-ratio}.
\end{proof}

\subsection{Deterministic power-law step-size bounds}
\label{app:proof-rate-deterministic-products}

\begin{proof}[Proof of Lemma~\ref{lem:rate-deterministic-products}]
View the real matrices as operators on the complexification of $\R^d$,
equipped with the Euclidean norm. This does not change the operator norm of
a real matrix.  Take a Jordan decomposition $\bG=\bV\bJ\bV^{-1}$.  Fix numbers
\[
  0<c<c_1<\min_{z\in\operatorname{spec}(\bG)}\operatorname{Re}(z).
\]
It is enough to treat one Jordan block of size \(m\in\{1,\ldots,d\}\),
$\bJ_z=z\bI+\bN$, where $\bN^m=0$.  Choose $j_0$ so large that, for every
$i\ge j_0$ and every eigenvalue $z$,
\begin{equation}
  \log|1-\eta_i z|\le-c_1\eta_i,
  \qquad |\eta_i z|\le\frac12.
  \label{eq:rate-jordan-scalar-contraction}
\end{equation}
Such a common \(j_0\) exists because the spectrum is finite,
\(\eta_i\to0\), and
\[
  \log|1-\eta_i z|
  =\frac12\log\{1-2\eta_i\operatorname{Re}(z)
                    +\eta_i^2|z|^2\}
  =-\eta_i\operatorname{Re}(z)+O(\eta_i^2).
\]
For $n\ge j\ge j_0$, put
$b_i=\eta_i/(1-\eta_i z)$. 
Factoring the scalar part from each factor
gives
\[
  \bI-\eta_i\bJ_z=(1-\eta_i z)(\bI-b_i\bN).
\]
Since all matrices \(\bI-b_i\bN\) are polynomials in the same nilpotent matrix
\(\bN\), they commute.  Expanding their product and using \(\bN^m=0\) yields
\begin{align}
  \prod_{i=j}^n(\bI-\eta_i\bJ_z)
  &={\prod_{i=j}^n(1-\eta_i z)}
    \prod_{i=j}^n(\bI-b_i\bN) \notag\\
  &={\prod_{i=j}^n(1-\eta_i z)}
    \sum_{r=0}^{m-1}(-1)^r e_r(b_j,\ldots,b_n)\bN^r,
  \label{eq:rate-jordan-product-expansion}
\end{align}
where \(e_0=1\) and, for \(1\le r<m\),
\[
  e_r(b_j,\ldots,b_n)
  \coloneqq\sum_{j\le i_1<\cdots<i_r\le n}
       b_{i_1}\cdots b_{i_r}.
\]
Indeed, selecting \(-b_i\bN\) from exactly \(r\) factors produces the
coefficient \((-1)^re_r\) of \(\bN^r\), while every term of degree at least
\(m\) vanishes.  Furthermore,
\(|1-\eta_i z|\ge1-|\eta_i z|\ge1/2\), so
\(|b_i|\le2\eta_i\).  Hence
\begin{align}
  |e_r(b_j,\ldots,b_n)|
  &\le\sum_{j\le i_1<\cdots<i_r\le n}
       |b_{i_1}|\cdots|b_{i_r}|\notag\\
  &\le\frac1{r!}\left(\sum_{i=j}^n|b_i|\right)^r
  \le\frac{(2\sum_{i=j}^n\eta_i)^r}{r!}.
  \label{eq:basic_polynomial_bound}
\end{align}
The middle inequality follows by expanding
\((\sum_i|b_i|)^r\): every product with \(r\) distinct indices occurs
exactly \(r!\) times, and all terms with repeated indices are nonnegative.

Combining \eqref{eq:rate-jordan-scalar-contraction}, \eqref{eq:rate-jordan-product-expansion} and \eqref{eq:basic_polynomial_bound} gives
\begin{align}
  \left\|\prod_{i=j}^n(\bI-\eta_i\bJ_z)\right\|
  &\le \exp\Big\{-c_1 \sum\nolimits_{i=j}^n\eta_i\Big\}
       \sum_{r=0}^{m-1}\frac{(2\sum_{i=j}^n\eta_i)^r}{r!}\|\bN^r\|\notag\\
  &\le C\exp\Big\{-c_1\sum\nolimits_{i=j}^n\eta_i\Big\}
       \sum_{r=0}^{m-1}\Big(\sum\nolimits_{i=j}^n\eta_i\Big)^r
  \le C\exp\Big\{-c\sum\nolimits_{i=j}^n\eta_i\Big\}.
\end{align}
For each \(0\le r\le m-1\),
\(
  \sup_{u\ge0}u^r e^{-(c_1-c)u}<\infty
\), which justifies the last inequality.

For a general Jordan normal form \(\bJ\),
write \(\bJ=\operatorname{diag}(\bJ_1,\ldots,\bJ_q)\), where
\(\bJ_1,\ldots,\bJ_q\) are its Jordan blocks.  The preceding argument applies
to every block.  Since the Euclidean operator norm of a block-diagonal matrix
equals the maximum of the operator norms of its blocks, after enlarging \(C\)
if necessary,
\[
  \left\|\prod_{i=j}^n(\bI-\eta_i\bJ)\right\|
  =\max_{1\le\ell\le q}
   \left\|\prod_{i=j}^n(\bI-\eta_i\bJ_\ell)\right\|
  \le C\exp\left(-c\sum_{i=j}^n\eta_i\right),
  \qquad n\ge j\ge j_0.
\]
Because \(\bG=\bV\bJ\bV^{-1}\) and \(\bB_i=\bI-\eta_i\bG\),
\begin{equation}
  \left\|\prod_{i=j}^n\bB_i\right\|
  \le \|\bV\|\,\|\bV^{-1}\|C
  \exp\left(-c\sum_{i=j}^n\eta_i\right)
  \le C\exp\left(-c\sum_{i=j}^n\eta_i\right).
  \label{eq:prod-B_i-bound}
\end{equation}
Since \(j_0\) is fixed, splitting the product at \(j_0\) when \(j<j_0\)
and absorbing the finitely many prefix and exponential factors, together with
the finitely many cases \(j\le n<j_0\), into \(C\) extends this bound to all
\(n\ge j\ge1\).

For the second inequality, applying the definition in
\eqref{eq:rate-recursion-coefficients}
gives
\begin{align}
  \norm{\bA_j^n}
  &\le \eta_j\sum_{k=j}^n
  \left\|\prod_{i=j+1}^k\bB_i\right\|
  \overset{
    \eqref{eq:prod-B_i-bound}}{\le} C\eta_j\sum_{k=j}^n
  \exp\left(-c\sum_{i=j+1}^k\eta_i\right).
  \label{eq:rate-A-by-contraction-sum}
\end{align}
Applying the first inequality in \eqref{eq:rate-contraction-sums} with
\(a=c\) gives
\(
  \norm{\bA_j^n}
  \le C\eta_j\eta_j^{-1}
  \le C
\)
uniformly over \(n\ge j\ge1\), which is the second assertion.
\end{proof}

\subsection{Quantitative recursion coefficients}
\label{app:proof-linear-rate-coefficients}

\begin{proof}[Proof of Lemma~\ref{lem:linear-rate-coefficients}]
Every factor in the matrix product \(\prod_{i=j+1}^k \bB_i\) is a polynomial in $\bG$. The factors therefore commute, and
\begin{equation*}
  \prod_{i=j+1}^k \bB_i -\prod_{i=j+1}^{k+1} \bB_i= \eta_{k+1}\bG
  \prod_{i=j+1}^k \bB_i.
\end{equation*}
Applying Abel's transform gives the exact identity
\begin{align}
  \bG\bA_j^T
  &= 
  \eta_j \bG \sum_{k=j}^T \prod_{i=j+1}^k \bB_i
  = \eta_j \sum_{k=j}^T \eta_{k+1}^{-1}\Bigg(\prod_{i=j+1}^k \bB_i
   -\prod_{i=j+1}^{k+1} \bB_i\Bigg)\notag\\
  &=
  \frac{\eta_j}{\eta_{j+1}}\bI
  -\frac{\eta_j}{\eta_{T+1}}\bigg(\prod_{i=j+1}^{T+1} \bB_i\bigg)
  +\sum_{k=j+1}^T
  \eta_j\left(\frac1{\eta_{k+1}}-\frac1{\eta_k}\right)
  \bigg(\prod_{i=j+1}^{k} \bB_i\bigg).
  \label{eq:linear-rate-Abel-identity}
\end{align}
Note that the sequence \(\frac{1}{\eta_k}\) is increasing and,
\begin{equation*}
  \delta_j=\sup_{k\ge j+1}
  \frac{\eta_{k-1}-\eta_k}{\eta_{k-1}^2}
  \le Cj^{\alpha-1},\quad \eta_j/\eta_{j+1}-1\le Cj^{-1}.
\end{equation*}
Lemma~\ref{lem:rate-deterministic-products} tells us that, for some $a>0$,
\begin{equation}
  \norm{\prod_{i=j+1}^k\bB_i}
  \le C\exp\left(-a\sum_{i=j+1}^k\eta_i\right).
  \label{eq:linear-rate-stable-product}
\end{equation}
Taking $b=a/2$ in~\eqref{eq:linear-rate-gain-ratio} bounds the terminal term
in~\eqref{eq:linear-rate-Abel-identity} by
\begin{equation}
  \frac{\eta_j}{\eta_{T+1}}\norm{\prod_{i=j+1}^{T+1} \bB_i}
  \le C\exp\left(-\frac a2\sum_{i=j+1}^{T+1}\eta_i\right).
  \label{eq:tmn_abel_GA}
\end{equation}
For the sum in~\eqref{eq:linear-rate-Abel-identity}, the gain-ratio bound yields the following estimate. The constant absorbs the extra factor involving \(\eta_{k+1}\) and the finitely many indices below the gain-ratio threshold:
\begin{equation}
  \eta_j\left(\frac1{\eta_{k+1}}-\frac1{\eta_k}\right)
  \le \delta_j\eta_k\frac{\eta_j}{\eta_{k+1}}
  \le C\delta_j \eta_k \exp\left(
    \frac{a}{2}\sum_{i=j+1}^k \eta_i
  \right).
  \label{eq:sum_abel_GA}
\end{equation}
Returning to \eqref{eq:linear-rate-Abel-identity}, these estimates control
\(\|\bA_j^T-\bG^{-1}\|\).  Indeed,
\begin{align*}
  \norm{\bA_j^T-\bG^{-1}}
  &\le \norm{\bG^{-1}} \norm{\bG\bA_j^T-\bI}\\
&\overset{\eqref{eq:linear-rate-Abel-identity}}{
  \le} \norm{\bG^{-1}}\left(\frac{\eta_j}{\eta_{j+1}} - 1\right)
  + \norm{\bG^{-1}}\frac{\eta_j}{\eta_{T+1}}
  \norm{\prod_{i=j+1}^{T+1} \bB_i}\\
  &\hspace{8em}+ \norm{\bG^{-1}}\sum_{k=j+1}^T \eta_{j}\left(\frac{1}{\eta_{k+1}
  } - \frac{1}{\eta_k}\right)
  \norm{\prod_{i=j+1}^{k} \bB_i}\\
  &\le C\norm{\bG^{-1}}\left\{j^{-1}
    +\exp\left(-\frac{a}{2}\sum_{i=j+1}^{T+1}\eta_i\right)\right\}\\
  &\quad+C\norm{\bG^{-1}}\delta_j\sum_{k=j+1}^T\eta_k
    \exp\left(-\frac{a}{2}\sum_{i=j+1}^k\eta_i\right)\\
  &\le C\left\{\delta_j+
    \exp\left(-\frac{a}{2}\sum_{i=j+1}^{T+1}\eta_i\right)\right\}.
\end{align*}
Here the last two inequalities use \eqref{eq:linear-rate-stable-product}, \eqref{eq:sum_abel_GA}, and \eqref{eq:rate-contraction-sums}.

To verify the weighted-sum bound, put \(v_k=\exp\{-(a/2)\sum_{i=j+1}^k\eta_i\}\). Then
\[
  v_{k-1}-v_k=v_k(e^{a\eta_k/2}-1)\ge(a/2)\eta_kv_k.
\]
Thus the forward weighted sum is uniformly bounded by telescoping.
It remains to sum the pointwise estimate.  Note that for \(s\ge 2\),
\(
  \sum_{j=1}^Tj^{-s(1-\alpha)}\le CT^\alpha
\).  For the terminal boundary layer,
\begin{equation}
\begin{split}
  \sum_{j=1}^T
  \exp\left(-c\sum_{i=j+1}^{T+1}\eta_i\right)
  &\le \eta_T^{-1}\sum_{j=1}^T\eta_j
  \exp\left(-c\sum_{i=j+1}^{T+1}\eta_i\right)\\
  &\le C\eta_T^{-1}\le CT^\alpha.
\end{split}
  \label{eq:linear-rate-reverse-exponential-sum}
\end{equation}
Indeed, with
$w_j=\exp\{-c\sum_{i=j+1}^{T+1}\eta_i\}$,
$w_j-w_{j-1}=w_j(1-e^{-c\eta_j})\gtrsim\eta_jw_j$, so the weighted sum in
\eqref{eq:linear-rate-reverse-exponential-sum} telescopes.  Applying
$(u+v)^s\le2^{s-1}(u^s+v^s)$ completes the proof.
\end{proof}

\subsection{Adaptive-block maximal inequality}
\label{app:proof-linear-rate-filter}

\begin{proof}[Proof of Lemma~\ref{lem:linear-rate-filter}]
Consider a constant $L\ge1$ such that
$Ce^{-c\eta2^{-\alpha}L}\le1/2$.
We divide the time horizon into several subintervals.  First, choose $t_0$
such that
\begin{equation}
  L s^{\alpha-1}+s^{-1}\le\frac12,
  \qquad s\ge t_0.
  \label{eq:linear-rate-block-comparability-condition}
\end{equation}
It is easy to see that \(t_0\) is a finite number and 
will not increase with \(T\). Then, we adopt the notation
\(y_* \coloneqq \max_{1\le t\le t_0}\norm{\by_t / \eta_t}\).
 This is the maximum of
finitely many linear combinations of the $\bepsilon_t$, so
$\E\!\left[ y_*^p\right]\le C$ and
\begin{equation}
  \Pp( y_*>\sqrt T x)\le CT^{-p/2}x^{-p}.
  \label{eq:linear-rate-prefix-tail}
\end{equation}
For \(T+1\le t_0\), the desired bound follows from~\eqref{eq:linear-rate-prefix-tail}, after increasing the constant. We may therefore assume \(T+1>t_0\) and inductively construct blocks from $s_0=t_0$ by
\begin{equation}
  s_{\ell+1}
  =\min\{T+1,s_\ell+\lceil L s_\ell^\alpha\rceil\}.
  \label{eq:cstct_s_l}
\end{equation}
Write $h_\ell=s_{\ell+1}-s_\ell$ and let $N_T$ be the number of blocks. 
Simple computation gives
\begin{equation}
  N_T\le CT^{1-\alpha},
  \qquad
  \max_{\ell<N_T}h_\ell\le CT^\alpha,
  \qquad
  \sum_{\ell<N_T}h_\ell\le T+1.
  \label{eq:linear-rate-block-counts}
\end{equation}
On a nontruncated block starting at $s=s_\ell$,
\eqref{eq:linear-rate-block-comparability-condition} implies
$s_{\ell+1}\le2s$ and hence
\begin{equation*}
  \sum_{t=s_\ell}^{s_{\ell+1}-1}\eta_t
  \ge c\,h_\ell s_\ell^{-\alpha}
  \overset{\eqref{eq:cstct_s_l}}{\ge} cL.
\end{equation*}
Together with Lemma~\ref{lem:rate-deterministic-products},
 this gives the uniform
contraction
\begin{equation}
  \frac{\eta_{s_\ell}}{
    \eta_{s_{\ell+1}}}\left\|\prod_{t=s_\ell}^{s_{\ell+1}-1}\bB_t\right\|
  \le Ce^{-cL} \le\frac12.
  \label{eq:linear-rate-block-contraction}
\end{equation}
It is worth emphasizing that the terminal block,
 which may be truncated by $T+1$, will not use this
contraction.

For $s_\ell\le k<s_{\ell+1}$, iteration of
\(\by_k\) yields
\begin{align*}
   \by_{k+1}/\eta_{k+1}&=
   \frac{\eta_{s_\ell}}{\eta_{k+1}}\left(\prod_{i=s_\ell}^k \bB_i\right)
   (\by_{s_\ell}/\eta_{s_\ell})
  +\sum_{j=s_\ell}^k \frac{\eta_j}{\eta_{k+1}}\left(\prod_{i=j+1}^k
   \bB_i\right)\bepsilon_j\\
  &\overset{\text{(a)}}{=} \frac{\eta_{s_\ell}}{\eta_{k+1}}
   \left(\prod_{i=s_\ell}^k \bB_i\right)
   (\by_{s_\ell}/\eta_{s_\ell})
   + \sum_{j=s_\ell}^k \frac{\eta_j}{\eta_{j+1}}\bepsilon_j\\
   &\quad + \sum_{j=s_{\ell}}^{k-1}\left[
    \frac{\eta_{j+1}}{\eta_{k+1}}\left(\prod_{i=j+1}^k \bB_i \right)
    - \frac{\eta_{j+2}}{\eta_{k+1}}\left(\prod_{i=j+2}^k \bB_i\right)
   \right]\left\{\sum_{i= s_{\ell}}^j \frac{\eta_i}{\eta_{i+1}}
   \bepsilon_i\right\}\\
  &= \frac{\eta_{s_\ell}}{\eta_{k+1}}\left(\prod_{i=s_\ell}^k \bB_i\right)
   (\by_{s_\ell}/\eta_{s_\ell})
   + \sum_{j=s_\ell}^k \frac{\eta_j}{\eta_{j+1}}\bepsilon_j\\
   &\quad + \sum_{j=s_{\ell}}^{k-1}
   \left( \frac{\eta_{j+1}}{\eta_{j+2}}(\bI - \eta_{j+1}\bG) - \bI \right)
   \frac{\eta_{j+2}}{\eta_{k+1}}\left(\prod_{i=j+2}^k \bB_i\right)
   \left\{\sum_{i= s_{\ell}}^j \frac{\eta_i}{\eta_{i+1}}
   \bepsilon_i\right\}
\end{align*}
Step (a) uses Abel summation.
For simplicity, we define
\[
\widehat{U}_{\ell} \coloneqq \max_{s_\ell\le k<s_{\ell+1}}
\bigg\|\sum_{j=s_\ell}^k (\eta_j/\eta_{j+1})\bepsilon_j\bigg\|.
\]
Taking norms on both sides of the above equation yields
\begin{align*}
  \norm{\by_{k+1}/\eta_{k+1}} &\le
  \frac{\eta_{s_\ell}}{\eta_{k+1}}\left\|\prod_{i=s_\ell}^k \bB_i\right\|
   \norm{\by_{s_\ell}/\eta_{s_\ell}}
   + \norm{\sum_{j=s_\ell}^k \frac{\eta_j}{\eta_{j+1}}\bepsilon_j}\\
   &\quad + \sum_{j=s_{\ell}}^{k-1}
   \left\| \frac{\eta_{j+1}}{\eta_{j+2}}(\bI - \eta_{j+1}\bG) - \bI \right\|
   \frac{\eta_{j+2}}{\eta_{k+1}}\left\|\prod_{i=j+2}^k \bB_i\right\|\times
   \left\|\sum_{i= s_{\ell}}^j \frac{\eta_i}{\eta_{i+1}}
   \bepsilon_i\right\|\\
   &\le \frac{\eta_{s_\ell}}{\eta_{k+1}}\left\|\prod_{i=s_\ell}^k
    \bB_i\right\|
   \norm{\by_{s_\ell}/\eta_{s_\ell}} + \widehat{U}_\ell
    + C \widehat{U}_\ell \sum_{j=s_\ell}^{k-1}
    \eta_{j+1}\frac{\eta_{j+2}}{\eta_{k+1}}
   \left\|\prod_{i=j+2}^k \bB_i\right\|\\
   &\le \frac{\eta_{s_\ell}}{\eta_{k+1}}\left\|\prod_{i=s_\ell}^k
    \bB_i\right\|
   \norm{\by_{s_\ell}/\eta_{s_\ell}}
   + C' \widehat{U}_\ell.
\end{align*}
Here the penultimate inequality uses
\[
 \left\|\frac{\eta_{j+1}}{\eta_{j+2}}\bB_{j+1}-\bI\right\|
 \le C(j^{-1}+\eta_{j+1})\le C\eta_{j+1}.
\]
For the last inequality, block comparability gives \(k+1\le2s_\ell\), uniformly bounded step-size ratios, and \(h_\ell\le C_Ls_\ell^\alpha\). Lemma~\ref{lem:rate-deterministic-products} bounds the product norms uniformly, and hence
\[
 \sum_{j=s_\ell}^{k-1}\eta_{j+1}
 \frac{\eta_{j+2}}{\eta_{k+1}}
 \left\|\prod_{i=j+2}^k\bB_i\right\|
 \le C h_\ell s_\ell^{-\alpha}\le C_L.
\]
The constant \(L\) is fixed, so these constants do not depend on \(T\).
The following endpoint bound holds on complete blocks, whereas the interior bound also holds on the possibly truncated terminal block:
\begin{subequations}
\begin{equation}
  \| \by_{s_{\ell + 1}}/ \eta_{s_{\ell+1}}\|
  \overset{
    \eqref{eq:linear-rate-block-contraction}
    }{\le} \frac{1}{2}\| \by_{s_\ell}/\eta_{s_\ell}\|+C \widehat{U}_\ell,
  \label{eq:linear-rate-block-endpoint}
\end{equation}
\begin{equation}
  \max_{s_\ell + 1\le k \le s_{\ell+1}}\| \by_k/\eta_k\|
  \le C\| \by_{s_\ell}/\eta_{s_\ell}\|+C \widehat{U}_\ell.
  \label{eq:linear-rate-block-interior}
\end{equation}  
\end{subequations}
Taking $C\ge1$, induction over the complete blocks gives, for every block starting point,
\begin{equation}
\| \by_{s_\ell}/ \eta_{s_\ell}\| \le 2C \left\{\max_{0\le m < N_T}
\widehat{U}_m + y_*\right\}
\label{eq:y_s_ell-unif-bd}
\end{equation}
and hence the uniform bound
\begin{align}
  \max_{1\le k\le T+1}\| \by_k/\eta_k\|
  &= \max\left\{y_*,\,
    \max_{\ell < N_T}\max_{s_\ell +1 \le k \le s_{\ell +1}}
    \|\by_k / \eta_k\|\right\}\notag\\
  &\le \max\left\{y_*,\,C\max_{\ell < N_T}\left\{
    \|\by_{s_\ell}/ \eta_{s_\ell}\| + \widehat{U}_\ell
  \right\}\right\}\notag\\
  &\overset{\eqref{eq:y_s_ell-unif-bd}}{\le}
   C'\Big\{ y_*+ \max_{\ell<N_T}\widehat{U}_\ell\Big\},
  \label{eq:linear-rate-filter-global-max}
\end{align}
where $C'=1+2C^2+C$ is independent of $T$. The initial maximum $y_*$ is retained throughout because the block estimates control only times after $t_0$.

It remains to bound $\max_{\ell < N_T}\widehat{U}_\ell$.
For each coordinate vector $\mathbf e_r$, conditional
Jensen's inequality and \eqref{eq:linear-rate-filter-moment} give, almost surely,
\begin{align*}
  \sum_{j=s_\ell}^{s_{\ell+1}-1}
  \E\!\left[\big(\mathbf e_r^\top (\eta_j/\eta_{j+1})\bepsilon_j
  \big)^2\mid\cF_{j-1}\right]
  &\le C(M_\epsilon,p)h_\ell,\\
  \sum_{j=s_\ell}^{s_{\ell+1}-1}
  \E\!\left[\big|\mathbf e_r^\top
  (\eta_j/\eta_{j+1})\bepsilon_j\big|^p\mid\cF_{j-1}\right]
  &\le C(M_\epsilon,p)h_\ell.
\end{align*}
Apply Lemma~\ref{lem:rate-fuk-nagaev} to each coordinate and use
$\|\bz\|\le\sqrt d\max_r|z_r|$.  Since $d$ is fixed, a union bound yields
\begin{equation*}
  \Pp\big(\widehat{U}_\ell>y\big)
  \le C\exp(-cy^2/h_\ell)+Ch_\ell y^{-p},
\end{equation*}
where \(C\) may depend on \(M_\epsilon, d\) and \(p\).
Set $y=\sqrt T x/(2C')$, use~\eqref{eq:linear-rate-prefix-tail},
and take a union bound in~\eqref{eq:linear-rate-filter-global-max}.  
The three relations in
\eqref{eq:linear-rate-block-counts} give
\begin{align*}
  \Pp\!\left(\max_{t\le T}\|\by_{t+1}/\eta_{t+1}\|>\sqrt T x\right)
  &\le \Pp\Big(C'y_* + C' \max_{\ell < N_T}
  \widehat{U}_\ell > \sqrt{T}x\Big)\\
  &\le \Pp\big(y_* > \sqrt{T}x/(2C')\big) +
  \sum_{\ell<N_T} \Pp\big(\widehat{U}_\ell > \sqrt{T}x/(2C')\big)\\
  &\le CT^{-p/2}x^{-p}
  +C\sum_{\ell<N_T}
  \left\{e^{-cTx^2/h_\ell}+h_\ell T^{-p/2}x^{-p}\right\}\\
  &\le CT^{1-\alpha}e^{-cT^{1-\alpha}x^2}
  +CT^{1-p/2}x^{-p}.
\end{align*}
This concludes the proof of the lemma.
\end{proof}

\subsection{Affine scalar martingale functional approximation}
\label{app:proof-rate-leading-mds-coupling}

\begin{proof}[Proof of Lemma~\ref{lem:rate-leading-mds-coupling}]
Write \(S_k=\sum_{j=1}^kY_j\), with \(S_0=0\).  If $\sigma^2=0$, the
terminal case of
\eqref{eq:rate-leading-scalar-covariance} and nonnegativity of conditional
variances give
\[
  \E\!\left[\frac1T\sum_{j=1}^T
      \E\!\left[Y_j^2\mid\cF_{j-1}\right]\right]
  \le KT^{-3\gamma}.
\]
Convexity of the absolute value on each affine interpolation interval gives
\[
  \tnorm{\mathcal S_T}
  =\frac1{\sqrt T}\max_{0\le k\le T}|S_k|.
\]
Consequently, Doob's $L^2$ inequality yields
\[
  \Pp\!\left(\tnorm{\mathcal S_T}>x\right)
  \le \frac{4\E\!\left[S_T^2\right]}{Tx^2}
  =\frac4{x^2}\E\!\left[\frac1T\sum_{j=1}^T
       \E\!\left[Y_j^2\mid\cF_{j-1}\right]\right]
  \le CT^{-3\gamma}x^{-2}.
\]
Taking $x=C_0T^{-\gamma}$, with $C_0$ sufficiently large, in
\eqref{eq:rate-coupling-prokhorov} proves
\eqref{eq:rate-leading-mds-coupling} when the Gaussian target is the zero
path, even without the logarithmic factor.  Hence assume below that
\(\sigma^2>0\).

Set
\(
  v_j\coloneqq\E\!\left[Y_j^2\mid\cF_{j-1}\right]
\)
and apply Lemma~\ref{lem:rate-scott-huggins-embedding}, using \(r=p\) in its
moment comparison.  On a complete extension of the probability space, the
lemma supplies a standard Brownian motion \(W\), stopping times
\(0=\tau_0\le\tau_1\le\cdots\), and a filtration
\((\mathcal G_j)_{j\ge0}\) such that, with
\(\rho_j\coloneqq\tau_j-\tau_{j-1}\),
\begin{align}
  W(\tau_k)&=S_k, \qquad k\ge0,
  \label{eq:rate-embedded-partial-sums}\\
  \E\!\left[\rho_j\mid\mathcal G_{j-1}\right]
  &=v_j,
  \label{eq:rate-embedded-time-mean}\\
  c_p\E\!\left[\rho_j^{p/2}\mid\mathcal G_{j-1}\right]
  &\le \E\!\left[|Y_j|^p\mid\cF_{j-1}\right]
  \le C_p\E\!\left[\rho_j^{p/2}\mid\mathcal G_{j-1}\right]
  \quad\text{almost surely}.
  \label{eq:rate-embedded-time-moment-comparison}
\end{align}
Moreover, \(\cF_{j-1}\subseteq\mathcal G_{j-1}\), \(\rho_j\) is
\(\mathcal G_j\)-measurable, and the extension leaves the law of
\((S_k)_{k\ge0}\) unchanged.  
Define the rescaled Brownian motion
\[
  \widetilde W_T(s)
  \coloneqq T^{-1/2}W(Ts), \qquad s\ge0,
\]
and the typical event
\begin{equation*}
  \mathcal E_T
  \coloneqq\left\{
    \max_{0\le k\le T}
    \left|\frac{\tau_k}{T}-\frac{k}{T}\sigma^2\right|
    \le2T^{-2\gamma}
  \right\}.
\end{equation*}
Put
\[
  \delta_T\coloneqq(2+\sigma^2)T^{-2\gamma},
  \qquad H\coloneqq\sigma^2+2.
\]
We first compare the two entire paths on \(\mathcal E_T\). The probability
of this event will be estimated afterward.  Given \(r\in[0,1]\), set
\[
  k\coloneqq\min\{\lfloor Tr\rfloor,T-1\},
  \qquad \lambda\coloneqq Tr-k.
\]
Then \(k\in\{0,\ldots,T-1\}\), \(0\le\lambda\le1\), and this convention
also covers \(r=1\), for which \(k=T-1\) and \(\lambda=1\).  By the affine
definition of \(\mathcal S_T\) and the embedding identity
\eqref{eq:rate-embedded-partial-sums},
\begin{equation}
  \mathcal S_T(r)
  =(1-\lambda)\widetilde W_T(\tau_k/T)
    +\lambda \widetilde W_T(\tau_{k+1}/T).
  \label{eq:rate-embedded-affine-representation}
\end{equation}
For each \(m\in\{k,k+1\}\), the definition of \(\mathcal E_T\) gives
\begin{align*}
  \left|\frac{\tau_m}{T}-\sigma^2r\right|
  &\le
  \left|\frac{\tau_m}{T}-\frac{m}{T}\sigma^2\right|
  +\sigma^2\left|\frac mT-r\right|
  \le2T^{-2\gamma}+\frac{\sigma^2}{T}
  \le\delta_T,
\end{align*}
where the last inequality uses \(2\gamma<1\).  On the same event,
\(\tau_m/T\le\sigma^2+2=H\), while \(0\le\sigma^2r\le H\).  Subtracting
\(\widetilde W_T(\sigma^2r)\) from
\eqref{eq:rate-embedded-affine-representation} and using
\(0\le\lambda\le1\) therefore gives
\begin{align*}
  \left|\mathcal S_T(r)-\widetilde W_T(\sigma^2r)\right|
  &\le(1-\lambda)
    \left|\widetilde W_T(\tau_k/T)
      -\widetilde W_T(\sigma^2r)\right|
  +\lambda
    \left|\widetilde W_T(\tau_{k+1}/T)
                 -\widetilde W_T(\sigma^2r)\right|\\
  &\le
  \sup_{\substack{s,t\in[0,H]\\|s-t|\le\delta_T}}
    |\widetilde W_T(t)-\widetilde W_T(s)|.
\end{align*}
Taking the supremum over \(r\in[0,1]\) proves the pathwise bound
\begin{equation}
  \tnorm{\mathcal S_T-\widetilde W_T(\sigma^2\,\cdot)}
  \le
  \sup_{\substack{s,t\in[0,H]\\|s-t|\le\delta_T}}
    |\widetilde W_T(t)-\widetilde W_T(s)|
  \qquad\text{on }\mathcal E_T.
  \label{eq:rate-embedded-pathwise-modulus}
\end{equation}

The discrepancy between the embedded Brownian time \(\tau_k\) and
\(k\sigma^2\) has two distinct sources.  We first control the
centered embedding fluctuation
\(\sum_{j=1}^k(\rho_j-v_j)\), whose tail is governed by the conditional
\(p\)th-moment bound.  Since \(v_j\) is
\(\mathcal G_{j-1}\)-measurable,
\eqref{eq:rate-embedded-time-mean} shows directly that
\((\rho_j-v_j,\mathcal G_j)_{j\ge1}\) is a martingale-difference sequence.
By \eqref{eq:rate-embedded-time-moment-comparison}, conditional Jensen's
inequality, and \eqref{eq:rate-leading-mds-moment},
\begin{align}
  \E\!\left[|\rho_j-v_j|^{p/2}\mid\mathcal G_{j-1}\right]
  &\le2^{p/2-1}\left\{
    \E\!\left[\rho_j^{p/2}\mid\mathcal G_{j-1}\right]+v_j^{p/2}
  \right\}\notag\\
  &\le C\E\!\left[|Y_j|^p\mid\cF_{j-1}\right]
  \le CM_Y
  \quad\text{almost surely}.
  \label{eq:rate-embedded-time-p-half-moment}
\end{align}
Here
\(v_j^{p/2}\le\E\!\left[|Y_j|^p\mid\cF_{j-1}\right]\) is precisely the
conditional Jensen step.  If \(p>4\), conditional Lyapunov's inequality
also gives
\begin{equation}
  \E\!\left[(\rho_j-v_j)^2\mid\mathcal G_{j-1}\right]
  \le
  \left\{
    \E\!\left[|\rho_j-v_j|^{p/2}\mid\mathcal G_{j-1}\right]
  \right\}^{4/p}
  \le C
  \quad\text{almost surely}.
  \label{eq:rate-embedded-time-variance}
\end{equation}

Let \(a_T=T^{-2\gamma}\) and
\(x_T=Ta_T=T^{1-2\gamma}\).  If \(2<p\le4\), Markov's inequality,
Lemma~\ref{lem:rate-subquadratic-bdg}, and
\eqref{eq:rate-embedded-time-p-half-moment} imply
\begin{align*}
  \Pp\!\left(
    \max_{k\le T}\left|\sum_{j=1}^k(\rho_j-v_j)\right|>x_T
  \right)
  &\le x_T^{-p/2}
    \E\!\left[
      \max_{k\le T}\left|\sum_{j=1}^k(\rho_j-v_j)\right|^{p/2}
    \right]\\
  &\le CTx_T^{-p/2}
   =CT^{1-p/2+p\gamma}
   \le CT^{-\gamma}.
\end{align*}
The last inequality is equivalent to
\(\gamma\le(p-2)/\{2(p+1)\}\).  If \(p>4\), apply
Lemma~\ref{lem:rate-fuk-nagaev} with moment exponent \(p/2\), predictable
variance bound \(CT\), and conditional \(p/2\)-moment bound \(CT\).  By
\eqref{eq:rate-embedded-time-p-half-moment} and
\eqref{eq:rate-embedded-time-variance},
\begin{align*}
  \Pp\!\left(
    \max_{k\le T}\left|\sum_{j=1}^k(\rho_j-v_j)\right|>x_T
  \right)
  &\le C\exp\{-cx_T^2/T\}+CTx_T^{-p/2}\\
  &\le C\exp\{-cT^{1-4\gamma}\}
       +CT^{1-p/2+p\gamma}
  \le CT^{-\gamma}.
\end{align*}
Here the exponential term is \(O(T^{-\gamma})\) because \(\gamma<1/4\),
and the polynomial term is controlled by the same restriction
\(\gamma\le(p-2)/\{2(p+1)\}\).  Thus, in both moment regimes,
\begin{equation}
  \Pp\!\left(
    \max_{k\le T}\left|\sum_{j=1}^k(\rho_j-v_j)\right|
    >T^{1-2\gamma}
  \right)
  \le CT^{-\gamma}.
  \label{eq:rate-embedded-time-fluctuation-tail}
\end{equation}

We next control the predictable conditional-covariance error
\(\sum_{j=1}^kv_j-k\sigma^2\).  The conditional-covariance bound
\eqref{eq:rate-leading-scalar-covariance} and Markov's inequality give
\begin{equation}
  \Pp\!\left(
    \max_{k\le T}\left|
      \frac1T\sum_{j=1}^kv_j-\frac{k}{T}\sigma^2
    \right|>T^{-2\gamma}
  \right)
  \le T^{2\gamma}\,KT^{-3\gamma}
  =KT^{-\gamma}.
  \label{eq:rate-conditional-covariance-tail}
\end{equation}
Since \(\tau_k=\sum_{j=1}^k\rho_j\), the exact decomposition
\begin{equation*}
  \frac{\tau_k}{T}-\frac{k}{T}\sigma^2
  =\frac1T\sum_{j=1}^k(\rho_j-v_j)
   +\left\{\frac1T\sum_{j=1}^kv_j-\frac{k}{T}\sigma^2\right\}
\end{equation*}
and a union bound using
\eqref{eq:rate-embedded-time-fluctuation-tail}--
\eqref{eq:rate-conditional-covariance-tail} show that
\begin{equation}
  \Pp(\mathcal E_T^c)\le CT^{-\gamma}.
  \label{eq:rate-embedded-time-good-event}
\end{equation}

Finally, \(\widetilde W_T\) is itself a standard Brownian motion,
and the path \(r\mapsto \widetilde W_T(\sigma^2r)\) has the law of
\(\sigma W_\star^{(1)}\).  Take
\(u_T=L T^{-\gamma}\sqrt{\log T}\).  For all sufficiently large \(T\),
\(\delta_T\le1\), so Lemma~\ref{lem:rate-brownian-modulus} gives
\begin{align*}
  &\Pp\!\left(
    \sup_{\substack{s,t\in[0,H]\\|s-t|\le\delta_T}}
    |\widetilde W_T(t)-\widetilde W_T(s)|>u_T
  \right)\\
  &\qquad\le
  C_H\delta_T^{-1}\exp\{-u_T^2/(C_H\delta_T)\}
  \le CT^{2\gamma-cL^2}
  \le CT^{-\gamma},
\end{align*}
where the last inequality follows by choosing the fixed constant \(L\)
sufficiently large.  Combining this estimate with
\eqref{eq:rate-embedded-pathwise-modulus} and
\eqref{eq:rate-embedded-time-good-event} yields
\[
  \Pp\!\left(
    \tnorm{\mathcal S_T-\widetilde W_T(\sigma^2\,\cdot)}>u_T
  \right)
  \le CT^{-\gamma}.
\]
Combining this estimate with the definition of Prokhorov distance yields
\eqref{eq:rate-leading-mds-coupling}.
\end{proof}

\section{Proofs of the lower bounds}

\subsection{Lower bound from conditional-covariance stabilization}
\label{app:rate-covariance-stabilization-lower}

\begin{proof}[Proof of Theorem~\ref{thm:rate-covariance-stabilization-lower}]
Put \(\gamma=\rateCov\).
Fix a construction horizon \(T\), and reserve \(m\)
for an arbitrary horizon in the uniform covariance-stabilization condition.
Let \(W\) be a standard Brownian motion with its augmented
natural filtration.
Put
\[
  h_T\coloneqq\frac14T^{1-2\gamma},
  \qquad
  \tau_T\coloneqq
  \inf\{s\ge T-h_T:W(s)=0\}\wedge T.
\]
For \(j\ge1\), let
\[
  \epsilon_j^{(T)}
  \coloneqq
  \begin{cases}
    W(j\wedge\tau_T)-W((j-1)\wedge\tau_T),
      &1\le j\le T,\\
    W(j)-W(j-1),&j>T,
  \end{cases}
\]
Let \(\mathbb P_{\mathcal M_T}\) be the Brownian probability law, let
\(\mathcal F_j^{\mathcal M_T}\) be the augmented Brownian filtration at time
\(j\), and set
\(\epsilon_j^{\mathcal M_T}\coloneqq\epsilon_j^{(T)}\) and
\(g_{\mathcal M_T}(x)=x\).  These objects define the stochastic oracle
\(\mathcal M_T\).  Let
\(\mathfrak C_{\mathrm{cov}}=\{\mathcal M_T:T\ge1\}\).
Exact linearity verifies Assumption~\ref{ass:rate-linear-dynamics}.

For \(j\le T\), the martingale-difference property follows by applying
optional stopping over the bounded interval \([j-1,j]\). For \(j>T\), it
follows from independent Brownian increments.  Moreover, for \(j\le T\),
\[
  |\epsilon_j^{\mathcal M_T}|
  \le
  \sup_{0\le u\le1}|W(j-1+u)-W(j-1)|,
\]
and the same bound is immediate for \(j>T\).  Since the future Brownian
increment process is independent of \(\mathcal F_{j-1}^{\mathcal M_T}\),
\[
  \E_{\mathcal M_T}\!\left[
    |\epsilon_j^{\mathcal M_T}|^p
    \mid\mathcal F_{j-1}^{\mathcal M_T}
  \right]
  \le
  \E_{\mathcal M_T}\!\left[\sup_{0\le u\le1}|W(u)|^p\right]
  \eqqcolon M_p<\infty,
\]
uniformly in \(j\) and \(T\).
Thus Assumption~\ref{ass:rate-martingale-noise} holds uniformly over
\(\mathfrak C_{\mathrm{cov}}\).

We next verify Assumption~\ref{ass:rate-covariance-stabilization}.  To this
end, fix an arbitrary auxiliary horizon \(m\ge2\).  For \(j\le T\),
conditional It\^o isometry for the stopped Brownian martingale gives
\[
  \E_{\mathcal M_T}\!\left[
    (\epsilon_j^{\mathcal M_T})^2
    \mid\mathcal F_{j-1}^{\mathcal M_T}
  \right]
  =
  \E_{\mathcal M_T}\!\left[
    (j\wedge\tau_T)-((j-1)\wedge\tau_T)
    \mid\mathcal F_{j-1}^{\mathcal M_T}
  \right].
\]
For \(j>T\), the same conditional second moment equals $1$.  Hence
\(0\le
\E_{\mathcal M_T}[(\epsilon_j^{\mathcal M_T})^2
\mid\mathcal F_{j-1}^{\mathcal M_T}]\le1\) for every \(j\ge1\), and
the conditional-covariance error in
Assumption~\ref{ass:rate-covariance-stabilization} satisfies
\begin{equation}
\begin{aligned}
  &\E_{\mathcal M_T}\!\left[
    \max_{0\le n\le m}
    \left|
      \frac1m\sum_{j=1}^n
        \E_{\mathcal M_T}\!\left[
          (\epsilon_j^{\mathcal M_T})^2
          \mid\mathcal F_{j-1}^{\mathcal M_T}
        \right]
      -\frac nm
    \right|
  \right]\\
  &\quad=
  \frac1m\sum_{j=1}^{m\wedge T}
  \E_{\mathcal M_T}\!\left[
    1-\E_{\mathcal M_T}\!\left[
      (\epsilon_j^{\mathcal M_T})^2
      \mid\mathcal F_{j-1}^{\mathcal M_T}
    \right]
  \right]
  =\frac1m\E_{\mathcal M_T}\!\left[(m\wedge T-\tau_T)_+\right].
\end{aligned}
  \label{eq:rate-covariance-lower-all-horizons}
\end{equation}

It remains to estimate the right-hand side of
\eqref{eq:rate-covariance-lower-all-horizons}. Let
\[
  \ell_t\coloneqq\sup\{0\le s\le t:W(s)=0\}.
\]
From the definition of \(\tau_T\), we can get
\begin{equation}
\begin{aligned}
  \mathbb P_{\mathcal M_T}\!\left(\tau_T<T\right)
  &= \mathbb P_{\mathcal M_T}\!(\ell_T > T- h_T)\\
  &\overset{\text{(a)}}{=}1-\frac2\pi\arcsin\sqrt{\frac{T-h_T}{T}}
  =\frac2\pi\arcsin\sqrt{\frac{h_T}{T}}
   =\frac1\pi T^{-\gamma}
     \{1+O(T^{-2\gamma})\}.
\end{aligned}
  \label{eq:rate-covariance-lower-hitting-probability}
\end{equation}
Here (a) follows from  L\'evy's last-zero arcsine law
\cite[Theorem~5.26(a)]{morters2010brownian},
\[
  \mathbb P_{\mathcal M_T}\!\left(\ell_t\le x\right)
  =\frac2\pi\arcsin\sqrt{\frac xt},
  \qquad 0\le x\le t,
\]
and the last equality holds by using the definition of \(h_T\) and
the exact formulation of \(\arcsin\) function.
The layer-cake formula and
\eqref{eq:rate-covariance-lower-hitting-probability} give
\begin{align}
  \E_{\mathcal M_T}(T-\tau_T)
  &=\int_0^{h_T}
    \mathbb P_{\mathcal M_T}\!\left(
      \tau_T<T-h_T+u
    \right)\,\mathrm du\notag\\
  &=\frac2\pi\int_0^{h_T}
      \arctan\sqrt{\frac{u}{T-h_T}}\,\mathrm du\notag\\
  &=\frac2\pi
    \left\{
      T\arctan\sqrt{\frac{h_T}{T-h_T}}
      -\sqrt{(T-h_T)h_T}
    \right\}\notag\\
  &=\frac1{6\pi}T^{1-3\gamma}
    \{1+O(T^{-2\gamma})\}.
  \label{eq:rate-covariance-lower-expected-loss}
\end{align}

Using these results, we can now analyze
the right-hand side of \eqref{eq:rate-covariance-lower-all-horizons}
exactly. Since \(T-h_T\ge3T/4\), there
are three cases.  If
\(m\le T-h_T\), then \(\tau_T\ge T-h_T\), so the right-hand side of
\eqref{eq:rate-covariance-lower-all-horizons} is zero.  If
\(T-h_T<m\le T\), then \eqref{eq:rate-covariance-lower-expected-loss} gives
\[
  \frac1m\E_{\mathcal M_T}\!\left[(m-\tau_T)_+\right]
  \le
  \frac1m\E_{\mathcal M_T}(T-\tau_T)
  \le CT^{-3\gamma}
  \le Cm^{-3\gamma}.
\]
Finally, if \(m\ge T\), then, because \(1-3\gamma\ge0\),
\(
  \frac1m\E_{\mathcal M_T}(T-\tau_T)
  \le C\frac{T^{1-3\gamma}}m
  \le Cm^{-3\gamma}
\).
Thus Assumption~\ref{ass:rate-covariance-stabilization}
holds uniformly in \(m\) and
\(T\), with a common constant \(C_S\).
Moreover, \cite[Corollary~2.5]{dereich2019general} verifies
Assumption~\ref{ass:rate-trajectory-stability} uniformly over
\(\mathfrak C_{\mathrm{cov}}\).

For this oracle, the terminal value of \(\bmM_T^{\mathcal M_T}\) telescopes to
\(
  \bmM_T^{\mathcal M_T}(1)=\frac{W(\tau_T)}{\sqrt T}
\).
On \(\{\tau_T<T\}\), this value is zero.  Hence
\eqref{eq:rate-covariance-lower-hitting-probability} implies that, for some
constant \(c_0>0\) and all sufficiently large \(T\),
\begin{equation}
  \mathbb P_{\mathcal M_T}\!\left(\bmM_T^{\mathcal M_T}(1)=0\right)
  \ge c_0T^{-\gamma}.
  \label{eq:rate-covariance-lower-atom}
\end{equation}

Consider the closed set
\(
  \Gamma\coloneqq
  \{x\in D([0,1],\R):x(1)=0\}
\).
On \(\{\tau_T<T\}\), \(\bmM_T^{\mathcal M_T}\in\Gamma\).
For its \(\delta\)-expansion in
Definition~\ref{def:rate-prokhorov}, note that
every increasing time change in \(\Lambda\) fixes
the endpoint \(1\).  Therefore,
Definition~\ref{def:rate-j1-metric} gives
\[
  d_{\mathrm S}(x,y)\ge|x(1)-y(1)|,
  \qquad
  \Gamma^\delta
  \subseteq\{y:|y(1)|<\delta\}.
\]
Since \(W_\star^{(1)}(1)\) is standard Gaussian,
\[
  \Pp\!\left(W_\star^{(1)}\in\Gamma^\delta\right)
  \le
  \Pp\!\left(|W_\star^{(1)}(1)|<\delta\right)
  \le\sqrt{\frac2\pi}\,\delta.
\]
Set
\(
  \delta_T
  \coloneqq
  \frac{c_0}{2\{1+\sqrt{2/\pi}\}}T^{-\gamma}
\).
Then \eqref{eq:rate-covariance-lower-atom} gives, for all sufficiently large
\(T\),
\[
  \mathbb P_{\mathcal M_T}\!\left(\bmM_T^{\mathcal M_T}\in\Gamma\right)
  >
  \Pp\!\left(W_\star^{(1)}\in\Gamma^{\delta_T}\right)+\delta_T.
\]
Thus the defining Prokhorov inequality fails at radius \(\delta_T\), and
\(
  d_{\mathrm P}\!\left(\bmM_T^{\mathcal M_T},W_\star^{(1)}\right)
  \ge\delta_T
  \asymp T^{-\gamma}
\).

It remains to transfer the lower bound to the full SA path.  In the decomposition
\eqref{eq:rate-four-remainder-pathwise}, exact linearity and zero
initialization give \(\bpsi_{0,T}=\bpsi_{1,T}=0\), while the leading process is
\(\bmM_T^{\mathcal M_T}\).  Lemma~\ref{lem:rate-probability-radius-calculus}
together with \eqref{eq:linear-rate-R2-ky-fan} and
\eqref{eq:linear-rate-R3-ky-fan}, used at moment order \(p\), gives
\[
  d_{\mathrm P}\!\left(
    \Phi_T^{\mathcal M_T},\bmM_T^{\mathcal M_T}
  \right)
  \lesssim
  T^{-\rateSm}\sqrt{\log T}
  +T^{-\frac{p/2-\alpha}{p+1}}
  +T^{-\rateOsc}
  =o(T^{-\gamma}).
\]
Here
\(
  \frac{p/2-\alpha}{p+1}
  =\rateOsc+\frac{1-\alpha}{p+1}
  >\rateOsc
\).
Combining this estimate with the leading-process lower bound and the triangle
inequality yields
\[
  d_{\mathrm P}\!\left(
    \Phi_T^{\mathcal M_T},W_\star^{(1)}
  \right) \ge
  d_{\mathrm P}\!\left(
    \bmM_T^{\mathcal M_T},W_\star^{(1)}
  \right)
  -d_{\mathrm P}\!\left(
    \Phi_T^{\mathcal M_T},\bmM_T^{\mathcal M_T}
  \right)
  \gtrsim T^{-\gamma}.
\]
Since \(\gamma=\rateCov\) and
\(\mathcal M_T\in\mathfrak C_{\mathrm{cov}}\), this proves
\eqref{eq:rate-covariance-stabilization-lower}.
\end{proof}

\subsection[Finite-moment lower bound for the linear-SA path]
{Finite-moment lower bound for the linear-SA path}
\label{app:rate-iid-leading-lower}

\begin{proof}[Proof of Theorem~\ref{thm:rate-iid-leading-minimax}]
Fix $M>1$.  Choose constants $K_*>4$ and $\kappa_*>0$, and set
\begin{equation*}
  \varepsilon_T=c_*T^{-\frac{p-2}{2(p+1)}},
  \quad
  q_T=\frac{\kappa_*\varepsilon_T}{T},
  \quad
  u_T=K_*\varepsilon_T\sqrt T,
  \quad
  v_T^2=\frac{1-q_Tu_T^2}{1-q_T},
\end{equation*}
where $c_*>0$ will be chosen below.  For every sufficiently large \(T\), let
\(\mathcal M_T\) be the exact-linear oracle with
\(g_{\mathcal M_T}(x)=x\) and i.i.d.\ noise
\((\epsilon_j^{\mathcal M_T})_{j\ge1}\), adapted to its natural filtration
\((\mathcal F_j^{\mathcal M_T})_{j\ge0}\), such that
\begin{align*}
  \Pp_{\mathcal M_T}(\epsilon_j^{\mathcal M_T}=u_T)
  =\Pp_{\mathcal M_T}(\epsilon_j^{\mathcal M_T}=-u_T)
  &=\frac{q_T}{2},\\
  \Pp_{\mathcal M_T}(\epsilon_j^{\mathcal M_T}=v_T)
  =\Pp_{\mathcal M_T}(\epsilon_j^{\mathcal M_T}=-v_T)
  &=\frac{1-q_T}{2}.
\end{align*}
Then
\(\E_{\mathcal M_T}[\epsilon_j^{\mathcal M_T}]=0\) and
\(\E_{\mathcal M_T}[(\epsilon_j^{\mathcal M_T})^2]=1\).
For all sufficiently large \(T\), \(u_T\ge1\), and hence
\(0<v_T\le1\).  Since \(p>2\), it follows that
\begin{align*}
  \E_{\mathcal M_T}\!\left[|\epsilon_j^{\mathcal M_T}|^p\right]
  &=(1-q_T)v_T^p+q_Tu_T^p
  \le(1-q_T)v_T^2+q_Tu_T^p\\
  &=1-q_Tu_T^2
    +\kappa_*K_*^p\varepsilon_T^{p+1}T^{p/2-1}
  \le1+\kappa_*K_*^pc_*^{p+1}.
\end{align*}
Choosing \(c_*>0\) sufficiently small that
\(\kappa_*K_*^pc_*^{p+1}\le M-1\) ensures
\(\E_{\mathcal M_T}[|\epsilon_j^{\mathcal M_T}|^p]\le M\).

Set \(\mathfrak C_{\mathrm{osc}}=\{\mathcal M_T:T\ge T_0\}\).  Uniformly over
\(\mathcal M\in\mathfrak C_{\mathrm{osc}}\),
\cite[Corollary~2.5]{dereich2019general} verifies
Assumption~\ref{ass:rate-trajectory-stability} for every
\(q\in[1,p/2]\).  Assumptions~\ref{ass:rate-linear-dynamics}--
\ref{ass:rate-covariance-stabilization} follow directly from
\(g_{\mathcal M}(x)=x\) and the defining i.i.d.\ noise conditions.
Let \(\Phi_T^{\mathcal M_T}\) be the corresponding linearly interpolated
linear-SA path.
To simplify the notation, whenever \(\eta_j^{-1}\) appears as a time increment,
summation limit, or discrete index in this argument, it denotes
\(\lfloor\eta_j^{-1}\rfloor\); we omit the integer-part symbol throughout.
Define the Borel set
\begin{equation}
  \mathcal J_T^{\mathrm{rough}}
  \coloneqq
  \left\{x:\
    \max_j
    \left|x\!\left(\frac{j+\eta_j^{-1}}{T}\right)
      -x\!\left(\frac{j}{T}\right)\right|
    \ge \frac{(1-e^{-1})K_*}{2}\varepsilon_T
  \right\},
  \label{eq:rate-iid-lower-separating-set}
\end{equation}
where the maximum is over the integer indices for which both displayed time
points belong to \([0,1]\).

For \(\lceil T/3\rceil\le j\le\lfloor T/2\rfloor\), let
\(\mathcal E_j\) be the event that
\(|\epsilon_j^{\mathcal M_T}|=u_T\) and
\(|\epsilon_i^{\mathcal M_T}|=v_T\) for every \(i\ne j\),
\(1\le i\le T\).
For all sufficiently large \(T\), the displayed time points in
\eqref{eq:rate-iid-lower-separating-set} belong to \([0,1]\).  By the
definition of \(\Phi_T^{\mathcal M_T}\) and the stochastic approximation
 update rule,
\begin{align*}
\Phi_T^{\mathcal M_T}\!\left(\frac{j+\eta_j^{-1}}{T}\right)
  -\Phi_T^{\mathcal M_T}\!\left(\frac{j}{T}\right)
&=
  \frac{\epsilon_j^{\mathcal M_T}\eta_j}{\sqrt T}
  \sum_{r=1}^{\eta_j^{-1}}\prod_{k=j+1}^{j+r}(1-\eta_k)\\
&\qquad+\underbrace{\frac1{\sqrt T}
  \sum_{\substack{1\le i\le j+\eta_j^{-1}\\i\ne j}}
    \left(A_i^{j+\eta_j^{-1}}
      -\mathbf1_{\{i\le j\}}A_i^j\right)
      \epsilon_i^{\mathcal M_T}}_{\eqqcolon R_{j,T}},
\end{align*}
where \(A_i^n\) is the scalar specialization of the coefficient in
\eqref{eq:rate-recursion-coefficients}.
For all sufficiently large \(T\), \(\eta_j<1\) on the indicated range and
\(\eta_i\le\eta_j\) for \(i\ge j\). For the coefficient of
\(\epsilon_j^{\mathcal M_T}/\sqrt{T}\) in the preceding equation, we have
\begin{align*}
  \eta_j\sum_{r=1}^{\eta_j^{-1}}
    \prod_{k=j+1}^{j+r}(1-\eta_k)
  &\ge \eta_j\sum_{r=1}^{\eta_j^{-1}}(1-\eta_j)^r\\
  &=(1-\eta_j)\{1-(1-\eta_j)^{\eta_j^{-1}}\}
  \ge \frac{3(1-e^{-1})}{4}.
\end{align*}
Since \(|\epsilon_j^{\mathcal M_T}|=u_T=K_*\varepsilon_T\sqrt T\) on
\(\mathcal E_j\), we have
\begin{equation}
  \left|
    \Phi_T^{\mathcal M_T}\!\left(\frac{j+\eta_j^{-1}}{T}\right)
    -\Phi_T^{\mathcal M_T}\!\left(\frac{j}{T}\right)
  \right|
  \ge \frac{3(1-e^{-1})K_*}{4}\varepsilon_T-|R_{j,T}|.
  \label{eq:rate-iid-lower-signal-remainder}
\end{equation}

Conditional on \(\mathcal E_j\), the variables
\((\epsilon_i^{\mathcal M_T})_{i\ne j}\) remain independent and symmetric,
with common absolute value \(v_T\le1\).  We first make the coefficients in \(R_{j,T}\)
explicit.  For \(i\le j\), the definition of \(A_i^n\) gives
\[
  A_i^{j+\eta_j^{-1}}-A_i^j
  =\eta_i\left\{\prod_{k=i+1}^{j}(1-\eta_k)\right\}
    \sum_{t=j+1}^{j+\eta_j^{-1}}
      \prod_{k=j+1}^{t}(1-\eta_k),
\]
whereas, for \(j<i\le j+\eta_j^{-1}\), the coefficient is simply
\(A_i^{j+\eta_j^{-1}}\).  Applying \eqref{eq:rate-contraction-sums}--
\eqref{eq:rate-deterministic-product-bounds} yields
\[
  \sum_{t=j+1}^{j+\eta_j^{-1}}
    \prod_{k=j+1}^{t}(1-\eta_k)
    \le C\eta_j^{-1},
  \qquad
  \sum_{i=1}^{j-1}\eta_i^2
    \prod_{k=i+1}^{j}(1-\eta_k)^2
    \le C\eta_j,
\]
and \(\sup_{n\ge i}A_i^n\le C\).
Since the conditional means of
\((\epsilon_i^{\mathcal M_T})_{i\ne j}\) are zero, all
cross terms vanish, and hence
\begin{align*}
  \E_{\mathcal M_T}\!\left[R_{j,T}^2\mid\mathcal E_j\right]
  &=\frac{v_T^2}{T}
    \sum_{\substack{1\le i\le j+\eta_j^{-1}\\i\ne j}}
    \left(A_i^{j+\eta_j^{-1}}
      -\mathbf1_{\{i\le j\}}A_i^j\right)^2\\
  &\le \frac{C}{T}\left\{
      \eta_j^{-2}\sum_{i=1}^{j-1}\eta_i^2
        \prod_{k=i+1}^{j}(1-\eta_k)^2
      +\sum_{i=j+1}^{j+\eta_j^{-1}}
        \big(A_i^{j+\eta_j^{-1}}\big)^2
    \right\}
  \le \frac{C}{T\eta_j}.
\end{align*}
Note that the threshold in \eqref{eq:rate-iid-lower-separating-set} and
\eqref{eq:rate-iid-lower-signal-remainder} imply
\[
  \{\Phi_T^{\mathcal M_T}\notin\mathcal J_T^{\mathrm{rough}}\}
    \cap\mathcal E_j
  \subseteq
  \left\{|R_{j,T}|\ge
    \frac{(1-e^{-1})K_*}{4}\varepsilon_T\right\}
    \cap\mathcal E_j.
\]
Since \(\eta_j\asymp T^{-\alpha}\) in the indicated range and
\(\rateOsc<\rateSm\), Chebyshev's inequality now yields, uniformly in these
indices,
\begin{align*}
  \Pp_{\mathcal M_T}\!\left(
    \Phi_T^{\mathcal M_T}\notin\mathcal J_T^{\mathrm{rough}}
    \mid\mathcal E_j\right)
  &\le
    \Pp_{\mathcal M_T}\!\left(
      |R_{j,T}|\ge\frac{(1-e^{-1})K_*}{4}\varepsilon_T
      \mathrel{\Big|}\mathcal E_j\right)\\
  &\le \frac{16}{(1-e^{-1})^2K_*^2\varepsilon_T^2}
    \E_{\mathcal M_T}\!\left[R_{j,T}^2\mid\mathcal E_j\right]
  \le \frac{C}{T\eta_j\varepsilon_T^2}=o(1).
\end{align*}
Combine this with
the fact that the events \(\mathcal E_j\) are disjoint.  Consequently,
\begin{align}
  \Pp_{\mathcal M_T}\!\left(
    \Phi_T^{\mathcal M_T}\in\mathcal J_T^{\mathrm{rough}}\right)
  &\ge (1-o(1))
    \sum_{j=\lceil T/3\rceil}^{\lfloor T/2\rfloor}
      \Pp_{\mathcal M_T}(\mathcal E_j)
  \asymp Tq_T(1-q_T)^{T-1}
   \asymp\varepsilon_T.
  \label{eq:rate-iid-lower-source-probability}
\end{align}

We next control the same set under Wiener measure.  Put
\(\delta_T=c_0\varepsilon_T\), where \(c_0>0\) will be chosen sufficiently
small.  If \(y\in(\mathcal J_T^{\mathrm{rough}})^{\delta_T}\), then according
to Definition~\ref{def:rate-prokhorov}, there are
\(x\in\mathcal J_T^{\mathrm{rough}}\) and \(\lambda\in\Lambda\) such that
\begin{equation}
  \sup_{s<t}
  \left|\log\frac{\lambda(t)-\lambda(s)}{t-s}\right|<\delta_T,
  \qquad
  \sup_r|x(\lambda(r))-y(r)|<\delta_T.
  \label{eq:log-slope}
\end{equation}
For an index \(j\) realizing the event in
\eqref{eq:rate-iid-lower-separating-set}, set
\(
  r_0=\lambda^{-1}\!\left(\frac{j}{T}\right),
  \,
  r_1=\lambda^{-1}\!\left(\frac{j+\eta_j^{-1}}{T}\right)
\).
Equation~\eqref{eq:log-slope} gives
\[
  0<r_1-r_0
  \le e^{\delta_T}\frac{\eta_j^{-1}}{T}
  \le CT^{\alpha-1},
  \qquad
  |y(r_1)-y(r_0)|
  \ge\left\{\frac{(1-e^{-1})K_*}{2}-2c_0\right\}\varepsilon_T.
\]
Choose \(c_0<(1-e^{-1})K_*/8\).  Lemma~\ref{lem:rate-brownian-modulus}
then gives
\begin{align}
  \Pp\!\left(W_\star^{(1)}\in
    (\mathcal J_T^{\mathrm{rough}})^{\delta_T}\right)
  &\le CT^{1-\alpha}
    \exp\{-cT^{1-\alpha}\varepsilon_T^2\}\notag\\
  &=CT^{1-\alpha}
    \exp\{-cT^{2(\rateSm-\rateOsc)}\}
   =o(\varepsilon_T).\notag
\end{align}
Hence this implies that, for all
sufficiently large \(T\),
\[
  \Pp_{\mathcal M_T}\!\left(
    \Phi_T^{\mathcal M_T}\in\mathcal J_T^{\mathrm{rough}}\right)
  >
  \Pp\!\left(W_\star^{(1)}\in
    (\mathcal J_T^{\mathrm{rough}})^{\delta_T}\right)+\delta_T.
\]
Thus the defining Prokhorov inequality fails at radius \(\delta_T\), and
\[
  d_{\mathrm P}(\Phi_T^{\mathcal M_T},W_\star^{(1)})
  \ge c_0\varepsilon_T
  =c_0c_*T^{-\frac{p-2}{2(p+1)}}.
\]
Because \(\mathcal M_T\in\mathfrak C_{\mathrm{osc}}\), the last display proves
\eqref{eq:linear-iid-leading-minimax}.
\end{proof}

\subsection{Sharp intrinsic rate in the scalar Gaussian model}
\label{app:linear-gaussian-j1-sharpness}

\begin{proof}[Proof of Theorem~\ref{thm:linear-gaussian-j1-sharpness}]
Let \(\mathcal M_{\mathrm{sm}}\) be the oracle defining
\(\mathfrak C_{\mathrm{sm}}\), and write
\(\Phi_T=\Phi_T^{\mathcal M_{\mathrm{sm}}}\).
For this oracle, \(G_{\mathcal M_{\mathrm{sm}}}
=S_{\mathcal M_{\mathrm{sm}}}=1\), and hence
\(Z^{\mathcal M_{\mathrm{sm}}}=W_\star^{(1)}\).
Since \(g(x_t)=\Delta_t\), the scalar
recursion is
\begin{equation}
  \Delta_{t+1}=(1-\eta_t)\Delta_t+\eta_t\epsilon_t,
  \qquad \Delta_1=0.
  \label{eq:linear-gaussian-direct-recursion}
\end{equation}
Iterating this recursion gives the exact filtered-Gaussian representation
\[
  \Delta_{t+1}
  =\sum_{j=1}^t\eta_j
    \left\{\prod_{i=j+1}^t(1-\eta_i)\right\}\epsilon_j,
  \qquad t\ge1,
\]
where an empty product equals one.  Consequently, at every grid point,
\begin{equation}
\begin{split}
  \Phi_T(k/T)
  &=\frac1{\sqrt T}\sum_{t=1}^k\Delta_{t+1}\\
  &=\frac1{\sqrt T}\sum_{j=1}^k
    \left[
      \eta_j\sum_{t=j}^k\prod_{i=j+1}^t(1-\eta_i)
    \right]\epsilon_j,
  \qquad 0\le k\le T.
\end{split}
  \label{eq:linear-gaussian-filtered-path}
\end{equation}
Thus every finite collection of grid values and grid increments of
\(\Phi_T\) is jointly centered Gaussian.

We next establish the variance scale of the iterates.  Since \(\Delta_t\)
is \(\cF_{t-1}\)-measurable and \(\epsilon_t\) is independent of
\(\cF_{t-1}\), centered, and has variance one,
\eqref{eq:linear-gaussian-direct-recursion} gives
\[
  \E\!\left[\Delta_{t+1}^2\right]
  =(1-\eta_t)^2\E\!\left[\Delta_t^2\right]+\eta_t^2.
\]
For the power-law step size,
\[
  \frac{1-\eta_{t+1}/\eta_t}{\eta_t}
  =\frac{1-\{t/(t+1)\}^{\alpha}}{\eta t^{-\alpha}}
  =O(t^{\alpha-1})\longrightarrow0.
\]
Choose \(t_0\) sufficiently large that, for every \(t\ge t_0\),
\[
  0<\eta_t\le\frac12,
  \qquad
  \eta_{t+1}\ge\eta_t(1-\eta_t/2).
\]
Choose \(C_2\ge2\) sufficiently large that
\(\E[\Delta_{t_0}^2]\le C_2\eta_{t_0}\).  If inductively
\(\E[\Delta_t^2]\le C_2\eta_t\), then
\begin{align*}
  \E\!\left[\Delta_{t+1}^2\right]
  &\le C_2\eta_t(1-\eta_t)^2+\eta_t^2\\
  &=C_2\eta_t
    \left\{(1-\eta_t)^2+\frac{\eta_t}{C_2}\right\}
  \le C_2\eta_t(1-\eta_t/2)
  \le C_2\eta_{t+1}.
\end{align*}
The penultimate inequality uses \(0<\eta_t\le1/2\) and
\(C_2\ge2\).  Induction therefore proves that
\(\E[\Delta_t^2]\le C_2\eta_t\) for every \(t\ge t_0\).  In
particular, for all sufficiently large \(T\),
\begin{equation}
  \sup_{T/6\le t\le5T/6}\E\!\left[\Delta_t^2\right]
  \le CT^{-\alpha}.
  \label{eq:linear-gaussian-interior-variance}
\end{equation}

Fix the nested intervals
\(
  I_0=[1/3,2/3],\,
  I_1=[1/4,3/4]
\)
and constants \(0<u<v<1\).  Let \(c_0>0\), to be chosen below, and define
\(
  h_T=c_0T^{\alpha-1}.
\)
Choose \(0<c_1<(v-u)/2\) and put
\[
  \delta_T= c_1 \big(2h_T\log (1/h_T)\big)^{1/2}.
\]
Because \(h_T\downarrow0\) and \(h_T\log(1/h_T)\to0\), we have
\(\delta_T\to0\).

We next show that the stochastic-approximation path has small oscillations
at the scale \(h_T\).  For all sufficiently large \(T\), suppose that
\(j/T,k/T\in[1/5,4/5]\) and \(0<k-j\le3c_0T^\alpha\).  By the definition of
\(\Phi_T\), Minkowski's inequality in \(L^2\), and
\eqref{eq:linear-gaussian-interior-variance},
\begin{align*}
  \operatorname{Var}^{1/2}
  \!\left\{\Phi_T(k/T)-\Phi_T(j/T)\right\}
  &=\left\|
    \frac1{\sqrt T}\sum_{t=j+1}^k\Delta_{t+1}
  \right\|_{L^2}\\
  &\le\frac1{\sqrt T}\sum_{t=j+1}^k
    \left\|\Delta_{t+1}\right\|_{L^2}
  \le CT^{-1/2}(k-j)T^{-\alpha/2}.
\end{align*}
After squaring and using \(k-j\le3c_0T^\alpha\), we obtain, for a constant
\(C_0\) independent of \(c_0,T,j,k\),
\begin{equation}
  \operatorname{Var}
  \!\left\{\Phi_T(k/T)-\Phi_T(j/T)\right\}
  \le C_0c_0^2T^{\alpha-1}=C_0c_0h_T.
  \label{eq:linear-gaussian-grid-increment-variance}
\end{equation}
Further, consider two time points \(t,s \in I_1\),
 owing to the convexity of linear interpolation, we have
\[
  |\Phi_T(t)-\Phi_T(s)|
  \le\max_{a,b\in\{0,1\}}
  \left|\Phi_T\{(\lfloor t T \rfloor +a)/T\}
  -\Phi_T\{(\lfloor s T \rfloor +b)/T\}\right|.
\]
If additionally \(|t-s|\le e^{\delta_T}h_T\), then, for all sufficiently
large \(T\), \(e^{\delta_T}\le2\), \(1/T<1/20\), and
\(2\le c_0T^\alpha\).  Therefore all four grid endpoints lie in
\([1/5,4/5]\), and
\[
  |(\lfloor tT\rfloor +a)-(\lfloor sT\rfloor +b)|
  \le T|t-s|+2
  \le2c_0T^\alpha+2
  \le3c_0T^\alpha.
\]
There are at most \(T+1\) choices for the first endpoint and at most
\(6c_0T^\alpha+3\) choices for the second one.  Thus the number of ordered
grid pairs that can occur is at most \(CT^{1+\alpha}\). Thus at most \(CT^{1+\alpha}\) instances of the bound \eqref{eq:linear-gaussian-grid-increment-variance} suffice to cover every pair
\(t,s \in I_1\) with \(|t-s|\le e^{\delta_T} h_T\) used
by the interpolation bound.

Every such grid increment is centered Gaussian by
\eqref{eq:linear-gaussian-filtered-path}, and its variance is bounded by
\eqref{eq:linear-gaussian-grid-increment-variance}.  The two-sided Gaussian
tail bound and a union bound therefore yield
\begin{equation}
\begin{split}
  &\Pp\!\left(
    \sup_{\substack{s,t\in I_1\\|t-s|\le e^{\delta_T}h_T}}
    |\Phi_T(t)-\Phi_T(s)|>u\big(2h_T\log (1/h_T)\big)^{1/2}
  \right)\\
  &\qquad\le
  2CT^{1+\alpha}
  \exp\left\{-\frac{u^2h_T\log (1/h_T)}{C_0c_0h_T}\right\}\\
  &\qquad=2CT^{1+\alpha}
  h_T^{\frac{u^2}{C_0c_0}} = 2Cc_0^{\frac{u^2}{C_0c_0}}T^{(1+\alpha)
  - \frac{(1-\alpha)u^2}{C_0c_0}} \longrightarrow 0.
\end{split}
  \label{eq:linear-gaussian-sa-window-tail}
\end{equation}
Here the last limit holds 
as long as we choose \(c_0\in(0,1]\) sufficiently small that
\(u^2(1-\alpha)/(C_0c_0)>1+\alpha\).
Now, define the testing event
\[
  \mathcal J_T^{\mathrm{smooth}}
  =\left\{
    x\in C([0,1],\R):
    \sup_{\substack{s,t\in I_1\\|t-s|\le e^{\delta_T}h_T}}
    |x(t)-x(s)|\le \sqrt{2u^2h_T\log(1/h_T)}
  \right\}.
\]
It is a Borel subset of \(D([0,1],\R)\). From 
\eqref{eq:linear-gaussian-sa-window-tail} we already know that
\begin{equation}
  \Pp(\Phi_T\in\mathcal J_T^{\mathrm{smooth}})\longrightarrow1.
  \label{eq:linear-gaussian-sa-window-small}
\end{equation}

We next show explicitly that Brownian motion has larger oscillations at the
same scale.  Let
\(
  m_T=\left\lfloor\frac1{3h_T}\right\rfloor\), and consider the following
  \(m_T + 1\) points
\[  
  r_{\ell,T}=\frac13+\ell h_T,
  \quad 0\le\ell\le m_T.
\]
All these points lie in \(I_0\), and \(m_T\ge(6h_T)^{-1}\) for all
sufficiently large \(T\).  The \(m_T\) increments
\[
  W_\star^{(1)}(r_{\ell,T})-W_\star^{(1)}(r_{\ell-1,T}),
  \qquad 1\le\ell\le m_T,
\]
are independent and have distribution \(N(0,h_T)\). 
Consequently, for \(Z\sim N(0,1)\),
\begin{align*}
  &\Pp \Bigg(
    \sup_{\substack{s,t\in I_0\\|t-s|\le h_T}}
    |W_\star^{(1)}(t)-W_\star^{(1)}(s)|
    \le \sqrt{2v^2h_T\log(1/h_T)}
  \Bigg)\\
  &\qquad\le 
  \Pp \left(\sup_{1\le \ell \le m_T}\big|
  W_\star^{(1)}(r_{\ell,T})-W_\star^{(1)}(r_{\ell-1,T})
  \big| \le \sqrt{2v^2h_T\log(1/h_T)}\right)\\
  &\qquad= \prod_{\ell = 1}^{m_T}\Pp\left(
    \big|
  W_\star^{(1)}(r_{\ell,T})-W_\star^{(1)}(r_{\ell-1,T})
  \big| \le \sqrt{2v^2h_T\log(1/h_T)}
  \right)
  \\
  &\qquad\le
  \left[
    1-\Pp \left(|Z|>v\frac{\sqrt{2h_T\log (1/h_T)}}{\sqrt{h_T}}\right)
  \right]^{m_T}
  =\left[
    1-\Pp \left(|Z|>v\sqrt{2\log (1/h_T)}\right)
  \right]^{m_T}.
\end{align*}
The Mills lower bound
\(\Pp(|Z|>x)\ge cx^{-1}e^{-x^2/2}\) for \(x\ge1\) gives, for all
sufficiently large \(T\),
\[
  \Pp\!\left(|Z|>v\sqrt{2\log(1/h_T)}\right)
  \ge c\log^{-\frac{1}{2}}(1/h_T)e^{-v^2\log(1/h_T)}
  =c\frac{h_T^{v^2}}{\sqrt{\log(1/h_T)}}.
\]
Using \(1-z\le e^{-z}\) and \(m_T\ge(6h_T)^{-1}\), we conclude that
\begin{equation}
\begin{split}
  &\Pp\!\left(
    \sup_{\substack{s,t\in I_0\\|t-s|\le h_T}}
    |W_\star^{(1)}(t)-W_\star^{(1)}(s)|
    \le \sqrt{2v^2h_T\log(1/h_T)}
  \right)\\
  &\qquad\le \left(1 - \frac{ch_T^{v^2}}{\sqrt{\log(1/h_T)}}
  \right)^{m_T}
  \le
  \exp\left\{-\frac{c h_T^{v^2-1}}{6\sqrt{\log(1/h_T)}}\right\}
  \longrightarrow0.
\end{split}
  \label{eq:linear-gaussian-brownian-window-large}
\end{equation}
The last limit is explicit since \(v<1\).

It remains to transfer this probability separation through the \(J_1\)
Skorokhod metric.  Suppose that \(x\in\mathcal J_T^{\mathrm{smooth}}\) and
\(d_{\mathrm S}(x,y)<\delta_T\).  By~\eqref{eq:rate-j1-metric}, there is a
time change \(\lambda\in\Lambda\) such that
\[
  \sup_{0\le s<t\le1}
  \left|\log\frac{\lambda(t)-\lambda(s)}{t-s}\right|<\delta_T,
  \qquad
  \sup_{0\le r\le1}|x\{\lambda(r)\}-y(r)|<\delta_T.
\]
It is not hard to compute from
the first inequality that, for every \(s<t\),
\begin{equation}
  e^{-\delta_T}(t-s)<\lambda(t)-\lambda(s)
  <e^{\delta_T}(t-s),
  \qquad
  \sup_{0\le r\le1}|\lambda(r)-r|
  \le e^{\delta_T}-1\le2\delta_T.
  \label{eq:linear-gaussian-time-change-control}
\end{equation}
 Since
\(\operatorname{dist}(I_0,I_1^c)=1/12\) and \(\delta_T\to0\),
\eqref{eq:linear-gaussian-time-change-control} gives
\(\lambda(I_0)\subset I_1\) for all sufficiently large \(T\).

Now consider a path \(y\in (\mathcal J_T^{\mathrm{smooth}})^{\delta_T}\),
take \(s,t\in I_0\) with \(|t-s|\le h_T\).  Then
\(|\lambda(t)-\lambda(s)|\le e^{\delta_T}h_T\), and hence
there is a path \(x \in \mathcal J_T^{\mathrm{smooth}}\) satisfying
\begin{align*}
  |y(t)-y(s)|
  &\le |y(t)-x\{\lambda(t)\}|
    +|x\{\lambda(t)\}-x\{\lambda(s)\}|
  +|x\{\lambda(s)\}-y(s)|\\
  &<2\delta_T+ \sqrt{2u^2 h_T\log(1/h_T)}\\
  &=(u+2c_1)\sqrt{2h_T\log (1/h_T)}<v \sqrt{2h_T\log (1/h_T)}.
\end{align*}
We thus have proved the pathwise inclusion
\begin{equation}
  (\mathcal J_T^{\mathrm{smooth}})^{\delta_T}
  \subset
  \left\{y:
    \sup_{\substack{s,t\in I_0\\|t-s|\le h_T}}
    |y(t)-y(s)|< \sqrt{2v^2h_T\log(1/h_T)}
  \right\}.
  \label{eq:linear-gaussian-j1-neighborhood}
\end{equation}

Equations~\eqref{eq:linear-gaussian-sa-window-small},
\eqref{eq:linear-gaussian-brownian-window-large}, and
\eqref{eq:linear-gaussian-j1-neighborhood} show that
\[
  \Pp(\Phi_T\in\mathcal J_T^{\mathrm{smooth}})\longrightarrow1,
  \qquad
  \Pp(W_\star^{(1)}\in(\mathcal J_T^{\mathrm{smooth}})^{\delta_T})\longrightarrow0.
\]
Because also \(\delta_T\to0\), for all sufficiently large \(T\),
\[
  \Pp(\Phi_T\in\mathcal J_T^{\mathrm{smooth}})
  >\Pp(W_\star^{(1)}\in(\mathcal J_T^{\mathrm{smooth}})^{\delta_T})+\delta_T.
\]
Therefore, by Definition~\ref{def:rate-prokhorov},
\[
  d_{\mathrm P}(\Phi_T,W_\star^{(1)})
  \ge\delta_T=c_1\sqrt{2h_T\log(1/h_T)}
  \asymp T^{-\frac{1-\alpha}{2}}\sqrt{\log T}.
\]
Since \(\mathfrak C_{\mathrm{sm}}=\{\mathcal M_{\mathrm{sm}}\}\), taking the
supremum over \(\mathfrak C_{\mathrm{sm}}\) proves
\eqref{eq:linear-gaussian-j1-order} and completes the proof.
\end{proof}

\subsection{Nonlinear obstruction under a trajectory moment bound}
\label{app:rate-nonlinear-oracle-lower}

\begin{proof}[Proof of Theorem~\ref{thm:rate-nonlinear-oracle-lower}]
\noindent\textit{Step 1: Construction.}
For every integer \(T\ge1\), introduce an oracle \(\mathcal M_T\) and put
\[
  \rho_T\coloneqq T^{-\frac{(\alpha-1/2)q}{q+1}}.
\]
Choose \(\lambda>0\) such that \(\lambda\eta\le1/4\).  Let
\(\chi:\R\to[0,1]\) be smooth, with
\(\chi(z)=0\) for \(z\le1\) and \(\chi(z)=1\) for \(z\ge2\).
For a constant \(c_0>0\), to be fixed below, define
\[
  g_{\mathcal M_T}(x)
  \coloneqq
  \begin{cases}
    \lambda x, & x\le0,\\
    \left[
      \lambda-\{\lambda-c_0T^{\alpha-1}\}
      \chi\!\left(\dfrac{\sqrt T\,x}{\rho_T^{3/2}}\right)
      \chi(3-x)
    \right]x, & x>0.
  \end{cases}
\]
After increasing a fixed lower index \(T_0\), we have
\(0<c_0T^{\alpha-1}\le\lambda\) and \(\rho_T\le1/2\) for every
\(T\ge T_0\).  Each \(g_{\mathcal M_T}\) is differentiable, is exactly
\(\lambda x\) on
\((-\infty,\rho_T^{3/2}/\sqrt T]\), and satisfies
\(xg_{\mathcal M_T}(x)>0\) for \(x\ne0\).  Hence \(g_{\mathcal M_T}\) satisfies
Assumption~\ref{ass:rate-linear-dynamics} with
\(G_{\mathcal M_T}=\lambda\): its quadratic
remainder obeys
\[
  |g_{\mathcal M_T}(x)-\lambda x|
  \le \lambda\frac{\sqrt T}{\rho_T^{3/2}}x^2,
  \qquad x\in\R.
\]
The coefficient in this quadratic bound is finite for each \(T\), although it
is not uniform in \(T\).  Moreover,
\(\sup_{x\in\R}|g_{\mathcal M_T}(x)-\lambda x|\le2\lambda\), because the second cutoff
vanishes for \(x\ge2\). Alternatively, the remainder is zero on a local ball of radius $\rho_T^{3/2}/\sqrt T$, but that radius shrinks with $T$. Neither formulation supplies both a fixed positive local radius and a uniform quadratic-remainder constant.

For each \(T\ge T_0\), let \(\omega_T\) be Bernoulli with success
probability \(\rho_T\) under \(\mathbb P_{\mathcal M_T}\).
Let \(\mathcal F_0^{\mathcal M_T}\) be trivial, put
\(\mathcal F_t^{\mathcal M_T}=\sigma(\omega_T)\) for \(t\ge1\), and define
\[
  \epsilon_1^{\mathcal M_T}
  \coloneqq
  \frac{\rho_T}{\eta\sqrt T(1-\rho_T)}(\omega_T-\rho_T),
  \qquad
  \epsilon_t^{\mathcal M_T}\coloneqq0\quad(t\ge2).
\]
These choices specify the stochastic oracle
\(\mathcal M_T=\big(g_{\mathcal M_T},
(\epsilon_t^{\mathcal M_T})_{t\ge1},
(\mathcal F_t^{\mathcal M_T})_{t\ge0},
\mathbb P_{\mathcal M_T}\big)\).

Write \(\Delta_t^{\mathcal M_T}=x_t^{\mathcal M_T}-x^\star\).  The recursion in the
theorem, with \(\eta_t=\eta t^{-\alpha}\), becomes
\[
  \Delta_{t+1}^{\mathcal M_T}
  =\Delta_t^{\mathcal M_T}
   -\eta_tg_{\mathcal M_T}(\Delta_t^{\mathcal M_T})
   +\eta_t\epsilon_t^{\mathcal M_T}.
\]

The two possible values of the first update are
\[
  \Delta_2^{\mathcal M_T}
  =
  \begin{cases}
    \rho_T/\sqrt T, & \omega_T=1,\\[2pt]
    -\rho_T^2/\{\sqrt T(1-\rho_T)\}, & \omega_T=0.
  \end{cases}
\]
On \(\{\omega_T=0\}\), the trajectory remains negative and follows the
exact-linear recursion.  On \(\{\omega_T=1\}\),
\(\Delta_2^{\mathcal M_T}=\rho_T/\sqrt T\).  Choose \(c_0\) sufficiently small and
then increase \(T_0\) so that, for some \(\kappa>0\),
\[
  \inf_{T\ge T_0}
  \prod_{s=2}^{T-1}(1-c_0T^{\alpha-1}\eta_s)
  \ge\kappa,
  \qquad
  2\rho_T^{1/2}<\kappa.
\]
Indeed, the first inequality follows from
\(c_0T^{\alpha-1}\sum_{s=2}^{T-1}\eta_s\le Cc_0\) and
\(\log(1-z)\ge-2z\) for \(0\le z\le1/2\).

We now show inductively that, on \(\{\omega_T=1\}\),
\[
  \Delta_t^{\mathcal M_T}
  =
  \frac{\rho_T}{\sqrt T}
  \prod_{s=2}^{t-1}(1-c_0T^{\alpha-1}\eta_s),
  \qquad
  \frac{\kappa\rho_T}{\sqrt T}
  \le\Delta_t^{\mathcal M_T}
  \le\frac{\rho_T}{\sqrt T},
  \qquad 2\le t\le T.
\]
The claim is immediate for \(t=2\).  Suppose that it holds at some
\(t<T\).  Then
\[
  \Delta_t^{\mathcal M_T}
  \ge \frac{\kappa\rho_T}{\sqrt T}
  > \frac{2\rho_T^{3/2}}{\sqrt T},
\]
while \(\Delta_t^{\mathcal M_T}\le\rho_T/\sqrt T\le1\).  Thus both cutoffs equal
one, and the definition of \(g_{\mathcal M_T}\) gives
\(g_{\mathcal M_T}(\Delta_t^{\mathcal M_T})
=c_0T^{\alpha-1}\Delta_t^{\mathcal M_T}\).  Since
\(\epsilon_t^{\mathcal M_T}=0\), the recursion yields
\[
  \Delta_{t+1}^{\mathcal M_T}
  =(1-c_0T^{\alpha-1}\eta_t)\Delta_t^{\mathcal M_T}
  =\frac{\rho_T}{\sqrt T}
   \prod_{s=2}^{t}(1-c_0T^{\alpha-1}\eta_s).
\]
Every factor lies in \((0,1]\), so this partial product lies between the
full product and \(1\), and hence between \(\kappa\) and \(1\).  The same
bounds therefore hold at time \(t+1\), which completes the induction.

Collecting these oracles over all construction horizons gives the fixed class
\[
  \mathfrak C_{\mathrm{nl}}
  \coloneqq \{\mathcal M_T:T\ge T_0\}.
\]

\noindent\textit{Step 2: Verification of the assumptions.}
The noise increments are martingale differences, and, uniformly in \(T\),
\[
  \E_{\mathcal M_T}\!\left[|\epsilon_1^{\mathcal M_T}|^p\right]
  \le C\rho_T\left(\frac{\rho_T}{\sqrt T}\right)^p
  \le C,
  \qquad
  \E_{\mathcal M_T}\!\left[
    |\epsilon_t^{\mathcal M_T}|^p
    \mid\mathcal F_{t-1}^{\mathcal M_T}\right]=0
  \quad(t\ge2).
\]
Moreover,
\[
  \E_{\mathcal M_T}\!\left[(\epsilon_1^{\mathcal M_T})^2\right]
  =\frac{\rho_T^3}{\eta^2T(1-\rho_T)}\le C.
\]
Their limiting variance is \(S_{\mathcal M_T}=0\).  For every integer
\(H\ge2\),
\[
  \E_{\mathcal M_T}\!\left[
    \max_{0\le n\le H}
    \left|
      \frac1H\sum_{t=1}^n
      \E_{\mathcal M_T}\!\left[
        (\epsilon_t^{\mathcal M_T})^2
        \mid\mathcal F_{t-1}^{\mathcal M_T}\right]
    \right|
  \right]
  =
  \frac1H\E_{\mathcal M_T}\!\left[(\epsilon_1^{\mathcal M_T})^2\right]
  \le CH^{-1}\le CH^{-3/4}.
\]
Thus Assumptions~\ref{ass:rate-martingale-noise} and
\ref{ass:rate-covariance-stabilization} hold with constants independent of
\(T\), limiting variance \(0\), and \(\gamma=1/4\).
Together with the properties of \(g_{\mathcal M_T}\) verified in Step~1, this establishes
Assumptions~\ref{ass:rate-linear-dynamics},
\ref{ass:rate-martingale-noise}, and
\ref{ass:rate-covariance-stabilization} for the constructed oracle.

It remains to verify
Assumption~\ref{ass:rate-trajectory-stability}.

First consider \(2\le t\le T\).  The induction in Step~1 gives the first
bound below.  For the second, use the displayed value of
\(\Delta_2^{\mathcal M_T}\), the identity
\(g_{\mathcal M_T}(x)=\lambda x\) for \(x\le0\), and
\(0<1-\lambda\eta_s\le1\):
\[
  |\Delta_t^{\mathcal M_T}|
  \le \frac{\rho_T}{\sqrt T}
  \quad\text{on }\{\omega_T=1\},
  \qquad
  |\Delta_t^{\mathcal M_T}|
  \le \frac{\rho_T^2}{\sqrt T(1-\rho_T)}
  \quad\text{on }\{\omega_T=0\}.
\]
Using \(\mathbb P_{\mathcal M_T}(\omega_T=1)=\rho_T\), \(\rho_T\le1/2\), and
\(q\ge1\), we obtain
\begin{align*}
  \E_{\mathcal M_T}\!\left[|\Delta_t^{\mathcal M_T}|^{2q}\right]
  &\le
  \rho_T\left(\frac{\rho_T}{\sqrt T}\right)^{2q}
  +(1-\rho_T)
   \left(\frac{\rho_T^2}{\sqrt T(1-\rho_T)}\right)^{2q}\\
  &\le C\rho_T^{2q+1}T^{-q}
  \le CT^{-\alpha q}
  \le Ct^{-\alpha q}.
\end{align*}
Here \(\rho_T^{2q+1}T^{-q}\le T^{-\alpha q}\).

Now let \(t>T\).  Since the noise vanishes after the first update, the
recursion gives
\begin{align*}
  |\Delta_t^{\mathcal M_T}|
  &\le \frac{\rho_T}{\sqrt T}
  \prod_{s=T}^{t-1}(1-c_0T^{\alpha-1}\eta_s)
  \le \frac{\rho_T}{\sqrt T}
  \exp\!\left[-c\{(t/T)^{1-\alpha}-1\}\right]
  &&\text{on }\{\omega_T=1\},\\
  |\Delta_t^{\mathcal M_T}|
  &\le \frac{2\rho_T^2}{\sqrt T}
  \prod_{s=T}^{t-1}(1-\lambda\eta_s)
  \le \frac{2\rho_T^2}{\sqrt T}
  \exp\!\left[-c\{(t/T)^{1-\alpha}-1\}\right]
  &&\text{on }\{\omega_T=0\}.
\end{align*}
The first line uses positivity and
\(g_{\mathcal M_T}(x)/x\ge c_0T^{\alpha-1}\) for \(x>0\). The second uses
\(g_{\mathcal M_T}(x)=\lambda x\) for \(x<0\) and
\(|\Delta_T^{\mathcal M_T}|\le2\rho_T^2/\sqrt T\).  Both lines use
\(
  T^{\alpha-1}\sum_{s=T}^{t-1}s^{-\alpha}
  \ge c\{(t/T)^{1-\alpha}-1\}
\).
Combining these two bounds with
\(\mathbb P_{\mathcal M_T}(\omega_T=1)=\rho_T\) gives
\[
  \E_{\mathcal M_T}\!\left[|\Delta_t^{\mathcal M_T}|^{2q}\right]
  \le
  CT^{-\alpha q}
  \exp\!\left[-2cq\{(t/T)^{1-\alpha}-1\}\right].
\]
Since
\(
  \sup_{y\ge1}y^{\alpha q}
  \exp[-2cq\{y^{1-\alpha}-1\}]<\infty
\), the last display is at most
\(CT^{-\alpha q}(T/t)^{\alpha q}=Ct^{-\alpha q}\).
The trajectory vanishes at \(t=1\).  Since
\(\eta_t^q=\eta^qt^{-\alpha q}\), this proves
Assumption~\ref{ass:rate-trajectory-stability}.

\noindent\textit{Step 3: Lower bound.}
At the horizon \(T\), on the event \(\{\omega_T=1\}\), all iterates remain
positive and
\[
  \Phi_T^{\mathcal M_T}(1)
  =\frac1{\sqrt T}\sum_{t=1}^T\Delta_{t+1}^{\mathcal M_T}
  \ge\frac{(T-1)\kappa\rho_T}{T}
  \ge\frac\kappa2\rho_T
\]
for all sufficiently large \(T\).  Because every time change in
\eqref{eq:rate-j1-metric} is onto,
\(d_{\mathrm S}(x,0)=\tnorm{x}\).  Let
\[
  \Gamma_T\coloneqq
  \{x:d_{\mathrm S}(x,0)\ge\kappa\rho_T/2\}.
\]
This is a closed, hence Borel, test set.  Moreover,
\(\mathbb P_{\mathcal M_T}
(\Phi_T^{\mathcal M_T}\in\Gamma_T)\ge\rho_T\).  If
\(0<\delta<c_1\rho_T\), where
\(c_1<\min\{\kappa/2,1\}\), the zero path does not belong to
\(\Gamma_T^\delta\), while \(\rho_T>\delta\).  The defining Prokhorov
inequality in \eqref{eq:rate-prokhorov-definition} therefore fails for
\(\delta\).  Hence
\[
  d_{\mathrm P}(\Phi_T^{\mathcal M_T},0)
  \ge c_1\rho_T
  =c_1T^{-\frac{(\alpha-1/2)q}{q+1}}.
\]
Since \(\mathcal M_T\in\mathfrak C_{\mathrm{nl}}\), taking the supremum over
\(\mathfrak C_{\mathrm{nl}}\) proves
\eqref{eq:rate-nonlinear-oracle-lower}.
\end{proof}



\endgroup

\end{document}